\documentclass[a4paper,10pt,fleqn]{article}
\usepackage[latin1]{inputenc}
\usepackage[T1]{fontenc}
\usepackage[left=2cm,right=2cm,bottom=3cm,top=3cm,]{geometry}

\usepackage[all,2cell,v2,arrow, curve, frame, matrix, graph, tips]{xy}

\UseAllTwocells

\usepackage[square,numbers,sort]{natbib}
\usepackage{tikz-cd}
\usepackage{ifthen}
\usepackage{mathbbol}
\usepackage{savesym}
\usepackage{mathtools}
\usepackage{amsmath}
\usepackage{amssymb}
\usepackage{tcolorbox}
\usepackage{enumerate}
\usepackage[thmmarks,standard]{ntheorem}
\usepackage{microtype}
\usepackage{stmaryrd}

\usepackage{verbatim}
\usepackage[bookmarks=true,
            bookmarksnumbered=true, breaklinks=true,
            pdfstartview=FitH, hyperfigures=false,
            plainpages=false, naturalnames=true,
            colorlinks=true,pagebackref=true,
            pdfpagelabels]{hyperref}
\usepackage{cleveref}
\usepackage[toc,page]{appendix}

\theoremstyle{plain}

\newtheorem*{conj}{Conjecture}
\newtheorem*{exercice}{Exercice :}
\newtheorem*{ex}{Example}

\newcommand{\E}{\ensuremath{\mathbb{S}\text{ets}}}
\newcommand{\CS}{\ensuremath{\mathbb{C}\mathbb{S}\text{ets}}}
\newcommand{\Cr}{\ensuremath{\mathbb{C}_\text{r}}}
\newcommand{\Csr}{\ensuremath{\mathbb{C}_\text{sr}}}
\newcommand{\I}{\ensuremath{\infty\text{-}}}
\newcommand{\CC}{\ensuremath{\infty\text{-}\mathbb{C}\mathbb{C}\text{AT}}}
\newcommand{\Cc}{\ensuremath{\infty\text{-}\mathbb{C}\mathbb{C}\text{at}}}

\newcommand{\Igp}{\ensuremath{\infty\text{-}\mathbb{C}\mathbb{G}\text{rp}}}

\newcommand*{\Etg}{\ensuremath{(\infty,0)\text{-}\mathbb{C}\mathbb{E}\text{tG}}}
\newcommand{\Mag}{\ensuremath{\mathbb{M}\text{ag}}}
\newcommand{\Cu}{\ensuremath{\mathbb{C}}}
\newcommand{\Trans}{\ensuremath{\mathbb{T}\text{rans}}}

\newcommand*{\Et}{\ensuremath{\infty\text{-}\mathbb{C}\mathbb{E}\text{tC}}}

\title{Aspects of Cubical Higher Category Theory} \author{Camell
  Kachour}
\begin{document}
\maketitle
\vspace*{3.5cm}
\begin{abstract}
 In this article\footnote{This article has been financially supported by the National French Program \textit{Aspie Friendly}.} we show how to build aspects of articles \cite{kamel,Cam,penon} but with the cubical geometry. Thus we define a monad on the category $\Cu\E$ of cubical sets which algebras are models of cubical weak $\infty$-categories. Also we define a monad on
 $\Cu\E$ which algebras are models of cubical weak $\infty$-groupoids with connections. And finally we define a monad on the 
 category\footnote{$\Cu\E^2$ means the cartesian product 
 $\Cu\E\times\Cu\E$, and $\Cu\E^n$ means the $n$-fold cartesian product of $\Cu\E$ with itself.} $\Cu\E^2$ which algebras are models of cubical weak $\infty$-functors, and a monad on the category 
 $\Cu\E^4$ which algebras are models of cubical weak natural $\infty$-transformations.
\end{abstract}
\begin{minipage}{118mm}{\small
      }\end{minipage}

\hypersetup{%
  linkcolor=blue}%
\tableofcontents
\vspace*{1cm}

\vspace*{1cm}  

\section*{Introduction}

In this article we explain how to build algebraic models of 
\begin{itemize}
\item cubical weak $\infty$-categories with connections (see \ref{stretchings})
\item cubical weak $\infty$-groupoids with connections (see \ref{isostretchings})
\item cubical weak $\infty$-functors (see \ref{cub-funct}) 
\item cubical weak natural $\infty$-transformations (see \ref{cub-trans})
\end{itemize}

In particular cubical weak $(\infty,0)$-categories known as 
\textit{cubical weak $\infty$-groupoids} are very important for us because other models of cubical weak $\infty$-groupoids 
exist but are defined in an non-algebraic way \cite{bro2,bro1,bro-hig2,bro-hig1,bro-spen,dakin,salleh}, i.e defined as kind of 
\textit{cubical Kan complexes}.

As a matter of fact a very important feature of cubical higher category theory is their flexible possibility to have models of 
higher structures build by mimic simplicial methods for presheaves on the classical category $\Delta$, to presheaves
on the \textit{reflexive cubical category} $\mathbb{C}_r$ of cubical sets with connections (see \ref{reflexive_cubical_sets}), 
and in an other hand to have also models of higher structures build by mimic algebraic methods of the globular setting (see 
\cite{kamel,Cam,penon}). 

For this last point it is important to notice that cubical strict $\infty$-categories (see \ref{strict-infty}) are very close in nature to their
globular analogue : first datas of it are given by countable family of sets $(C_n)_{n\in\mathbb{N}}$, equipped both with kind of
sources and targets, and partial operations, and two kinds of \textit{reflexors} (that we call \textit{the classical reflexions} and the
\textit{connections}) on each set $C_n$, subject to axioms. See 
\cite{brown,bro-mosa,steiner}. 

Cubical sets have richer structure than globular sets, analogue to simplicial sets, and 
this richness allows to translate many definitions of simplicial higher category to cubical higher category 
(see \cite{antolini,dakin,porter}). But as we shall see, cubical higher category theory has the algebraic flexibility 
of globular higher category theory, which is a feature we have difficulty to see for simplicial higher category theory.
This important aspect of cubical higher category theory push us to see them as a bridge between simplicial
higher category theory and globular higher category theory. 

We believe that our models of cubical weak $\infty$-groupoids should opens new perspective 
to the Grothendieck conjecture on homotopy types of spaces \ref{isostretchings}, which is stated in the globular setting, and is as
follow :

\begin{conj}[Grothendieck]
The category of (some) models of globular weak $\infty$-groupoids is equipped with a Quillen model structure 
which is Quillen equivalent to the category of spaces equipped with its usual Quillen model structure, i.e those
which weak equivalences are given by the homotopy groups, and which fibrations are Serre fibrations.
\end{conj}

Finally it is important to notice that cubical strict higher structures have already applications
and impacts in homology \cite{ashley,bro-hig3,bro-hig2,bro-hig1} and in algebraic topology
\cite{bro-lod,gilbert,porter}. The use of connections with simplicial method can be found in 
\cite{antolini,dakin,porter}.

\textit{Cubical Stretchings} are the main tools of this article : the first stretchings occurred in the work \cite{penon} of 
Jacques Penon in 1999 under the globular paradigm; they are example of \textit{globular stretchings} and were the central tools to build 
algebraic models of weak $\infty$-categories in Penon's sense. Let us briefly recall how they weakened 
algebraic structures; five ingredients build them : 
\begin{itemize}
\item A projectively sketchable category $\mathcal{G}$ which presheaves\footnote{Usually $\mathcal{G}$ is just 
equipped with the empty cone.} formalize a given 
geometry\footnote{Globular, cubical, multiple, etc.}, i.e where a notion of dimension is involved 
and where the objects of the category of elements of any of its presheaves are called cells.
\item A magmatic higher structure based on $\mathcal{G}$ with chosen partial operations; 
\item A strict higher structure based on $\mathcal{G}$, which in fact is a magmatic higher structure
with chosen equations build with its operations.
\item A morphism of magmas between a magmatic higher structure and a strict higher 
structure based on the geometry of $\mathcal{G}$\footnote{Both in the same kind of higher structure like $\mathcal{G}$.}.
A magma means a structure without equations.
\item A family of partial operations, called \textit{bracketings}, which build coherences (or homotopy) between
two cells of its underlying magma, for those couple of cells with the same dimension and with some other extra conditions.
\end{itemize}
With its bracketing operations a stretching fill its underlying magma with coherence cells, and this filling is controlled\footnote{It is 
similar to cofibrantly generated weak factorization system where the small object argument of Kan-Quillen
shows that a given factorization is build by using a sequence controlled by a small set of maps.}
by its strict structure. Thus it equipped its underlying magma with a weakened version of its underlying strict higher 
structure. The category $\mathbb{S}\text{tr}$ of such stretchings is all the time projectively sketchable 
and one of its sketch $\mathcal{E}$ countains the projective sketch 
$\mathcal{E}'$ of $\mathcal{G}$ :
\begin{tikzcd}
\mathcal{E}'\arrow[r,"i",hook]&\mathcal{E}
\end{tikzcd}, 
thus the induced functor\footnote{This induced functor $i^{\star}$ has no reasons to be monadic.} 
\begin{tikzcd}
\mathbb{M}\text{od}(\mathcal{E})\arrow[r,"i^{\star}"]&\mathbb{M}\text{od}(\mathcal{E}')
\end{tikzcd} 
is a right adjoint \cite{foltz}, and we get a monad $\mathbb{W}$ on $\mathbb{M}\text{od}(\mathcal{E}')$. 
$\mathbb{W}$-algebras are models of the weakened version of its corresponding strict higher structures.
The free algebraic structures are the free stretchings on objects of $\mathcal{G}$, and they are our basic weak 
higher structure which weakened the strict higher structures involved. In \cite{penon} it is shown that free
categorical stretchings on globular sets can be described by using basic method of logic, i.e constructing 
terms with an adapted language to the globular geometry which is defined with operations such as the $\circ^{n}_{p}$ and the 
bracketings $[-;-]$; it has no variables and has only constants built out with globular sets. Such method
can be apply to build free stretchings of this article, but it is quite technical and goes beyond the
scope of this article. The category of strict higher structures of such stretchings 
is the category of algebras for a monad $\mathbb{S}$ on $\mathcal{G}$. The monad $\mathbb{W}$ is 
a kind of (or weakly equivalent to) cofibrant replacement of the monad $\mathbb{S}$.
A true advantage of weakened with stretchings is it doesn't need the cartesianity
of the monad $\mathbb{S}$ of its underlying strict higher structures. But it seems that the corresponding
monad $\mathbb{W}$ could be an operad as it is shown in \cite{bat-pen}. Some good other features of the monad 
$\mathbb{W}$ is it preserves $\alpha$-filtered colimits
where $\alpha$ is a cardinal which bound the size of each cones of both sketches $\mathcal{E}$ and $\mathcal{E}'$.
This shows that the category $\mathbb{W}\text{-}\mathbb{A}\text{lg}$ of $\mathbb{W}$-algebras is locally 
presentable. 

This article is mainly concerned with \textit{cubical stretchings} and \textit{cubical higher structures} they provide. 
It is devoted to several work which main steps are as follow :
\begin{itemize}

\item We define our own terminology in order to be as close as possible to the notation of the globular environment, and in it we define the monad of cubical strict $\infty$-categories on 
the category of cubical sets\footnote{Cubical sets in our terminology are the \textit{precubical sets} of Richard Steiner \cite{steiner}.} 

\item We define the category of cubical categorical stretchings which is the cubical anologue of the category of
globular categorical stretchings of \cite{penon}. The key ingredient is a cubical analogue of the globular contractions
build in \cite{penon}. Then we give a monad $\mathbb{W}$ on the category of cubical sets which 
algebras are our models of cubical weak $\infty$-categories. This monad is the cubical analogue of the 
monad $\mathbb{P}^{0}_C$ in \cite{penon}, which $\mathbb{P}^{0}_C$-algebras are the globular weak $\infty$-categories of Penon. 

\item We start to define cubical $(\infty,0)$-sets, which main tools are the \textit{cubical reversors}. These are
the cubical analogue of the globular $(\infty,0)$-sets\footnote{Also called $(\infty,0)$-graphs in \cite{Cam}}
which has been defined in \cite{Cam}. More precisely they are build by using cubical analogue of minimal 
$(\infty,0)$-structures in the sense of \cite{Cam}. Then we define the category of cubical reflexive $(\infty,0)$-magmas 
and then the category of cubical $(\infty,0)$-groupoidal stretchings, which is the cubical analogue of the globular 
$(\infty,0)$-categorical stretchings which has been defined in \cite{Cam}. The add of this new sketches inside 
cubical stretchings allows us to build 
a monad on the category of cubical sets which algebras are our models of cubical weak $\infty$-groupoids. 
This monad is the cubical analogue of the monad on globular sets build in \cite{Cam} and
which algebras are globular models of weak $\infty$-groupoids.

\item In the sections \ref{cub-funct} and \ref{cub-trans} we extend globular weak 
$\infty$-functors and globular weak $\infty$-natural transformations
to the cubical setting. In particular we shall see that the monad of 
cubical weak $\infty$-functors act on the category 
$\Cu\E\times\Cu\E$ and the monad of cubical weak natural
$\infty$-transformations act on the category 
$\Cu\E\times\Cu\E\times\Cu\E\times\Cu\E$. In these 
last sections some interesting internal $2$-cubes appears in
$\CC$ which actually are all classical cubical strict $2$-categories. We finish our article
by sketching an expected cocubical object of monads which leads to an operadic approach
of the cubical weak $\infty$-category of cubical weak $\infty$-categories.
\end{itemize} 

In \cite{penon,kamel,Cam} some computations have been described for globular higher structures born 
with globular stretchings. For example in \cite{penon} it is proved that in dimension $2$, globular weak $\infty$-categories
of Penon are bicategories; in \cite{kamel} it is proved that in dimension $2$ globular weak $\infty$-functors are 
pseudo-$2$-functors. In \ref{computations} we gave a precise definition of the dimension for algebras of 
our models of cubical weak $\infty$-categories. Computations in dimension $2$ lead to long computations and go beyond 
the scope of this article, but the reader interested in the dimension $2$ can verify that our models of dimension $2$ and 
without connections are cubical bicategories \cite{verity-thesis} in the sense of Verity \ref{computations} (which is our terminology for weak double categories); however in \ref{example-coherence}
we exhibited an example of cubical coherence cell\footnote{In \cite{camark-cub-1} we exhibited the same kind of $2$-cell for the cubical 
operad of cubical weak $\infty$-categories. Globular stretchings and globular operads are quite close in nature (for example they 
share the same notion of contractions) and for computations \cite{bat-pen}.} in dimension $2$. 

In this article the reader has to be aware to not confuse $(\infty,m)$-structures and $(n,\infty)$-structures : 
$(\infty,m)$-structures as described in \ref{isostretchings} are kind of higher categories with invertible cells; $(n,\infty)$-structures 
as described in \ref{cub-funct} \ref{cub-trans} lead to $n$-cells in the cubical weak $\infty$-category of cubical weak $\infty$-categories\footnote{We discuss some important progress for its existence in \cite{camark-cub-2}.}; for example $(0,\infty)$-structures 
in \ref{cub-funct} means all higher structures surrounded cubical weak $\infty$-functors, and $(1,\infty)$-structures in 
\ref{cub-trans} means all higher structures which surround cubical weak $\infty$-natural transformations.
   
Two sketches are \textit{Morita equivalent} 
if their category of models are equivalent. Also when we write "sketch of a category $\mathcal{C}$", we mean that this category 
$\mathcal{C}$ is projectively sketchable, and that a sketch of it has been fixed up to Morita equivalence. 

{\bf Acknowledgement.} 
I thank mathematicians of the AGA\footnote{Arithmétique et Géométrie Algébrique, LMO, Paris sud.}, and the good ambience provided by the team in the lab, especially 
I want to mention Elisabeth Gassiat, Olivier Schiffmann, Christophe Breuil, Valentin Hernandez, Philip Boalch, 
François Charles and Benjamin Hennion. I want also to mention mathematicians who help me a lot :
 Maxim Kontsevich, Pierre Cartier, Ronnie Brown, Vasily Pestun, Jordi Lopez-Abad, Bertrand Monthubert, Ghislain Rémy, Michael Batanin, Ross Street
 and Mark Weber. I thank mathematicians of the IRIF where I started to write this paper\footnote
 {Three months have been financially supported (December 2016 until February 2017) under an European Research Council Project called \textit{Duall} : \url{https://www.irif.fr/~mgehrke/DuaLL.htm.}}. Especially
 I want to mention Paul-André Méllies, Mai Gehrke, Pierre Louis Curien and Thomas Ehrhard.
 
 Finally I want to mention other wonderful persons : 
 Stèf Bonnot-Briey, Frédérique Vidal, Sophie Cluzel, Jean-Philippe Bourgoin, Alexandra Van Cauteren, 
 Françoise Dorocq and 
 Jean-Pierre Ledru. 
 
 I dedicate this work to Ronnie Brown.

\section{Cubical sets}

Cubical sets are presheaf on a specific small category $\mathbb{C}$ \ref{cub-sets}. This small
category contains the combinatorics underlying the geometric idea of cubes and higher
cubes. The following internal cocubical complex in $\mathbb{T}\text{op}$ where $I=[0,1]$ is the usual interval in 
$\mathbb{R}$ and where $I^{n}$ is the product of $I$ with itself $n$-times \cite{brown-book} :

 $$\begin{tikzcd}
 I^{0}\arrow[rr, yshift=1.5ex,"s^{0}_{1}"]
 \arrow[rr, yshift=-1.5ex,"t^{0}_{1}"{below}]
 &&I^{1} 
\arrow[rr, yshift=1.5ex,"s^{2}_{1,1}"]
\arrow[rr, yshift=-1.5ex,"t^{2}_{1,1}"{below}]
\arrow[rr, yshift=4.5ex,"s^{2}_{1,2}"] 
\arrow[rr, yshift=-4.5ex,"t^{2}_{1,2}"{below}]   
        && I^{2} 
         \arrow[rr, yshift=1.5ex,"s^{3}_{2,1}"]
        \arrow[rr, yshift=-1.5ex,"t^{3}_{2,1}"{below}]
        \arrow[rr, yshift=4.5ex,"s^{3}_{2,2}"] 
        \arrow[rr, yshift=-4.5ex,"t^{3}_{2,2}"{below}]
        \arrow[rr, yshift=7.5ex,"s^{3}_{2,3}"] 
        \arrow[rr, yshift=-7.5ex,"t^{3}_{2,3}"{below}]&& 
         I^{3} 
        \arrow[rr, yshift=1.5ex,"s^{4}_{3,4}"]
        \arrow[rr, yshift=-1.5ex,"t^{4}_{3,4}"{below}]
        \arrow[rr, yshift=4.5ex,"s^{4}_{3,3}"] 
        \arrow[rr, yshift=-4.5ex,"t^{4}_{3,3}"{below}]
        \arrow[rr, yshift=7.5ex,"s^{4}_{3,2}"] 
        \arrow[rr, yshift=-7.5ex,"t^{4}_{3,2}"{below}]        
        \arrow[rr, yshift=10.5ex,"s^{4}_{3,1}"] 
        \arrow[rr, yshift=-10.5ex,"t^{4}_{3,1}"{below}] && 
        I^{4}\cdots I^{n-1}  
\arrow[rr, yshift=1.5ex,"s^{n}_{n-1,n-1}"]
        \arrow[rr, yshift=-1.5ex,"t^{n}_{n-1,n-1}"{below}]
        \arrow[rr, yshift=4.5ex,dotted] 
        \arrow[rr, yshift=-4.5ex,dotted]
        \arrow[rr, yshift=7.5ex,"s^{n}_{n-1,i}"] 
        \arrow[rr, yshift=-7.5ex,"t^{n}_{n-1,i}"{below}]        
        \arrow[rr, yshift=10.5ex,dotted] 
        \arrow[rr, yshift=-10.5ex,dotted]    
  \arrow[rr, yshift=13.5ex,"s^{n}_{n-1,1}"] 
        \arrow[rr, yshift=-13.5ex,"t^{n}_{n-1,1}"{below}]&& I^{n}\cdots    
  \end{tikzcd}$$
is an archetypal example of shape modelised by $\mathbb{C}^{op}$. 
Here 
\textit{sources} $\xymatrix{\underline{n-1}\ar[rr]^{s^{n}_{n-1,j}}&&\underline{n}}$ for each $j\in\{1,..,n\}$ and \textit{targets}  
$\xymatrix{\underline{n-1}\ar[rr]^{t^{n}_{n-1,j}}&&\underline{n}}$ for each $j\in\{1,..,n\}$ such that for $1\leq i<j\leq n$, follow
the cocubical relations

\begin{enumerate}[(i)]
\item $ s^{n}_{n-1,j}\circ s^{n-1}_{n-2,i}=s^{n}_{n-1,i}\circ s^{n-1}_{n-2,j-1},$
\item $t^{n}_{n-1,j}\circ s^{n-1}_{n-2,i}=s^{n}_{n-1,i}\circ t^{n-1}_{n-2,j-1},$
\item $s^{n}_{n-1,j}\circ t^{n-1}_{n-2,i}=t^{n}_{n-1,i}\circ s^{n-1}_{n-2,j-1},$
\item $t^{n}_{n-1,j}\circ t^{n-1}_{n-2,i}=t^{n}_{n-1,i}\circ t^{n-1}_{n-2,j-1}$,
\end{enumerate}

which are the dual relations described below in \ref{cub-sets}. This cocubical
complex is used in \cite{camark-cub-2} to build the functor of fundamental
cubical weak $\infty$-groupoids for spaces\footnote{It is a $\mathbb{B}^{0}_C$-coalgebra, 
where $\mathbb{B}^{0}_C$ is the cubical operad of cubical weak $\infty$-categories \cite{camark-cub-1}.}.
More references on cubical sets can be found in \cite{antolini,grandis-maur}.
\subsection{The cubical category}
\label{cub-sets}

Consider the small category $\mathbb{C}$ with integers $\underline{n}\in\mathbb{N}$ as objects. 
Generators for $\mathbb{C}$ are, for all $\underline{n}\in\mathbb{N}$
given by \textit{sources} $\xymatrix{\underline{n}\ar[rr]^{s^{n}_{n-1,j}}&&\underline{n-1}}$ for each $j\in\{1,..,n\}$, and 
by \textit{targets} $\xymatrix{\underline{n}\ar[rr]^{t^{n}_{n-1,j}}&&\underline{n-1}}$ for each $j\in\{1,..,n\}$ such that for 
$1\leq i<j\leq n$ we have the following cubical relations 

\begin{enumerate}[(i)]

\item $ s^{n-1}_{n-2,i}\circ s^{n}_{n-1,j}=s^{n-1}_{n-2,j-1}\circ s^{n}_{n-1,i},$
\item $s^{n-1}_{n-2,i}\circ t^{n}_{n-1,j}=t^{n-1}_{n-2,j-1}\circ s^{n}_{n-1,i},$
\item $t^{n-1}_{n-2,i}\circ s^{n}_{n-1,j}=s^{n-1}_{n-2,j-1}\circ t^{n}_{n-1,i},$
\item $t^{n-1}_{n-2,i}\circ t^{n}_{n-1,j}=t^{n-1}_{n-2,j-1}\circ t^{n}_{n-1,i}$
\end{enumerate}

These generators plus these relations give the small category $\mathbb{C}$ called the \textit{cubical category} that we 
may represent schematically with the low dimensional diagram :

$$\begin{tikzcd}
  \cdots C_{4}  \arrow[rr, yshift=1.5ex,"s^{4}_{3,1}"]
  \arrow[rr, yshift=-1.5ex,"t^{4}_{3,1}"{below}]
  \arrow[rr, yshift=4.5ex,"s^{4}_{3,2}"] \arrow[rr, yshift=-4.5ex,"t^{4}_{3,2}"{below}]
  \arrow[rr, yshift=7.5ex,"s^{4}_{3,3}"] 
  \arrow[rr, yshift=-7.5ex,"t^{4}_{3,3}"{below}]
  \arrow[rr, yshift=10.5ex,"s^{4}_{3,4}"] 
  \arrow[rr, yshift=-10.5ex,"t^{4}_{3,4}"{below}]      && 
  C_{3} \arrow[rr, yshift=1.5ex,"s^{3}_{2,1}"]
  \arrow[rr, yshift=-1.5ex,"t^{3}_{2,1}"{below}]
  \arrow[rr, yshift=4.5ex,"s^{3}_{2,2}"] 
  \arrow[rr, yshift=-4.5ex,"t^{3}_{2,2}"{below}]
  \arrow[rr, yshift=7.5ex,"s^{3}_{2,3}"] 
  \arrow[rr, yshift=-7.5ex,"t^{3}_{2,3}"{below}]
      && C_{2} 
      \arrow[rr, yshift=1.5ex,"s^{2}_{1,1}"]
      \arrow[rr, yshift=-1.5ex,"t^{2}_{1,1}"{below}]
      \arrow[rr, yshift=4.5ex,"s^{2}_{1,2}"] 
      \arrow[rr, yshift=-4.5ex,"t^{2}_{1,2}"{below}]
      && C_{1} \arrow[rr, yshift=1.5ex,"s^{1}_{0}"]\arrow[rr, yshift=-1.5ex,"t^{1}_{0}"{below}]   
        && C_{0}
\end{tikzcd}$$
and this category $\Cu$ gives also the sketch $\mathcal{E}_\text{S}$ of cubical sets used
especially in \ref{monad-strict}, \ref{stretchings} and \ref{isostretchings} to produce the monads 
$\mathbb{S}=(S,\lambda,\mu)$, $\mathbb{W}=(W,\eta,\nu)$ and
$\mathbb{W}^m=(W^m,\eta^m,\nu^m)$ on $\Cu\E$, which algebras are respectively cubical strict $\infty$-categories,
cubical weak $\infty$-categories and cubical weak $(\infty,m)$-categories.

\begin{definition}
\label{CS}
The category of cubical sets $\CS$ is the category of presheaves $[\mathbb{C};\E]$. The terminal cubical set is denoted $1$.
\end{definition}
Occasionally a cubical set shall be denoted with the notation 
$$\mathcal{C}=(C_n,s^{n}_{n-1,j},t^{n}_{n-1,j})_{1\leq j\leq n,\hspace{.1cm}n\in\mathbb{N}}$$
in case
we want to point out its underlying structures.

\subsection{Reflexive cubical sets}
\label{reflexive_cubical_sets}
Reflexivity for cubical sets are of two sorts : one is "classical" in the sense that they are very similar to their globular analogue;
thus we shall use the notation $(1^{n}_{n+1,j})_{n\in\mathbb{N}, j\in\{1,..,n\}}$ to denote these maps 
$\xymatrix{C(n)\ar[rr]^{1^{n}_{n+1,j}}&&C(n+1)}$ which formally behave like globular reflexivity (\cite{Cam}); the others are called \textit{connections} and are given by maps $\xymatrix{C(n)\ar[rr]^{\Gamma}&&C(n+1)}$ where the notation
using the greek letter \textit{"Gamma"} seems to be the usual notation. However we do prefer to use instead the notation 
$\xymatrix{C(n)\ar[r]^(.45){1^{n,\gamma}_{n+1,j}}&C(n+1)}$ ($\gamma\in\{+,-\}$) in order to point out
the reflexive nature of connections.

Consider the cubical category $\mathbb{C}$. For all $n\in\mathbb{N}$ we add in it generators 
$\xymatrix{\underline{n-1}\ar[r]^(.55){1^{n-1}_{n,j}}&\underline{n}}$ for each $j\in\{1,..,n\}$ subject
to the relations :
\begin{enumerate}[(i)]
\item $1^{n}_{n+1,i}\circ 1^{n-1}_{n,j}= 1^{n}_{n+1,j+1}\circ 1^{n-1}_{n,i} \hspace{.2cm}\text{if } \hspace{.2cm}1 \leq i\leq j\leq n$;
\item $ s^{n}_{n-1,i}\circ 1^{n-1}_{n,j}= 1^{n-2}_{n-1,j-1}\circ s^{n-1}_{n-2,i}\hspace{.2cm} \text{if } \hspace{.2cm}1\leq i<j\leq n$;
\item $ s^{n}_{n-1,i}\circ 1^{n-1}_{n,j}= 1^{n-2}_{n-1,j}\circ s^{n-1}_{n-2,i-1}\hspace{.2cm} \text{if } \hspace{.2cm} 1\leq j<i\leq n$; 
\item $s^{n}_{n-1,i}\circ 1^{n-1}_{n,j} = id(\underline{n-1}) \hspace{.2cm}\text{if }\hspace{.2cm}  i=j$.
\end{enumerate}

and

\begin{enumerate}[(i)]
\item $1^{n}_{n+1,i}\circ 1^{n-1}_{n,j}=1^{n}_{n+1,j+1}\circ 1^{n-1}_{n,i} \hspace{.2cm}\text{if }\hspace{.2cm}1 \leq i\leq j\leq n$;
\item $t^{n}_{n-1,i}\circ 1^{n-1}_{n,j}$=$1^{n-2}_{n-1,j-1}\circ t^{n-1}_{n-2,i}\hspace{.2cm} \text{if } \hspace{.2cm}1\leq i<j\leq n$;
\item $ t^{n}_{n-1,i}\circ 1^{n-1}_{n,j}=1^{n-2}_{n-1,j}\circ t^{n-1}_{n-2,i-1}\hspace{.2cm} \text{if } \hspace{.2cm} 1\leq j<i\leq n$; 
\item $t^{n}_{n-1,i}\circ 1^{n-1}_{n,j} =id(\underline{n-1}) \hspace{.2cm}\text{if }\hspace{.2cm}  i=j$.
\end{enumerate}

These generators and relations give the small category $\Csr$ called the \textit{semireflexive cubical category} where 
a quick look at its underlying semireflexive structure is given by the following diagram :

$$\begin{tikzcd}
  C_{0} \arrow[rr,"1^{0}_{1}"]
    && C_{1}  \arrow[rr, yshift=1.5ex,"1^{1}_{2,1}"]
        \arrow[rr, yshift=-1.5ex,"1^{1}_{2,2}"]
&& C_{2}  \arrow[rr,"1^{2}_{3,1}"]
\arrow[rr, yshift=3.5ex,"1^{2}_{3,2}"] 
\arrow[rr, yshift=-3.5ex,"1^{2}_{3,3}"]
        && 
        C_{3}\arrow[rr, yshift=1.5ex,"1^{3}_{4,1}"] 
        \arrow[rr, yshift=-1.5ex,"1^{3}_{4,2}"]
        \arrow[rr, yshift=5.5ex,"1^{3}_{4,3}"] 
        \arrow[rr, yshift=-5.5ex,"1^{3}_{4,4}"]
        &&C_{4}\cdots
\end{tikzcd}$$

\begin{definition}
The category of semireflexive cubical sets $\Csr\E$ is the category of presheaves $[\Csr;\E]$. 
The terminal semireflexive cubical set is denoted $1_\text{sr}$
\end{definition}

Consider the semireflexive cubical category $\Csr$. For all integers $n\geq 1$ we add in it generators 
$\xymatrix{\underline{n-1}\ar[r]^(.6){1^{n-1,\gamma}_{n,j}}&\underline{n}}$ for each $j\in\{1,..,n-1\}$ subject
to the relations :
\begin{enumerate}[(i)]

\item for $1\leq j<i\leq n$, $1^{n,\gamma}_{n+1,i}\circ 1^{n-1,\gamma}_{n,j}=1^{n,\gamma}_{n+1,j+1}\circ 1^{n-1,\gamma}_{n,i}$;

\item for $1\leq i\leq n-1$, $1^{n,\gamma}_{n+1,i}\circ 1^{n-1,\gamma}_{n,i}=1^{n,\gamma}_{n+1,i+1}\circ 1^{n-1,\gamma}_{n,i}$;  

\item for $1\leq i, j\leq n$, 
$\left\{
\begin{array}{rl}
  1^{n,\gamma}_{n+1,i}\circ 1^{n-1}_{n,j} &= 1^{n}_{n+1,j+1}\circ 1^{n-1,\gamma}_{n,i} \quad\text{if}\quad1\leq i<j\leq n\\   
 &= 1^{n}_{n+1,j} \circ 1^{n-1,\gamma}_{n,i-1}\hspace{.2cm}\text{if} \hspace{.2cm}1\leq j<i\leq n
 \end{array}
\right.$;

\item for $1\leq j\leq n$, $1^{n,\gamma}_{n+1,j}\circ 1^{n-1}_{n,j}=1^{n}_{n+1,j}\circ 1^{n-1}_{n,j}$; 

\item for $1\leq i, j\leq n$, 

$\left\{
\begin{array}{rl}
  s^{n}_{n-1,i}\circ 1^{n-1,\gamma}_{n,j}&= 1^{n-2,\gamma}_{n-1,j-1}\circ s^{n-1}_{n-2,i}\hspace{.2cm}\text{if}\hspace{.2cm} 1\leq i<j\leq n-1  \\
 &= 1^{n-2,\gamma}_{n-1,j}\circ s^{n-1}_{n-2,i-1} \hspace{.2cm}\text{if}\hspace{.2cm} 2\leq j+1<i \leq n
  \end{array}
\right.$; 

and

$\left\{
\begin{array}{rl}
  t^{n}_{n-1,i}\circ 1^{n-1,\gamma}_{n,j}&= 1^{n-2,\gamma}_{n-1,j-1}\circ t^{n-1}_{n-2,i}\hspace{.2cm}\text{if}\hspace{.2cm} 1\leq i<j\leq n-1  \\
 &= 1^{n-2,\gamma}_{n-1,j}\circ t^{n-1}_{n-2,i-1} \hspace{.2cm}\text{if}\hspace{.2cm} 2\leq j+1<i \leq n
  \end{array}
\right.$;\\

\item for $1\leq j\leq n-1$, $s^{n}_{n-1,j}\circ 1^{n-1,-}_{n,j}=s^{n}_{n-1,j+1}\circ 1^{n-1,-}_{n,j}=1_{\underline{n-1}}$
and $t^{n}_{n-1,j}\circ 1^{n-1,+}_{n,j}=t^{n}_{n-1,j+1}\circ 1^{n-1,+}_{n,j}=1_{\underline{n-1}}$; 

\item for $1\leq j\leq n-1$, $s^{n}_{n-1,j}\circ 1^{n-1,+}_{n,j}= s^{n}_{n-1,j+1}\circ 1^{n-1,+}_{n,j}= 1^{n-2}_{n-1,j}\circ s^{n-1}_{n-2,j}$;

\item for $1\leq j\leq n-1$, $t^{n}_{n-1,j}\circ 1^{n-1,-}_{n,j}= t^{n}_{n-1,j+1}\circ 1^{n-1,-}_{n,j}= 1^{n-2}_{n-1,j}\circ t^{n-1}_{n-2,j}$.

\end{enumerate}

These generators and relations give the small category $\mathbb{C}_{r}$ called the 
\textit{reflexive cubical category} and in it, connections have the following shape :

$$\begin{tikzcd}
 C_{1} \arrow[rr, yshift=1.5ex,"1^{1,-}_{2,1}"]
 \arrow[rr, yshift=-1.5ex,"1^{1,+}_{2,1}"{below}]
 && C_{2} 
\arrow[rr, yshift=1.5ex,"1^{2,-}_{3,1}"]
\arrow[rr, yshift=-1.5ex,"1^{2,+}_{3,1}"{below}]
\arrow[rr, yshift=4.5ex,"1^{2,-}_{3,2}"] 
\arrow[rr, yshift=-4.5ex,"1^{2,+}_{3,2}"{below}]   
        && C_{3}
        \arrow[rr, yshift=1.5ex,"1^{3,-}_{4,1}"]
        \arrow[rr, yshift=-1.5ex,"1^{3,+}_{4,1}"{below}]
        \arrow[rr, yshift=4.5ex,"1^{3,-}_{4,2}"] 
        \arrow[rr, yshift=-4.5ex,"1^{3,+}_{4,2}"{below}]
        \arrow[rr, yshift=7.5ex,"1^{3,-}_{4,3}"] 
        \arrow[rr, yshift=-7.5ex,"1^{3,+}_{4,3}"{below}]&& 
        C_{3} 
        \arrow[rr, yshift=1.5ex,"1^{4,-}_{5,1}"]
        \arrow[rr, yshift=-1.5ex,"1^{4,+}_{5,1}"{below}]
        \arrow[rr, yshift=4.5ex,"1^{4,-}_{5,2}"] 
        \arrow[rr, yshift=-4.5ex,"1^{4,+}_{5,2}"{below}]
        \arrow[rr, yshift=7.5ex,"1^{4,-}_{5,3}"] 
        \arrow[rr, yshift=-7.5ex,"1^{4,+}_{5,3}"{below}]        
        \arrow[rr, yshift=10.5ex,"1^{4,-}_{5,4}"] 
        \arrow[rr, yshift=-10.5ex,"1^{4,+}_{5,4}"{below}] && 
        C_{4}\cdots          
\end{tikzcd}$$

\begin{definition}
The category of reflexive cubical sets $\Cr\E$ is the category of presheaves $[\Cr;\E]$. 
The terminal reflexive cubical set is denoted $1_\text{r}$
\end{definition}

It is important to note that this small category $\Cr$ is a \textit{strict test category} \cite{malts}, that
is the category $\Cr\E$ of reflexive cubical sets can be equipped with a genious Quillen model structure
which is Quillen equivalent to the category of spaces equipped with its classical
Quillen model structure \cite{hirschhorn}.

\section{The category of strict cubical $\infty$-categories}
\label{strict-infty}
Cubical strict $\infty$-categories have been studied in \cite{brown,bro-mosa,steiner}. In \cite{brown} the authors proved that 
the category of cubical strict $\infty$-categories with cubical strict $\infty$-functors as 
morphisms is equivalent to the category of globular strict $\infty$-categories with globular
strict $\infty$-functors as morphisms. 

These higher structures are very
important for us because they are used in \ref{stretchings} to control the cubical coherences cells that we need to add in
reflexive cubical $\infty$-magmas in order to be our basic cubical weak $\infty$-categories : in fact they allow to 
build free cubical weak $\infty$-categories for each cubical sets. The primary structure behind cubical strict $\infty$-categories
are cubical $\infty$-magmas that we are going to define quickly. 

Consider a cubical reflexive set 
$$(C,(1^{n}_{n+1,j})_{n\in\mathbb{N},j\in\llbracket 1,n+1 \rrbracket}, (1^{n,\gamma}_{n+1,j})_{n\geq 1,j\in\llbracket 1,n \rrbracket})$$
equipped with partial operations 
$(\circ^{n}_{j})_{n\geq 1, j\in\llbracket 1,n \rrbracket}$ 
where if $a, b \in C(n)$ then $a\circ^{n}_{j}b$ is defined for $j\in\{1,...,n\}$ if 
$s^{n}_{j}(b)=t^{n}_{j}(a)$. We also require these operations to follow the following axioms of positions :

\begin{enumerate}[(i)]

\item For $1\leq j\leq n$ we have : $s^{n}_{n-1,j}(a\circ^{n}_{j}b)=s^{n}_{n-1,j}(a)$ and $t^{n}_{n-1,j}(a\circ^{n}_{j}b)=t^{n}_{n-1,j}(a)$,\\

\item $s^{n}_{n-1,i}(a\circ^{n}_{j}b)=\left\{
\begin{array}{rl}
  s^{n}_{n-1,i}(a)\circ^{n-1}_{j-1}s^{n}_{n-1,i}(b) \hspace{.1cm}\text{if} \hspace{.1cm}1\leq i<j\leq n \\
  s^{n}_{n-1,i}(a)\circ^{n-1}_{j}s^{n}_{n-1,i}(b)\hspace{.1cm}\text{if}\hspace{.1cm}1\leq j<i\leq n
 \end{array}
\right.$ 

\item $t^{n}_{n-1,i}(a\circ^{n}_{j}b)=\left\{
\begin{array}{rl}
  t^{n}_{n-1,i}(a)\circ^{n-1}_{j-1}t^{n}_{n-1,i}(b)\hspace{.1cm}\text{if}\hspace{.1cm}1\leq i<j\leq n \\
  t^{n}_{n-1,i}(a)\circ^{n-1}_{j}t^{n}_{n-1,i}(b)\hspace{.1cm}\text{if}\hspace{.1cm}1\leq j<i\leq n
 \end{array}
\right.$ 
\end{enumerate}

\begin{definition}
Cubical $\infty$-magmas are cubical sets equipped with partial operations like above. A morphism between two cubical $\infty$-magmas is a morphism of their underlying
cubical sets which respects partial operations 
$(\circ^{n}_{j})_{n\geq 1, j\in\llbracket 1,n \rrbracket}$.The category of cubical $\infty$-magmas is noted $\I\Cu\Mag$
\end{definition}

The following sketch $\mathcal{E}_M$ of \textit{axioms of positions} as
above shall be used in \ref{monad-strict} to justify the existence 
of the monad on $\Cu\E$ of cubical strict $\infty$-categories. It is
important to notice that the sketch just below has only
\textit{one generation} which means that diagrams and cones
involved in it are not build with previous data of other diagrams
and cones. The terminology used in \cite{lair-triplabilite} is \textit{sketch with one
floor}\footnote{In French we say "\textit{esquisse à un étage}".}. See also \cite{laircoppey:esquisses}.

\begin{itemize}
\item For $1\leq i<j\leq n$ we consider the following two cones :

\begin{tikzcd}
M_n\underset{M_{n-1,j}}\times M_n\arrow[d,"\pi^{n}_{0,j}"{left}]
\arrow[rr,"\pi^n_{1,j}"]&&M_n\arrow[d,"s^{n}_{n-1,j}"]\\
M_n\arrow[rr,"t^{n}_{n-1,j}"{below}]&&M_{n-1}\end{tikzcd}
\qquad\begin{tikzcd}
M_{n-1}\underset{M_{n-2,j-1}}\times M_{n-1}\arrow[d,"\pi^{n-1}_{0,j-1}"{left}]
\arrow[rr,"\pi^{n-1}_{1,j-1}"]&&M_{n-1}\arrow[d,"s^{n-1}_{n-2,j-1}"]\\
M_{n-1}\arrow[rr,"t^{n-1}_{n-2,j-1}"{below}]&&M_{n-2}
\end{tikzcd}

 and the following commutative diagram 
(definition of $s^{n}_{n-1,i}\underset{j,j-1}\times s^{n}_{n-1,i}$)

\begin{tikzcd}
M_n\underset{M_{n-1,j}}\times M_n
\arrow[rrr,"\pi^n_{1,j}"]\arrow[dd,"\pi^n_{0,j}"{left}]
\arrow[rrd,dotted,"s^{n}_{n-1,i}\underset{j,j-1}\times s^{n}_{n-1,i}"]&&&M_n
\arrow[rrd,"s^{n}_{n-1,i}"]\\
&&M_{n-1}\underset{M_{n-2,j-1}}\times M_{n-1}\arrow[dd,"\pi^{n-1}_{0,j-1}"{left}]\arrow[rrr,"\pi^{n-1}_{1,j-1}"]&&&M_{n-1}\arrow[dd,"s^{n-1}_{n-2,j-1}"]\\
M_n\arrow[rrd,"s^{n}_{n-1,i}"{below}]\\
&&M_{n-1}\arrow[rrr,"t^{n-1}_{n-2,j-1}"{below}]&&&M_{n-2}
\end{tikzcd}

 which gives the following commutative diagram 

\begin{tikzcd}
M_n\underset{M_{n-1,j}}\times M_n\arrow[d,"\star^{n}_{j}"{left}]
\arrow[rr,"s^{n}_{n-1,i}\underset{j,j-1}\times s^{n}_{n-1,i}"]&&
M_{n-1}\underset{M_{n-2,j-1}}\times M_{n-1}\arrow[d,"\star^{n-1}_{j-1}"]\\
M_n\arrow[rr,"s^{n}_{n-1,i}"{below}]&&M_{n-1}
\end{tikzcd}

\item For $1\leq j<i\leq n$ we consider the following two cones :

\begin{tikzcd}
M_n\underset{M_{n-1,j}}\times M_n\arrow[d,"\pi^{n}_{0,j}"{left}]
\arrow[rr,"\pi^n_{1,j}"]&&M_n\arrow[d,"s^{n}_{n-1,j}"]\\
M_n\arrow[rr,"t^{n}_{n-1,j}"{below}]&&M_{n-1}\end{tikzcd}
\qquad\begin{tikzcd}
M_{n-1}\underset{M_{n-2,j}}\times M_{n-1}\arrow[d,"\pi^{n-1}_{0,j}"{left}]
\arrow[rr,"\pi^{n-1}_{1,j}"]&&M_{n-1}\arrow[d,"s^{n-1}_{n-2,j}"]\\
M_{n-1}\arrow[rr,"t^{n-1}_{n-2,j}"{below}]&&M_{n-2}\end{tikzcd}

and the following commutative diagram 
(definition of $s^{n}_{n-1,i}\underset{j,j}\times s^{n}_{n-1,i}$)

\begin{tikzcd}
M_n\underset{M_{n-1,j}}\times M_n
\arrow[rrr,"\pi^{n}_{1,j}"]\arrow[dd,"\pi^{n}_{0,j}"{left}]
\arrow[rrd,dotted,"s^{n}_{n-1,i}\underset{j,j}\times s^{n}_{n-1,i}"]&&&M_n
\arrow[rrd,"s^{n}_{n-1,i}"]\\
&&M_{n-1}\underset{M_{n-2,j}}\times M_{n-1}\arrow[dd,"\pi^{n-1}_{0,j}"{left}]\arrow[rrr,"\pi^{n-1}_{1,j}"]&&&M_{n-1}\arrow[dd,"s^{n-1}_{n-2,j}"]\\
M_n\arrow[rrd,"s^{n}_{n-1,i}"{below}]\\
&&M_{n-1}\arrow[rrr,"t^{n-1}_{n-2,j}"{below}]&&&M_{n-2}
\end{tikzcd}

\end{itemize}

The previous datas gives the following commutative diagram of axioms

\begin{tikzcd}
M_n\underset{M_{n-1}}\times M_n\arrow[d,"\star^{n}_{j}"{left}]
\arrow[rr,"s^{n}_{n-1,i}\underset{j,j}\times s^{n}_{n-1,i}"]&&
M_{n-1}\underset{M_{n-2}}\times M_{n-1}\arrow[d,"\star^{n-1}_{j}"]\\
M_n\arrow[rr,"s^{n}_{n-1,i}"{below}]&&M_{n-1}
\end{tikzcd}

and for $1\leq j\leq n$ we have the following commutative diagram of axioms

\begin{tikzcd}
M_n\underset{M_{n-1}}\times M_n\arrow[d,"\star^{n}_{j}"{left}]
\arrow[rr,"\pi_1"]&&
M_n\arrow[d,"s^{n}_{n-1,j}"]\\
M_n\arrow[rr,"s^{n}_{n-1,j}"{below}]&&M_{n-1}\end{tikzcd}

which actually complete the description of $\mathcal{E}_M$

\begin{definition} 
Cubical reflexive $\infty$-magmas are cubical reflexive set equipped a structure of $\infty$-magmas. A morphism between 
two cubical reflexive $\infty$-magmas is a morphism of their underlying
cubical reflexive sets which respects partial operations $(\circ^{n}_{j})_{n\geq 1, j\in\llbracket 1,n \rrbracket}$. 
The category of cubical reflexive $\infty$-magmas is noted $\I\Cu\Mag_\text{r}$
\end{definition}


\subsection{Definition}
\label{strict-cat}

Strict cubical $\infty$-categories are cubical reflexive $\infty$-magmas such that partials operations are associative and also
we require the following axioms :

\begin{enumerate}[(i)]
\item \label{interchange} The interchange laws : 
    $(a\circ^{n}_{i}b)\circ^{n}_{j}(c\circ^{n}_{i}d)=(a\circ^{n}_{j}c)\circ^{n}_{i}(b\circ^{n}_{j}d)$ whenever
    both sides are defined \\
    
  \item $1^{n}_{n+1,i}(a\circ^{n}_{j}b)=1^{n}_{n+1,i}(a)\circ^{n+1}_{j+1}1^{n}_{n+1,i}(b)$ if $1\leq i\leq j\leq n$
  
  $1^{n}_{n+1,i}(a\circ^{n}_{j}b)=1^{n}_{n+1,i}(a)\circ^{n+1}_{j}1^{n}_{n+1,i}(b)$ if $1\leq j<i\leq n+1$  
  
  \item $1^{n,\gamma}_{n+1,i}(a\circ^{n}_{j}b)=1^{n,\gamma}_{n+1,i}(a)\circ^{n+1}_{j+1}1^{n,\gamma}_{n+1,i}(b)$ if $1\leq i< j\leq n$
  
  $1^{n,\gamma}_{n+1,i}(a\circ^{n}_{j}b)=1^{n,\gamma}_{n+1,i}(a)\circ^{n+1}_{j}1^{n,\gamma}_{n+1,i}(b)$ if $1\leq j< i\leq n$
  
\item First transport laws : for $1\leq j\leq n$
\[
1^{n,+}_{n+1,j}(a\circ^{n}_{j}b)=
  \begin{bmatrix}
    1^{n,+}_{n+1,j}(a) & 1^{n}_{n+1,j}(a) \\
    1^{n}_{n+1,j+1}(a)& 1^{n,+}_{n+1,j}(b) 
 \end{bmatrix}
\]

   \item Second transport laws : for $1\leq j\leq n$
\[
1^{n,-}_{n+1,j}(a\circ^{n}_{j}b)=
  \begin{bmatrix}
    1^{n,-}_{n+1,j}(a) & 1^{n}_{n+1,j+1}(b) \\
    1^{n}_{n+1,j}(b)& 1^{n,-}_{n+1,j}(b) 
 \end{bmatrix}
\]

 \item  for $1\leq j\leq n$, $1^{n,+}_{n+1,i}(x)\circ^{n+1}_{i}1^{n,-}_{n+1,i}(x)=1^{n}_{n+1,i+1}(x)$ and $1^{n,+}_{n+1,i}(x)
 \circ^{n+1}_{i+1}1^{n,-}_{n+1,i}(x)=1^{n}_{n+1,i}(x)$
 
 \end{enumerate}

The category $\CC$ of strict cubical $\infty$-categories is the full subcategory of 
$\infty\text{-}\mathbb{C}\mathbb{M}\text{ag}_\text{r}$ spanned by strict cubical $\infty$-categories.
A morphism in $\CC$ is called a \textit{strict cubical $\infty$-functor}. We study it more specifically in \ref{cub-funct}
with the perspective to weakened it and to obtain cubical model of weak $\infty$-functors.

\begin{ex}
The following $2$-cubical diagram shows us how compositions and degeneracies work
for strict cubical $\infty$-categories in low dimensions;  consider the following 
$2$-cubical shape : 

\begin{tikzcd}
1(0)\arrow[d,"1(1)"{left}]\arrow[r,"1(1)"]\arrow[rd,phantom,"1(2)"]&1(0)\arrow[d,"1(1)"{right}]\\
1(0)\arrow[d,"1(1)"{left}]\arrow[r,"1(1)"]\arrow[rd,phantom,"1(2)"]&1(0)\arrow[d,"1(1)"{right}, near end]\arrow[r,"1(1)"]
\arrow[rd,phantom,"1(2)"]&1(0)\arrow[d,"1(1)"{right}]\\
1(0) \arrow[r,"1(1)"{below}]&1(0)\arrow[r,"1(1)"{below}]&1(0)
\end{tikzcd}

Our objectif is to see it as a $2$-cube and degeneracies 
permit it :

\begin{tikzcd}
1(0)\arrow[dd,"1(1)"{left}]\arrow[rr,"1(1)"]\arrow[rrdd,phantom,"1(2)"]
&&1(0)\arrow[dd,"1(1)", near start]\arrow[rr,"1^{0}_{1}(1(0))"]\arrow[rrdd,phantom,"1^1_{2,2}(1(1))"]&&1(0)\arrow[dd,"1(1)", near start]\arrow[rr,"1(1)"]
\arrow[rrdd,phantom,"1^{1,-}_{2,1}(1(1))"]
&&1(0)\arrow[dd,"1^{0}_{1}(1(0))", near start]\arrow[rr,"1^{0}_{1}(1(0))"]\arrow[rrdd,phantom,"1^{1,+}_{2,1}(1(1))"]&&1(0)\arrow[dd,"1(1)"]\\\\
1(0)\arrow[dd,"1(1)"{left}]\arrow[rr,"1(1)"]\arrow[rrdd,phantom,"1(2)"]&&1(0)\arrow[dd,"1(1)"]\arrow[rr,"1^{0}_{1}(1(0))"]\arrow[rrrrdd,phantom,"1^1_{2,2}(1(1))"]&&1(0)
\arrow[rr,"1^{0}_{1}(1(0))"]&&1(0)\arrow[dd,"1(1)"{left}]\arrow[rr,"1(1)"]\arrow[rrdd,phantom,"1(2)"]&&1(0)\arrow[dd,"1(1)"]\\\\
1(0)\arrow[rr,"1(1)"{below}]&&1(0)\arrow[rrrr,"1^{0}_{1}(1(0))"{below}]&&&&1(0)\arrow[rr,"1(1)"{below}]&&1(0)
\end{tikzcd}

We can recognize in this low dimensional diagram the connections $1^{1,-}_{2,1}$, $1^{1,+}_{2,1}$ and the more 
"classical" reflexivity  $1^{0}_{1}$ and $1^1_{2,2}$ which are important tools provided by the underlying cubical 
strict omega structure. Thanks to the axioms for cubical strict $\infty$-categories the two shapes above represent 
the same operation $1(2)\circ^{2}_{1}1(2)\circ^{2}_{2}1(2)$.
\end{ex}
 
\subsection{The monad of cubical strict $\infty$-categories}
\label{monad-strict}
In this section we describe cubical strict $\infty$-categories as algebras for a monad $\mathbb{S}$ on 
$\mathbb{C}\E$. We hope it to be a specific ingredient to compare globular strict
$\infty$-categories with cubical strict $\infty$-categories. Also we conjecture that $\mathbb{S}$ is cartesian.

Consider the forgetful functor : $\xymatrix{\CC\ar[rr]^{U}&&\mathbb{C}\E}$
which associate to any strict cubical $\infty$-category its underlying cubical set and which associate to any
strict cubical $\infty$-functor its underlying morphism of cubical sets. 
\begin{proposition}
The functor $U$ is right adjoint and monadic.
\end{proposition}
Its left adjoint is denoted $F$. The proof is based on exhibiting a good morphism of projective sketches 
and it is also based on two results of Foltz \cite{foltz} and Lair \cite{lair-triplabilite} : the non-trivial part is 
to exhibit the projective sketch $\mathcal{E}_\text{C}$ of the 
category $\CC$ and we shall easily see that we get an inclusion of projective sketches :
\begin{tikzcd}
\mathcal{E}_\text{S}\arrow[rr,"i",hook]&&\mathcal{E}_\text{C}
\end{tikzcd}
where $\mathcal{E}_\text{S}$ is the sketch of cubical sets \ref{cub-sets}.
Now we have the commutative diagram 

\begin{tikzcd}
Mod(\mathcal{E}_\text{C})\arrow[dd,"iso"{left}]\arrow[rr,"Mod(i)"]&&Mod(\mathcal{E}_\text{C})\arrow[dd,"iso"]\\\\
\CC\arrow[rr,"U"]&&\Cu\E
\end{tikzcd}

when passing to models in $\E$.
It shows that $i$ induces the forgetful functor $U$ and that $U$ 
is right adjoint thanks to the \textit{sheafification theorem} of Foltz \cite{foltz}. 
\begin{remark}
This result of Foltz is called the \textit{sheafification theorem}, because it generalizes the construction
of sheaves associated to presheaves for a given site.
\end{remark}
Following the terminology
of \cite{lair-triplabilite} we say that the functor $U$ is \textit{projectively sketchable}. 

Also we shall easily see that each distinguished cone of $\mathcal{E}_\text{C}$
has its base which factorize $i$ and each object of $\mathcal{E}_\text{C}$ which is not in the image of $i$ is the
vertex of at least one distinguished cone of $\mathcal{E}_\text{C}$. Thus by the theorem of Lair in \cite{lair-triplabilite}
about monadicity, it shows that $U$ is monadic. 
\begin{proof}
\label{sketch-c}
The proof is very similar to those in \cite{penon} : actually we are going to see that 
the category $\I\mathbb{C}\mathbb{C}\text{AT}$ is projectively sketchable. Let us denote by $\mathcal{E}_\text{C}$ the sketch of 
$\CC$.    
The description of $\mathcal{E}_\text{C}$ started with the 
description of $\mathcal{E}_\text{M}$ in \ref{strict-infty}. We carry
on to it in describing the sketch behind the interchange laws, which
shall complete main parts of $\mathcal{E}_\text{C}$ :

\begin{itemize}
\item In the first generation of $\mathcal{E}_\text{C}$ we start with three cones :

\begin{tikzcd}
Z_n\underset{Z_{n-1,i}}\times Z_n\arrow[d,"\rho^{n}_{0,i}"{left}]
\arrow[rr,"\rho^{n}_{1,i}"]&&
Z_n\arrow[d,"s^{n}_{n-1,i}"]\\
Z_n\arrow[rr,"t^{n}_{n-1,i}"{below}]&&Z_{n-1}\end{tikzcd}

\begin{tikzcd}
Z_n\underset{Z_{n-1,j}}\times Z_n\arrow[d,"\rho^{n}_{0,j}"{left}]
\arrow[rr,"\rho^{n}_{1,j}"]&&
Z_n\arrow[d,"s^{n}_{n-1,j}"]\\
Z_n\arrow[rr,"t^{n}_{n-1,j}"{below}]&&Z_{n-1}\end{tikzcd}

\begin{tikzcd}
&&&E_{ijn}\arrow[llldd,"\pi^n_{00}"{left}]\arrow[ldd,"\pi^n_{10}"]\arrow[rdd,"\pi^n_{01}"]\arrow[rrrdd,"\pi^n_{11}"]\\\\
Z_n\arrow[rd,"t^{n}_{n-1,i}"{left}]&&Z_n\arrow[ld,"s^{n}_{n-1,i}"]\arrow[rd,"s^{n}_{n-1,j}"{left}]&&
Z_n\arrow[ld,"t^{n}_{n-1,j}"]\arrow[rd,"t^{n}_{n-1,i}"{left}]
&&Z_n\arrow[ld,"s^{n}_{n-1,i}"]\\
&Z_{n-1}&&Z_{n-1}&&Z_{n-1}
\end{tikzcd}

\item Then we consider the following commutative diagrams :

\[\begin{tikzcd}
E_{ijn}\arrow[ddddrr,"\pi^n_{00}"{left}]\arrow[ddrr,dotted,"p^n_{1000}"]\arrow[ddrrrr,"\pi^n_{10}"]\\\\
&&Z_n\underset{Z_{n-1,i}}\times Z_n\arrow[rr,"\rho^{n}_{1,i}"]\arrow[dd,"\rho^{n}_{0,i}"{left}]&&Z_n
\arrow[dd,"s^{n}_{n-1,i}"]\\\\
&&Z_n\arrow[rr,"t^{n}_{n-1,i}"{below}]&&Z_n
\end{tikzcd}\qquad \begin{tikzcd}
E_{ijn}\arrow[ddddrr,"\pi^n_{01}"{left}]\arrow[ddrr,dotted,"p^n_{1101}"]\arrow[ddrrrr,"\pi^n_{11}"]\\\\
&&Z_n\underset{Z_{n-1,i}}\times Z_n\arrow[rr,"\rho^{n}_{1,i}"]\arrow[dd,"\rho^{n}_{0,i}"{left}]&&Z_n
\arrow[dd,"s^{n}_{n-1,i}"]\\\\
&&Z_n\arrow[rr,"t^{n}_{n-1,i}"{below}]&&Z_n
\end{tikzcd}\]

\[\begin{tikzcd}
E_{ijn}\arrow[ddddrr,"\pi^n_{11}"{left}]\arrow[ddrr,dotted,"p^n_{1011}"]\arrow[ddrrrr,"\pi^n_{10}"]\\\\
&&Z_n\underset{Z_{n-1,j}}\times Z_n\arrow[rr,"\rho^{n}_{1,j}"]\arrow[dd,"\rho^{n}_{0,j}"{left}]&&Z_n
\arrow[dd,"s^{n}_{n-1,j}"]\\\\
&&Z_n\arrow[rr,"t^{n}_{n-1,j}"{below}]&&Z_n
\end{tikzcd}\qquad \begin{tikzcd}
E_{ijn}\arrow[ddddrr,"\pi^n_{01}"{left}]\arrow[ddrr,dotted,"p^n_{0001}"]\arrow[ddrrrr,"\pi^n_{00}"]\\\\
&&Z_n\underset{Z_{n-1,j}}\times Z_n\arrow[rr,"\rho^{n}_{1,j}"]\arrow[dd,"\rho^{n}_{0,j}"{left}]&&Z_n
\arrow[dd,"s^{n}_{n-1,j}"]\\\\
&&Z_n\arrow[rr,"t^{n}_{n-1,j}"{below}]&&Z_n
\end{tikzcd}\]

\item We consider then (still in the first generation) the following two commutative diagrams :

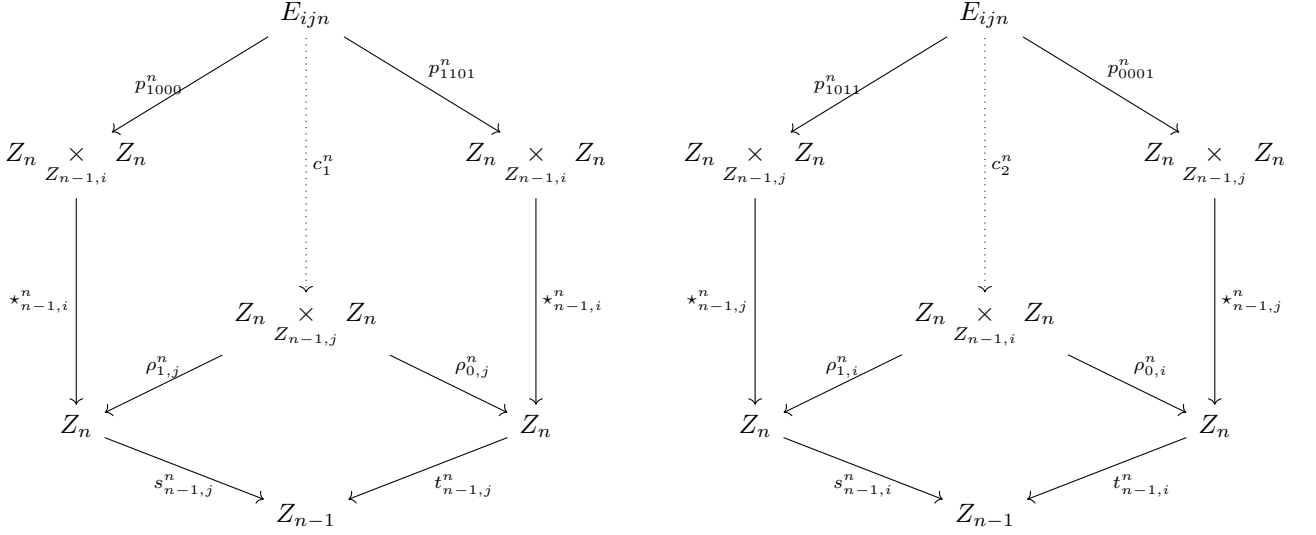
\[\begin{tikzcd}
&E_{ijn}\arrow[ldd,"p^n_{1000}"{left}]\arrow[dddd,dotted,"c^n_1"]\arrow[rdd,"p^n_{1101}"]\\\\
Z_n\underset{Z_{n-1,i}}\times Z_n \arrow[ddd,"\star^{n}_{n-1,i}"{left}]&&
Z_n\underset{Z_{n-1,i}}\times Z_n\arrow[ddd,"\star^{n}_{n-1,i}"]\\\\
&Z_n\underset{Z_{n-1,j}}\times Z_n\arrow[ld,"\rho^{n}_{1,j}"{above}]\arrow[rd,"\rho^{n}_{0,j}"]\\
Z_n\arrow[rd,"s^{n}_{n-1,j}"{below}]&&Z_n\arrow[ld,"t^{n}_{n-1,j}"]\\
&Z_{n-1}
\end{tikzcd}\qquad \begin{tikzcd}
&E_{ijn}\arrow[ldd,"p^n_{1011}"{left}]\arrow[dddd,dotted,"c^n_2"]\arrow[rdd,"p^n_{0001}"]\\\\
Z_n\underset{Z_{n-1,j}}\times Z_n \arrow[ddd,"\star^{n}_{n-1,j}"{left}]&&
Z_n\underset{Z_{n-1,j}}\times Z_n\arrow[ddd,"\star^{n}_{n-1,j}"]\\\\
&Z_n\underset{Z_{n-1,i}}\times Z_n\arrow[ld,"\rho^{n}_{1,i}"{above}]\arrow[rd,"\rho^{n}_{0,i}"]\\
Z_n\arrow[rd,"s^{n}_{n-1,i}"{below}]&&Z_n\arrow[ld,"t^{n}_{n-1,i}"]\\
&Z_{n-1}
\end{tikzcd}\]

\item Finally we consider the following commutative diagram of interchange
laws

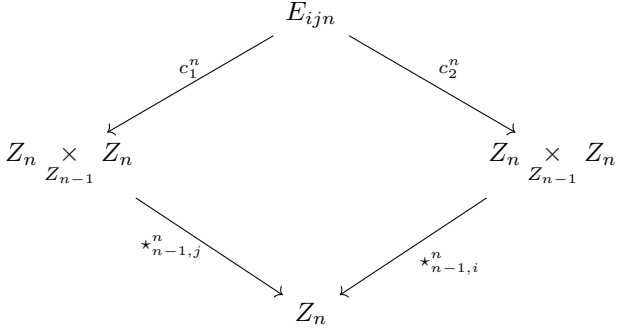
\begin{tikzcd}
&&E_{ijn}\arrow[ddll,"c^n_1"{above}]\arrow[ddrr,"c^n_2"]\\\\
Z_n\underset{Z_{n-1}}\times Z_n\arrow[rrdd,"\star^{n}_{n-1,j}"{left}]&&&&Z_n\underset{Z_{n-1}}\times Z_n
\arrow[lldd,"\star^{n}_{n-1,i}"]\\\\
&&Z_n
\end{tikzcd}
\end{itemize}

\end{proof}

The monad of strict cubical $\infty$-categories on cubical sets is denoted $\mathbb{S}=(S,\lambda,\mu)$. Here $\lambda$ is 
the unit map of $\mathbb{S}$ : 

$$\xymatrix{1_{\mathbb{C}\E}\ar[rr]^{\lambda}&&S}$$
and $\mu$ is the multiplication of 
$\mathbb{S}$ : 
$$\xymatrix{S^2\ar[rr]^{\mu}&&S}$$ 
We conjecture the following result \cite{camark-cart} :
\begin{conj} 
\label{conj-cart}
$\mathbb{S}=(S,\lambda,\mu)$ is cartesian.
\end{conj}

\section{The category of cubical weak $\infty$-categories}
\label{stretchings}

In this section we exhibit the first algebraic models of cubical weak $\infty$-categories. Thus these models are
algebras for a specific monad that we describe below, acting on the category $\CS$ of cubical sets \ref{CS}. We shall propose in \cite{camark-cub-1}
 other algebraic models of cubical weak $\infty$-categories defined as algebras
for a specific cubical operad, and this operadical models are possible up to the conjecture \ref{conj-cart} just above. However
the algebraic models of this article doesn't need the conjecture \ref{conj-cart}, and this is one advantage of the 
\textit{Weakened by Stretchings} initiated by Jacques Penon in \cite{penon} which doesn't care about cartesianess of monads.
However because in \cite{bat-pen} Michael Batanin had proved that there is a morphism of operads
\begin{tikzcd}
\mathbb{B}^{0}_C\arrow[r]&\mathbb{P}^{0}_C
\end{tikzcd}
where $\mathbb{B}^{0}_C$-algebras are weak $\infty$-categories of Batanin and $\mathbb{P}^{0}_C$-algebras are weak 
$\infty$-categories of Penon, we suspect the same phenomena between the operad which underlies the monad 
described just below and the operad of cubical weak $\infty$-categories described in \cite{camark-cub-1}
.
\subsection{The category of cubical categorical stretchings}
\label{stretchings}

We have defined the category $\I\mathbb{C}\mathbb{M}\text{ag}_\text{r}$ of cubical reflexive $\infty$-magmas in \ref{strict-infty}. 
Objects of this category plus cubical strict $\infty$-categories, allow to define \textit{cubical categorical stretchings} (see below), 
which are objects of the category 
$\I\mathbb{C}\mathbb{E}\text{tC}$. This category is the key ingredient to weakened cubical strict $\infty$-categories as it 
was done in \cite{penon} for the globular paradigm. Our cubical weak $\infty$-categories are algebraic in the 
sense that they are algebras (\ref{weak-def}) for a monad on $\mathbb{C}\E$ which is build by using the category 
of cubical categorical stretchings. Our way to build the category $\I\mathbb{C}\mathbb{M}\text{ag}_\text{r}$ allow
to weakened the whole structure of cubical strict $\infty$-categories. As we shall see, the central notion of 
\textit{cubical contractions} (see below) are more subtle than globular contractions of \cite{batanin-main,penon} : 
in particular they must be thought with an inductive definition on the dimension $n$ of the $n$-cells ($n\in\mathbb{N}$).
 
 The category $\Et$ of cubical categorical stretchings has as objects 
 quintuples 
  \[ \mathbb{E}=(M,C,\pi,([-;-]^{n}_{n+1,j})_{n\in\mathbb{N};j\in\{1,...,n+1\}},
  ([-;-]^{n,\gamma}_{n+1,j})_{n\geq 1;j\in\{1,...,n\};\gamma\in\{-,+\}})\] 
  where $M$ is a reflexive cubical $\infty$-magma, $C$ is a cubical strict $\infty$-category, $\pi$ is
a morphism in $\I\mathbb{C}\mathbb{M}\text{ag}_\text{r}$
\[\xymatrix{M\ar[rr]^{\pi}&&C}\] 
and $([-;-]^{n}_{n+1,j})_{n\in\mathbb{N};j\in\{1,...,n+1\}}, ([-;-]^{n,\gamma}_{n+1,j})_{n\geq 1;j\in\{1,...,n\};\gamma\in\{-,+\}}$ are
extra structures called the \textit{cubical bracketing structures}, and which are the cubical analogue of the key
structure of the Penon approach to weakening the axioms of strict
$\infty$-categories; it is for us the key structures which are going to weakening the axioms of cubical strict $\infty$-categories. Let be more precise about it :  

For $n\geq 1$ and for all integer $k\geq 1$, consider
the following subsets of $M_{n}\times M_{n}$

\begin{itemize}
\item $\underline{M}_{n}=\{(\alpha,\beta)\in M_{n}\times M_{n}: \pi_{n}(\alpha)=\pi_{n}(\beta)\}$ 
 
\item $\underline{M}^s_{n,j}=\{(\alpha,\beta)\in M_{n}\times M_{n}: 
s^{n}_{n-1,j}(\alpha)=s^{n}_{n-1,j}(\beta)\text{ and }
\pi_{n}(\alpha)=\pi_{n}(\beta)\}$

\item $\underline{M}^t_{n,j}=\{(\alpha,\beta)\in M_{n}\times M_{n}: 
t^{n}_{n-1,j}(\alpha)=t^{n}_{n-1,j}(\beta)\text{ and }
\pi_{n}(\alpha)=\pi_{n}(\beta)\}$
\end{itemize}
and also we consider
$\underline{M}_{0}=\{(\alpha,\beta)\in M_{0}\times M_{0}:
  \alpha=\beta\}$ 
 
Thus these extra structures are given by maps
\[([-;-]^{n}_{n+1,j} :\xymatrix{\underline{M}_{n}\ar[r]^{}&M_{n+1}})_{n\in\mathbb{N};j\in\{1,...,n+1\}}\]

such that
\begin{itemize}
\item If $1\leq i<j\leq n+1$, then
  $$s^{n+1}_{n,i}([\alpha,\beta]^{n}_{n+1,j})=[s^{n}_{n-1,i}(\alpha),s^{n}_{n-1,i}(\beta)]^{n-1}_{n,j-1},
  \text{ and }t^{n+1}_{n,i}([\alpha,\beta]^{n}_{n+1,j})=[t^{n}_{n-1,i}(\alpha),t^{n}_{n-1,i}(\beta)]^{n-1}_{n,j-1}$$
  
\item If $1\leq j<i\leq n+1$ then 
$$s^{n+1}_{n,i}([\alpha,\beta]^{n}_{n+1,j})=[s^{n}_{n-1,i-1}(\alpha),s^{n}_{n-1,i-1}(\beta)]^{n-1}_{n,j}, 
\text{ and }t^{n+1}_{n,i}([\alpha,\beta]^{n}_{n+1,j})=[t^{n}_{n-1,i-1}(\alpha),t^{n}_{n-1,i-1}(\beta)]^{n-1}_{n,j}$$
  
\item If $i=j$ then
$$s^{n+1}_{n,i}([\alpha,\beta]^{n}_{n+1,j})=\alpha \text{ and } t^{n+1}_{n,i}([\alpha,\beta]^{n}_{n+1,j})=\beta$$

\item
  $\pi_{n+1}([\alpha,\beta]^{n}_{n+1,j})=1^{n}_{n+1,j}(\pi_{n}(\alpha))=1^{n}_{n+1,j}(\pi_{n}(\beta))$,
\item $\forall \alpha\in M_{n}, [\alpha,\alpha]^{n}_{n+1,j}=1^{n}_{n+1,j}(\alpha)$.
\end{itemize}

and also are given by maps

\[([-;-]^{n,-}_{n+1,j} :\xymatrix{\underline{M}^s_{n,j}\ar[r]^{}&
M_{n+1}})_{n\geq 1;j\in\{1,...,n\}}\text{ and } ([-;-]^{n,+}_{n+1,j} :\xymatrix{\underline{M}^{t}_{n,j}\ar[r]^{}&M_{n+1}})_{n\geq 1;j\in\{1,...,n\}}\]
such that

\begin{itemize}
\item for $1\leq j\leq n$ we have :
\begin{itemize}

 \item $s^{n+1}_{n,j}([\alpha;\beta]^{n,-}_{n+1,j})=\alpha$ and
 $s^{n+1}_{n,j+1}([\alpha;\beta]^{n,-}_{n+1,j})=\beta$
\item $t^{n+1}_{n,j}([\alpha;\beta]^{n,+}_{n+1,j})=\alpha$ and $t^{n+1}_{n,j+1}([\alpha;\beta]^{n,-}_{n+1,j})=\beta$
\item $s^{n+1}_{n,j}([\alpha;\beta]^{n,+}_{n+1,j})=s^{n+1}_{n,j+1}([\alpha;\beta]^{n,+}_{n+1,j})=
[s^{n}_{n-1,j}(\alpha);s^{n}_{n-1,j}(\beta)]^{n-1}_{n,j}$
\item $t^{n+1}_{n,j}([\alpha;\beta]^{n,-}_{n+1,j})=t^{n+1}_{n,j+1}([\alpha;\beta]^{n,-}_{n+1,j})=
[t^{n}_{n-1,j}(\alpha);t^{n}_{n-1,j}(\beta)]^{n-1}_{n,j}$
\end{itemize}

\item for $1\leq i,j\leq n+1$
\begin{itemize}
\item $s^{n+1}_{n,i}([\alpha;\beta]^{n,\gamma}_{n+1,j})=
\begin{cases}
  [s^{n}_{n-1,i}(\alpha);s^{n}_{n-1,i}(\beta)]^{n-1,\gamma}_{n,j-1}\text{  if  } 
  
  1\leq i<j\leq n \\
  [s^{n}_{n-1,i-1}(\alpha);s^{n}_{n-1,i-1}(\beta)]^{n-1,\gamma}_{n,j}\text{  if  }
  2\leq j+1<i\leq n+1
\end{cases}$ 

\item $t^{n+1}_{n,i}([\alpha;\beta]^{n,\gamma}_{n+1,j})=
\begin{cases}
  [t^{n}_{n-1,i}(\alpha);t^{n}_{n-1,i}(\beta)]^{n-1,\gamma}_{n,j-1}\text{  if  } 
  
  1\leq i<j\leq n \\
  [t^{n}_{n-1,i-1}(\alpha);t^{n}_{n-1,i-1}(\beta)]^{n-1,\gamma}_{n,j}\text{  if  }
  2\leq j+1<i\leq n+1
\end{cases}$ 

\end{itemize}
\item $\pi_{n+1}([\alpha;\beta]^{n,\gamma}_{n+1,j})=1^{n,\gamma}_{n+1,j}(\pi_{n}(\alpha))
=1^{n,\gamma}_{n+1,j}(\pi_{n}(\beta))$
\item $\forall \alpha\in M_{n}, [\alpha,\alpha]^{n,\gamma}_{n+1,j}=1^{n,\gamma}_{n+1,j}(\alpha)$.
\end{itemize}

A morphism of cubical categorical stretchings
\[\xymatrix{\mathbb{E}\ar[rr]^{(m,c)}&&\mathbb{E}'}\]
is given by the following commutative square in $\I\mathbb{C}\mathbb{M}\text{ag}_\text{r}$,
\[\xymatrix{M\ar[d]_{\pi}\ar[rr]^{m}&&M'\ar[d]^{\pi'}\\
  C\ar[rr]_{c}&&C' }\]
such that for all $n\in \mathbb{N}$, and for all
$(\alpha,\beta)\in\widetilde{M}_{n}$,
\[m_{n+1}([\alpha,\beta]^{n}_{n+1,j})=[m_{n}(\alpha),m_{n}(\beta)]^{n}_{n+1,j} 
\hspace{.2cm}(j\in\{1,...,n+1\})\]

and

\[m_{n+1}([\alpha,\beta]^{n,\gamma}_{n+1,j})=[m_{n}(\alpha),
m_{n}(\beta)]^{n,\gamma}_{n+1,j}\hspace{.2cm}(j\in\{1,...,n\}, \gamma\in\{-,+\})\]

The category of cubical categorical stretchings is denoted $\Et$

Now consider the forgetful functor: 
$\xymatrix{\Et\ar[rr]^{U_{}}&&\Cu\E}$
given by :

$$\xymatrix{(M,C,\pi,([-;-]^{n}_{n+1,j})_{n\in\mathbb{N};j\in\{1,...,n\}},
([-;-]^{n,\gamma}_{n+1,j})_{n\in\mathbb{N};j\in\{1,...,n\};\gamma\in\{-,+\}})\ar@{|->}[r]^{}&M },$$

\begin{proposition}
 The functor $U$ just above
 has a left adjoint which produces a monad
$\mathbb{W}=(W,\eta,\nu)$ on the category of cubical sets.
\end{proposition}
\begin{proof}

The proof is very similar to those in \cite{penon,kamel} : Actually it is not difficult to see that 
the category $\Et$ and the category $\mathbb{C}\E$
are both projectively sketchable. The sketch of cubical sets is denoted by $\mathcal{E}_\text{S}$ (see \ref{cub-sets})
and the sketch of the cubical categorical stretchings is denoted by $\mathcal{E}_\text{E}$. Main parts of
this sketch is described just below, and we see that $\mathcal{E}_\text{E}$ contains $\mathcal{E}_\text{S}$ :
\begin{tikzcd}
\mathcal{E}_\text{S}\arrow[rr,"j",hook]&&\mathcal{E}_\text{E}
\end{tikzcd},
and is such that it induces a forgetful functor $\xymatrix{\Et\ar[rr]^{U}&&\mathbb{C}\E}$
such that we have the commutative diagram 

\begin{tikzcd}
Mod(\mathcal{E}_\text{E})\arrow[dd,"iso"{left}]\arrow[rr,"Mod(j)"]&&Mod(\mathcal{E}_\text{S})\arrow[dd,"iso"]\\\\
\Et\arrow[rr,"U"]&&\Cu\E
\end{tikzcd}

 which shows that  $U$ is right adjoint by the theorem of Foltz \cite{foltz}.
  
  Actually in \ref{strict-infty} we described the sketch $\mathcal{E}_\text{M}$ of cubical $\infty$-magmas, which where used to describe in \ref{sketch-c} main part of the 
  sketch $\mathcal{E}_\text{C}$ of cubical strict $\infty$-categories. Thus we have already some part of the sketch $\mathcal{E}_\text{E}$ that we complete by
  sketching operations  $[-;-]^{n}_{n+1,j}$ and $[-;-]^{n,\gamma}_{n+1,j}$ plus their axioms. With previous descriptions of sketches, and 
the one below, we shall see that we obtain the following inclusions
of sketches :

$$\begin{tikzcd}
\mathcal{E}_\text{S}\arrow[rr,hook]&&\mathcal{E}_\text{M}
\arrow[rr,hook]&&
\mathcal{E}_\text{C}\arrow[rr,hook]&&\mathcal{E}_\text{E} 
\end{tikzcd}$$
  
\begin{description}
  
\item \underline{Description of $\mathcal{E}_\text{E}$}

\begin{itemize}

\item In the first generation we start with the following four cones :

\begin{tikzcd}
\underline{M_n}\arrow[dd,"\pi^{n}_{0}"{left}]\arrow[rr,"\pi^{n}_{1}"]
&&M_{n}\arrow[dd,"\pi_n"]\\\\
M_n\arrow[rr,"\pi_n"{below}]&&Z_{n}\end{tikzcd}
\qquad\begin{tikzcd}
\underline{M}^{s}_{n,j}\arrow[rr,"\pi^{n,s}_{1}"]\arrow[dd,"\pi^{n,s}_{0}"{left}]
&&M_{n}\arrow[dd,"s^{n}_{n-1,j}"]\\\\
M_{n}\arrow[rr,"s^{n}_{n-1,j}"{below}]&&Z_{n}
\end{tikzcd}\

\begin{tikzcd}
\underline{M}^{t}_{n,j}\arrow[rr,"\pi^{n,t}_{1}"]\arrow[dd,"\pi^{n,t}_{0}"{left}]
&&M_{n}\arrow[dd,"t^{n}_{n-1,j}"]\\\\
M_{n}\arrow[rr,"t^{n}_{n-1,j}"{below}]&&Z_{n}
\end{tikzcd}\qquad\begin{tikzcd}
&&M_{n}\times M_n\arrow[ddll,"p^n_0"{above}]\arrow[ddrr,"p^n_1"]\\\\
M_n&&&&M_n\end{tikzcd}

\item We consider the following commutative diagrams :

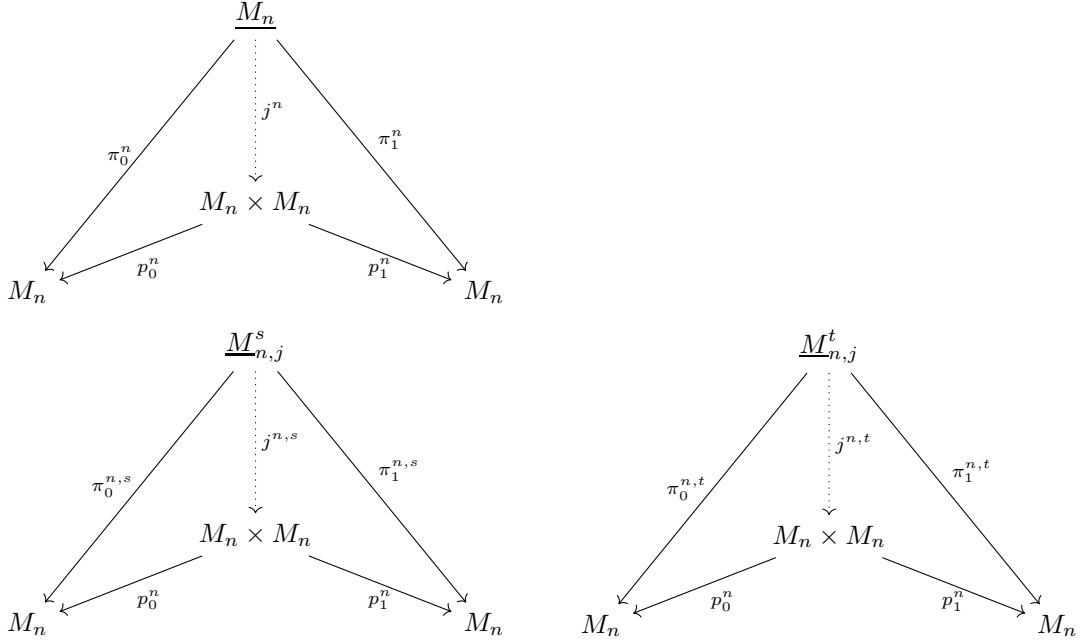
\begin{tikzcd}
&&\underline{M_n}\arrow[ddddll,"\pi^{n}_{0}"{left}]
\arrow[ddddrr,"\pi^{n}_{1}"]\arrow[ddd,dotted,"j^n"]\\\\\\
&&M_{n}\times M_n\arrow[dll,"p^n_0"]\arrow[drr,"p^n_1"{below}]\\
M_n&&&&M_n\end{tikzcd}

\begin{tikzcd}
&&\underline{M}^{s}_{n,j}\arrow[ddddll,"\pi^{n,s}_{0}"{left}]
\arrow[ddddrr,"\pi^{n,s}_{1}"]\arrow[ddd,dotted,"j^{n,s}"]\\\\\\
&&M_{n}\times M_n\arrow[dll,"p^n_0"]\arrow[drr,"p^n_1"{below}]\\
M_n&&&&M_n\end{tikzcd}\qquad\begin{tikzcd}
&&\underline{M}^{t}_{n,j}\arrow[ddddll,"\pi^{n,t}_{0}"{left}]
\arrow[ddddrr,"\pi^{n,t}_{1}"]\arrow[ddd,dotted,"j^{n,t}"]\\\\\\
&&M_{n}\times M_n\arrow[dll,"p^n_0"]\arrow[drr,"p^n_1"{below}]\\
M_n&&&&M_n\end{tikzcd}
\end{itemize}

\item We have the following commutative diagrams which define
the $"s\times s"$ :

\begin{itemize}

\item If $1\leq i<j\leq n+1$

\begin{itemize}
\item for sources

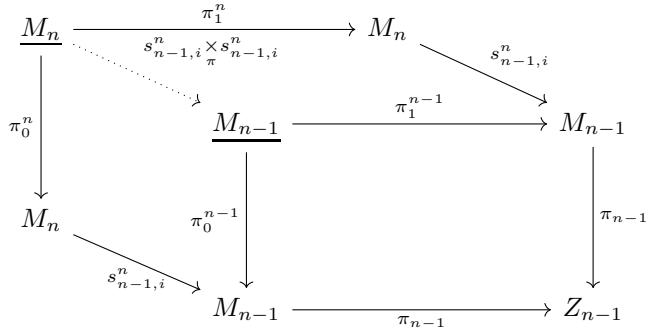
\begin{tikzcd}
\underline{M_n}
\arrow[rrr,"\pi^{n}_{1}"]\arrow[dd,"\pi^{n}_{0}"{left}]
\arrow[rrd,dotted,"s^{n}_{n-1,i}\underset{\pi}\times s^{n}_{n-1,i}"]&&&M_n
\arrow[rrd,"s^{n}_{n-1,i}"]\\
&&\underline{M_{n-1}}
\arrow[dd,"\pi^{n-1}_{0}"{left}]\arrow[rrr,"\pi^{n-1}_{1}"]&&&M_{n-1}\arrow[dd,"\pi_{n-1}"]\\
M_{n}\arrow[rrd,"s^{n}_{n-1,i}"{below}]\\
&&M_{n-1}\arrow[rrr,"\pi_{n-1}"{below}]&&&Z_{n-1}
\end{tikzcd}

\item for targets

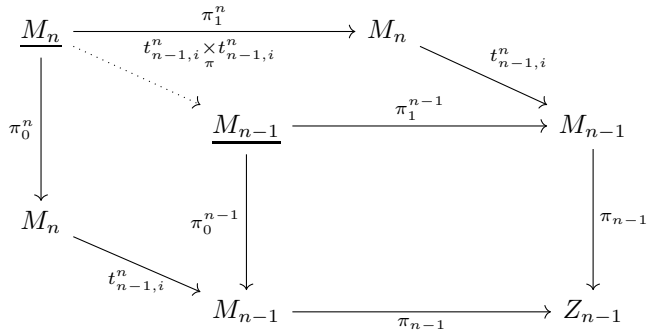
\begin{tikzcd}
\underline{M_n}
\arrow[rrr,"\pi^{n}_{1}"]\arrow[dd,"\pi^{n}_{0}"{left}]
\arrow[rrd,dotted,"t^{n}_{n-1,i}\underset{\pi}\times t^{n}_{n-1,i}"]&&&M_n
\arrow[rrd,"t^{n}_{n-1,i}"]\\
&&\underline{M_{n-1}}
\arrow[dd,"\pi^{n-1}_{0}"{left}]\arrow[rrr,"\pi^{n-1}_{1}"]&&&M_{n-1}\arrow[dd,"\pi_{n-1}"]\\
M_{n}\arrow[rrd,"t^{n}_{n-1,i}"{below}]\\
&&M_{n-1}\arrow[rrr,"\pi_{n-1}"{below}]&&&Z_{n-1}
\end{tikzcd}
\end{itemize}

If $1\leq j<i\leq n+1$
\begin{itemize}

\item for sources

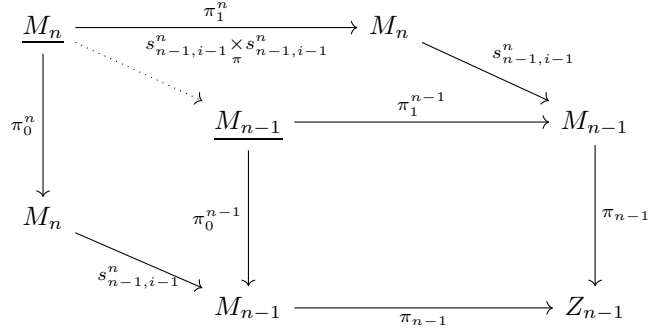
\begin{tikzcd}
\underline{M_n}
\arrow[rrr,"\pi^{n}_{1}"]\arrow[dd,"\pi^{n}_{0}"{left}]
\arrow[rrd,dotted,"s^{n}_{n-1,i-1}\underset{\pi}\times s^{n}_{n-1,i-1}"]&&&M_n
\arrow[rrd,"s^{n}_{n-1,i-1}"]\\
&&\underline{M_{n-1}}
\arrow[dd,"\pi^{n-1}_{0}"{left}]\arrow[rrr,"\pi^{n-1}_{1}"]&&&M_{n-1}\arrow[dd,"\pi_{n-1}"]\\
M_{n}\arrow[rrd,"s^{n}_{n-1,i-1}"{below}]\\
&&M_{n-1}\arrow[rrr,"\pi_{n-1}"{below}]&&&Z_{n-1}
\end{tikzcd}

\item for targets

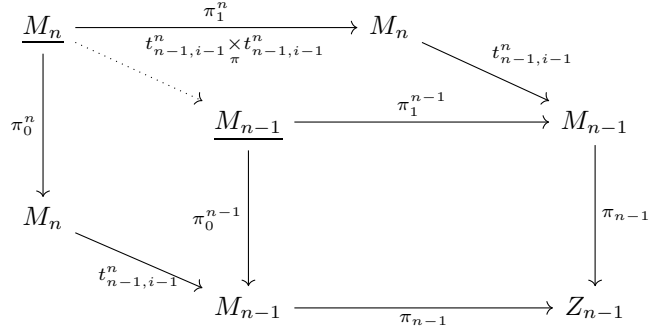
\begin{tikzcd}
\underline{M_n}
\arrow[rrr,"\pi^{n}_{1}"]\arrow[dd,"\pi^{n}_{0}"{left}]
\arrow[rrd,dotted,"t^{n}_{n-1,i-1}\underset{\pi}\times t^{n}_{n-1,i-1}"]&&&M_n
\arrow[rrd,"t^{n}_{n-1,i-1}"]\\
&&\underline{M_{n-1}}
\arrow[dd,"\pi^{n-1}_{0}"{left}]\arrow[rrr,"\pi^{n-1}_{1}"]&&&M_{n-1}\arrow[dd,"\pi_{n-1}"]\\
M_{n}\arrow[rrd,"t^{n}_{n-1,i-1}"{below}]\\
&&M_{n-1}\arrow[rrr,"\pi_{n-1}"{below}]&&&Z_{n-1}
\end{tikzcd}
\end{itemize}
\end{itemize}
\item In the following diagrams we have $\rho\in\{s,t\}$
\begin{itemize}

\item If $1\leq i<j\leq n$
\begin{itemize}

\item for sources

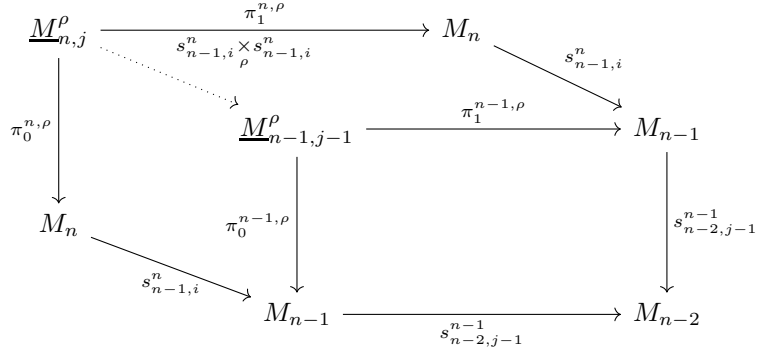
\begin{tikzcd}
\underline{M}^{\rho}_{n,j}\arrow[rrr,"\pi^{n,\rho}_{1}"]
\arrow[dd,"\pi^{n,\rho}_{0}"{left}]
\arrow[rrd,dotted,"s^{n}_{n-1,i}
\underset{\rho}\times s^{n}_{n-1,i}"]&&&M_n
\arrow[rrd,"s^{n}_{n-1,i}"]\\
&&\underline{M}^\rho_{n-1,j-1}
\arrow[dd,"\pi^{n-1,\rho}_{0}"{left}]\arrow[rrr,"
\pi^{n-1,\rho}_{1}"]&&&M_{n-1}\arrow[dd,"s^{n-1}_{n-2,j-1}"]\\
M_{n}\arrow[rrd,"s^{n}_{n-1,i}"{below}]\\
&&M_{n-1}\arrow[rrr,"s^{n-1}_{n-2,j-1}"{below}]&&&M_{n-2}
\end{tikzcd}

\item for targets

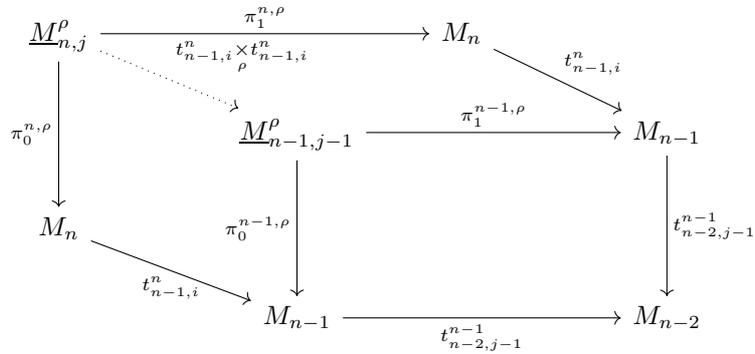
\begin{tikzcd}
\underline{M}^{\rho}_{n,j}
\arrow[rrr,"\pi^{n,\rho}_{1}"]
\arrow[dd,"\pi^{n,\rho}_{0}"{left}]
\arrow[rrd,dotted,"t^{n}_{n-1,i}
\underset{\rho}\times t^{n}_{n-1,i}"]&&&M_n
\arrow[rrd,"t^{n}_{n-1,i}"]\\
&&\underline{M}^\rho_{n-1,j-1}
\arrow[dd,"\pi^{n-1,\rho}_{0}"{left}]
\arrow[rrr,"\pi^{n-1,\rho}_{1}"]&&&M_{n-1}
\arrow[dd,"t^{n-1}_{n-2,j-1}"]\\
M_{n}\arrow[rrd,"t^{n}_{n-1,i}"{below}]\\
&&M_{n-1}\arrow[rrr,"t^{n-1}_{n-2,j-1}"{below}]&&&M_{n-2}
\end{tikzcd}
\end{itemize}

If $1\leq j<i\leq n+1$
\begin{itemize}
\item for sources

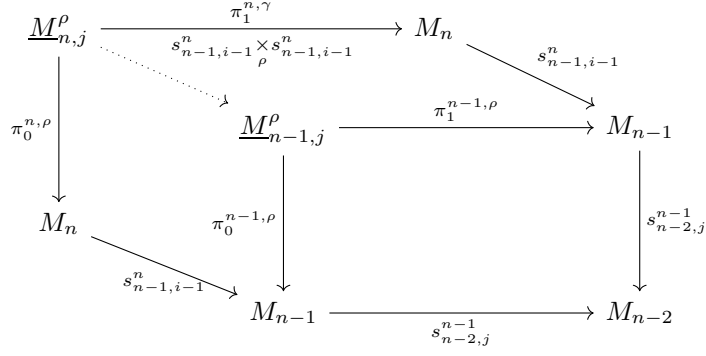
\begin{tikzcd}
\underline{M}^{\rho}_{n,j}\arrow[rrr,"
\pi^{n,\gamma}_{1}"]
\arrow[dd,"\pi^{n,\rho}_{0}"{left}]
\arrow[rrd,dotted,"s^{n}_{n-1,i-1}
\underset{\rho}\times s^{n}_{n-1,i-1}"]&&&M_n
\arrow[rrd,"s^{n}_{n-1,i-1}"]\\
&&\underline{M}^{\rho}_{n-1,j}
\arrow[dd,"\pi^{n-1,\rho}_{0}"{left}]
\arrow[rrr,"\pi^{n-1,\rho}_{1}"]&&&
M_{n-1}\arrow[dd,"s^{n-1}_{n-2,j}"]\\
M_{n}\arrow[rrd,"s^{n}_{n-1,i-1}"{below}]\\
&&M_{n-1}\arrow[rrr,"s^{n-1}_{n-2,j}"{below}]&&&M_{n-2}
\end{tikzcd}

\item for targets

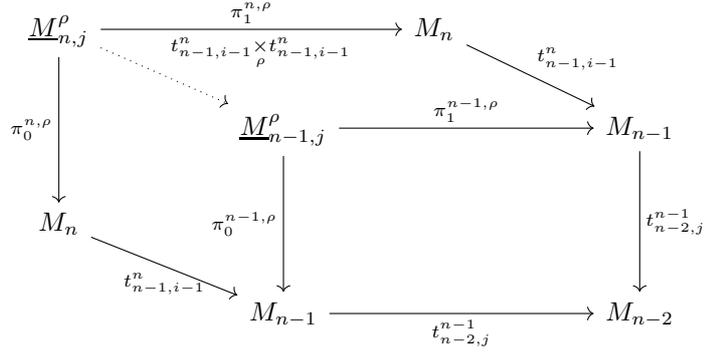
\begin{tikzcd}
\underline{M}^{\rho}_{n,j}
\arrow[rrr,"\pi^{n,\rho}_{1}"]
\arrow[dd,"\pi^{n,\rho}_{0}"{left}]
\arrow[rrd,dotted,"t^{n}_{n-1,i-1}
\underset{\rho}\times t^{n}_{n-1,i-1}"]&&&M_n
\arrow[rrd,"t^{n}_{n-1,i-1}"]\\
&&\underline{M}^{\rho}_{n-1,j}
\arrow[dd,"\pi^{n-1,\rho}_{0}"{left}]
\arrow[rrr,"\pi^{n-1,\rho}_{1}"]&&&
M_{n-1}\arrow[dd,"t^{n-1}_{n-2,j}"]\\
M_{n}\arrow[rrd,"t^{n}_{n-1,i-1}"{below}]\\
&&M_{n-1}\arrow[rrr,"t^{n-1}_{n-2,j}"{below}]&&&M_{n-2}
\end{tikzcd}

\end{itemize}
\end{itemize}

\item \underline{Second generation}\footnote{Indeed the sketch $\mathcal{E}_E$ has "two floors", if we follow the French 
terminology "\textit{esquisses à deux étages}" in \cite{lair-triplabilite}. We have preferred the name "\textit{generation}" instead,
in order to provide the intuition that these operations have "operations ancestors".}
\begin{itemize}
\item We consider the two cones :

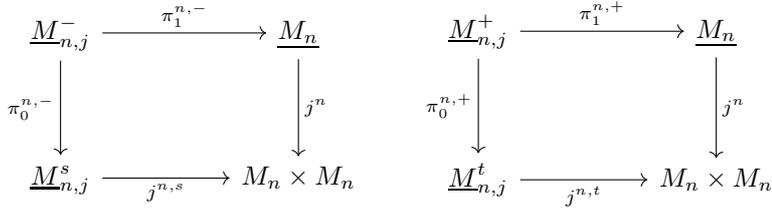
\begin{tikzcd}
\underline{M}^{-}_{n,j}\arrow[rr,"\pi^{n,-}_{1}"]\arrow[dd,"\pi^{n,-}_{0}"{left}]
&&\underline{M_n}\arrow[dd,"j^{n}"]\\\\
\underline{M}^{s}_{n,j}\arrow[rr,"j^{n,s}"{below}]&&M_{n}\times M_n
\end{tikzcd}\qquad\begin{tikzcd}
\underline{M}^{+}_{n,j}\arrow[rr,"\pi^{n,+}_{1}"]\arrow[dd,"\pi^{n,+}_{0}"{left}]
&&\underline{M_n}\arrow[dd,"j^n"]\\\\
\underline{M}^{t}_{n,j}\arrow[rr,"j^{n,t}"{below}]&&M_{n}\times M_n
\end{tikzcd}

\item We denote $k^-=j^{n,s}\circ \pi^{n,-}_{0}=j^{n}\circ \pi^{n,-}_{1}$ et 
$k^+=j^{n,t}\circ \pi^{n,+}_{0}=j^{n}\circ \pi^{n,+}_{1}$, and in this 
case we obtain the following commutative diagrams :

\begin{tikzcd}
&&\underline{M}^{-}_{n,j}\arrow[ddddll,"q^{n,-}_{0}"{left}]
\arrow[ddddrr,"q^{n,-}_{1}"]\arrow[ddd,dotted,"k^-"]\\\\\\
&&M_{n}\times M_n\arrow[dll,"p^n_0"]\arrow[drr,"p^n_1"{below}]\\
M_n&&&&M_n\end{tikzcd}\qquad\begin{tikzcd}
&&\underline{M}^{+}_{n,j}\arrow[ddddll,"q^{n,+}_{0}"{left}]
\arrow[ddddrr,"q^{n,+}_{1}"]\arrow[ddd,dotted,"k+"]\\\\\\
&&M_{n}\times M_n\arrow[dll,"p^n_0"]\arrow[drr,"p^n_1"{below}]\\
M_n&&&&M_n\end{tikzcd}
\end{itemize}

\item And we obtain the following commutative diagrams which give
the definition of the $s\times s$ for sets $\underline{M}^{-}_{n,j}$

\begin{itemize}
\item If $1\leq i<j\leq n$
\begin{itemize}

\item for sources

\begin{tikzcd}
\underline{M}^{-}_{n,j}\arrow[rrr,"\pi^{n,-}_{1}"]
\arrow[dd,"\pi^{n,-}_{0}"{left}]
\arrow[rrd,dotted,"s^{n}_{n-1,i}
\underset{-}\times s^{n}_{n-1,i}"]&&&\underline{M_n}
\arrow[rrd,"s^{n}_{n-1,i}
\underset{\pi}\times s^{n}_{n-1,i}"]\\
&&\underline{M}^{-}_{n-1,j}
\arrow[dd,"\pi^{n-1,-}_{0}"{left}]\arrow[rrr,"
\pi^{n-1,-}_{1}"]&&&M_{n-1}\arrow[dd,"j^{n-1}"]\\
\underline{M}^{s}_{n,j}\arrow[rrd,"s^{n}_{n-1,i}
\underset{s}\times s^{n}_{n-1,i}"{below}]\\
&&\underline{M}^{s}_{n-1,j}\arrow[rrr,"j^{n-1,s}"{below}]&&&
M_{n-1}\times M_{n-1}
\end{tikzcd}

\item for targets

\begin{tikzcd}
\underline{M}^{-}_{n,j}\arrow[rrr,"\pi^{n,-}_{1}"]
\arrow[dd,"\pi^{n,-}_{0}"{left}]
\arrow[rrd,dotted,"t^{n}_{n-1,i}
\underset{-}\times t^{n}_{n-1,i}"]&&&\underline{M_n}
\arrow[rrd,"t^{n}_{n-1,i}
\underset{\pi}\times t^{n}_{n-1,i}"]\\
&&\underline{M}^{-}_{n-1,j}
\arrow[dd,"\pi^{n-1,-}_{0}"{left}]\arrow[rrr,"
\pi^{n-1,-}_{1}"]&&&M_{n-1}\arrow[dd,"j^{n-1}"]\\
\underline{M}^{s}_{n,j}\arrow[rrd,"t^{n}_{n-1,i}
\underset{s}\times t^{n}_{n-1,i}"{below}]\\
&&\underline{M}^{s}_{n-1,j}\arrow[rrr,"j^{n-1,s}"{below}]&&&
M_{n-1}\times M_{n-1}
\end{tikzcd}

\end{itemize}

\item If $1\leq j<i\leq n+1$
\begin{itemize}

\item for sources

\begin{tikzcd}
\underline{M}^{-}_{n,j}\arrow[rrr,"\pi^{n,-}_{1}"]
\arrow[dd,"\pi^{n,-}_{0}"{left}]
\arrow[rrd,dotted,"s^{n}_{n-1,i-1}
\underset{-}\times s^{n}_{n-1,i-1}"]&&&\underline{M_n}
\arrow[rrd,"s^{n}_{n-1,i-1}
\underset{\pi}\times s^{n}_{n-1,i-1}"]\\
&&\underline{M}^{-}_{n-1,j}
\arrow[dd,"\pi^{n-1,-}_{0}"{left}]\arrow[rrr,"
\pi^{n-1,-}_{1}"]&&&M_{n-1}\arrow[dd,"j^{n-1}"]\\
\underline{M}^{s}_{n,j}\arrow[rrd,"s^{n}_{n-1,i}
\underset{s}\times s^{n}_{n-1,i}"{below}]\\
&&\underline{M}^{s}_{n-1,j}\arrow[rrr,"j^{n-1,s}"{below}]&&&
M_{n-1}\times M_{n-1}
\end{tikzcd}

\item for targets

\begin{tikzcd}
\underline{M}^{-}_{n,j}\arrow[rrr,"\pi^{n,-}_{1}"]
\arrow[dd,"\pi^{n,-}_{0}"{left}]
\arrow[rrd,dotted,"t^{n}_{n-1,i-1}
\underset{-}\times t^{n}_{n-1,i-1}"]&&&\underline{M_n}
\arrow[rrd,"t^{n}_{n-1,i-1}
\underset{\pi}\times t^{n}_{n-1,i-1}"]\\
&&\underline{M}^{-}_{n-1,j}
\arrow[dd,"\pi^{n-1,-}_{0}"{left}]\arrow[rrr,"
\pi^{n-1,-}_{1}"]&&&M_{n-1}\arrow[dd,"j^{n-1}"]\\
\underline{M}^{s}_{n,j}\arrow[rrd,"t^{n}_{n-1,i}
\underset{s}\times t^{n}_{n-1,i}"{below}]\\
&&\underline{M}^{s}_{n-1,j}\arrow[rrr,"j^{n-1,s}"{below}]&&&
M_{n-1}\times M_{n-1}
\end{tikzcd}

\end{itemize}
\end{itemize}

\item And we obtain the following commutative diagrams which give
the definition of the $s\times s$ for sets $\underline{M}^{+}_{n,j}$

\begin{itemize}

\item If $1\leq i<j\leq n$

\begin{itemize}

\item for sources

\begin{tikzcd}
\underline{M}^{+}_{n,j}\arrow[rrr,"\pi^{n,+}_{1}"]
\arrow[dd,"\pi^{n,+}_{0}"{left}]
\arrow[rrd,dotted,"s^{n}_{n-1,i}
\underset{+}\times s^{n}_{n-1,i}"]&&&\underline{M_n}
\arrow[rrd,"s^{n}_{n-1,i}
\underset{\pi}\times s^{n}_{n-1,i}"]\\
&&\underline{M}^{+}_{n-1,j}
\arrow[dd,"\pi^{n-1,+}_{0}"{left}]\arrow[rrr,"
\pi^{n-1,+}_{1}"]&&&M_{n-1}\arrow[dd,"j^{n-1}"]\\
\underline{M}^{t}_{n,j}\arrow[rrd,"s^{n}_{n-1,i}
\underset{t}\times s^{n}_{n-1,i}"{below}]\\
&&\underline{M}^{t}_{n-1,j}\arrow[rrr,"j^{n-1,t}"{below}]&&&
M_{n-1}\times M_{n-1}
\end{tikzcd}

\item for targets

\begin{tikzcd}
\underline{M}^{+}_{n,j}\arrow[rrr,"\pi^{n,+}_{1}"]
\arrow[dd,"\pi^{n,+}_{0}"{left}]
\arrow[rrd,dotted,"t^{n}_{n-1,i}
\underset{+}\times t^{n}_{n-1,i}"]&&&\underline{M_n}
\arrow[rrd,"t^{n}_{n-1,i}
\underset{\pi}\times t^{n}_{n-1,i}"]\\
&&\underline{M}^{+}_{n-1,j}
\arrow[dd,"\pi^{n-1,+}_{0}"{left}]\arrow[rrr,"
\pi^{n-1,+}_{1}"]&&&M_{n-1}\arrow[dd,"j^{n-1}"]\\
\underline{M}^{t}_{n,j}\arrow[rrd,"t^{n}_{n-1,i}
\underset{t}\times t^{n}_{n-1,i}"{below}]\\
&&\underline{M}^{t}_{n-1,j}\arrow[rrr,"j^{n-1,t}"{below}]&&&
M_{n-1}\times M_{n-1}
\end{tikzcd}
\end{itemize}

\item If $1\leq j<i\leq n+1$

\begin{itemize}

\item for sources

\begin{tikzcd}
\underline{M}^{+}_{n,j}\arrow[rrr,"\pi^{n,+}_{1}"]
\arrow[dd,"\pi^{n,+}_{0}"{left}]
\arrow[rrd,dotted,"s^{n}_{n-1,i-1}
\underset{+}\times s^{n}_{n-1,i-1}"]&&&\underline{M_n}
\arrow[rrd,"s^{n}_{n-1,i-1}
\underset{\pi}\times s^{n}_{n-1,i-1}"]\\
&&\underline{M}^{+}_{n-1,j}
\arrow[dd,"\pi^{n-1,+}_{0}"{left}]\arrow[rrr,"
\pi^{n-1,+}_{1}"]&&&M_{n-1}\arrow[dd,"j^{n-1}"]\\
\underline{M}^{t}_{n,j}\arrow[rrd,"s^{n}_{n-1,i-1}
\underset{t}\times s^{n}_{n-1,i-1}"{below}]\\
&&\underline{M}^{t}_{n-1,j}\arrow[rrr,"j^{n-1,t}"{below}]&&&
M_{n-1}\times M_{n-1}
\end{tikzcd}

\item for targets

\begin{tikzcd}
\underline{M}^{+}_{n,j}\arrow[rrr,"\pi^{n,+}_{1}"]
\arrow[dd,"\pi^{n,+}_{0}"{left}]
\arrow[rrd,dotted,"t^{n}_{n-1,i-1}
\underset{+}\times t^{n}_{n-1,i-1}"]&&&\underline{M_n}
\arrow[rrd,"t^{n}_{n-1,i-1}
\underset{\pi}\times t^{n}_{n-1,i-1}"]\\
&&\underline{M}^{+}_{n-1,j}
\arrow[dd,"\pi^{n-1,+}_{0}"{left}]\arrow[rrr,"
\pi^{n-1,+}_{1}"]&&&M_{n-1}\arrow[dd,"j^{n-1}"]\\
\underline{M}^{t}_{n,j}\arrow[rrd,"t^{n}_{n-1,i-1}
\underset{t}\times t^{n}_{n-1,i-1}"{below}]\\
&&\underline{M}^{t}_{n-1,j}\arrow[rrr,"j^{n-1,t}"{below}]&&&
M_{n-1}\times M_{n-1}
\end{tikzcd}
\end{itemize}

\end{itemize}

\item We consider the following commutative diagrams :

\begin{itemize}

\item If $1\leq i<j\leq n+1$

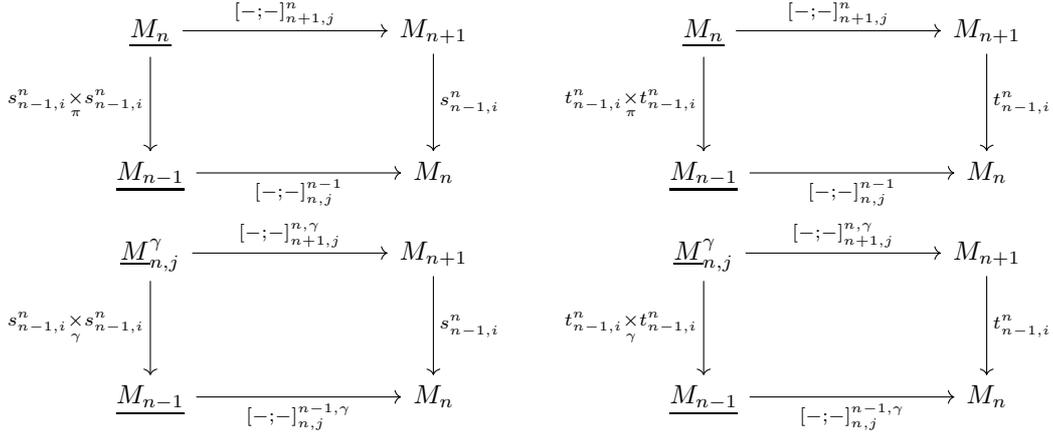
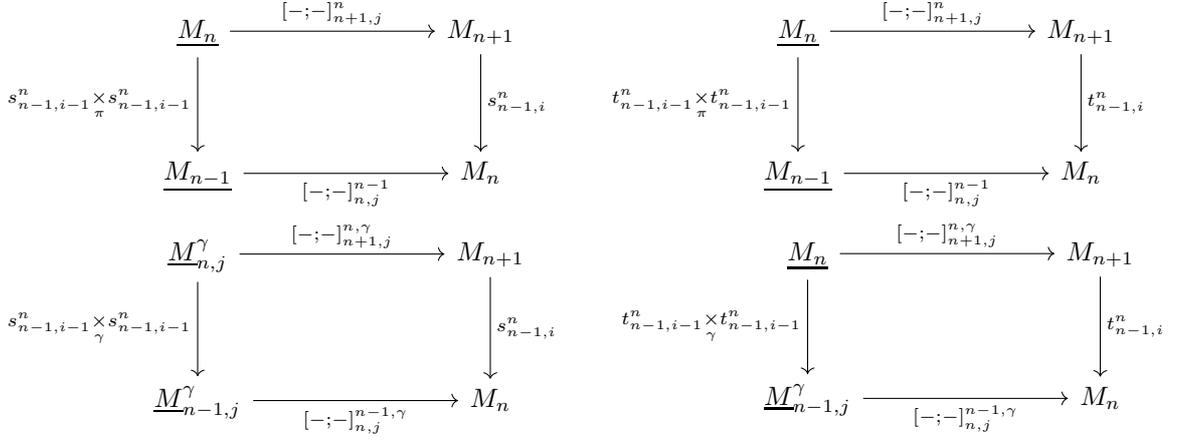
\begin{tikzcd}
\underline{M_n}\arrow[dd,"s^{n}_{n-1,i}\underset{\pi}\times s^{n}_{n-1,i}"{left}]
\arrow[rrr,"{[-;-]}^{n}_{n+1,j}"]&&&M_{n+1}
\arrow[dd,"s^{n}_{n-1,i}"]\\\\
\underline{M_{n-1}}\arrow[rrr,"{[-;-]}^{n-1}_{n,j}"{below}]&&&M_n
\end{tikzcd}\qquad\begin{tikzcd}
\underline{M_n}\arrow[dd,"t^{n}_{n-1,i}\underset{\pi}\times t^{n}_{n-1,i}"{left}]
\arrow[rrr,"{[-;-]}^{n}_{n+1,j}"]&&&M_{n+1}
\arrow[dd,"t^{n}_{n-1,i}"]\\\\
\underline{M_{n-1}}\arrow[rrr,"{[-;-]}^{n-1}_{n,j}"{below}]&&&M_n
\end{tikzcd}

\begin{tikzcd}
\underline{M}^{\gamma}_{n,j}\arrow[dd,"s^{n}_{n-1,i}\underset{\gamma}\times s^{n}_{n-1,i}"{left}]
\arrow[rrr,"{[-;-]}^{n,\gamma}_{n+1,j}"]&&&M_{n+1}
\arrow[dd,"s^{n}_{n-1,i}"]\\\\
\underline{M_{n-1}}\arrow[rrr,"{[-;-]}^{n-1,\gamma}_{n,j}"{below}]&&&M_n
\end{tikzcd}\qquad\begin{tikzcd}
\underline{M}^{\gamma}_{n,j}\arrow[dd,"t^{n}_{n-1,i}\underset{\gamma}\times t^{n}_{n-1,i}"{left}]
\arrow[rrr,"{[-;-]}^{n,\gamma}_{n+1,j}"]&&&M_{n+1}
\arrow[dd,"t^{n}_{n-1,i}"]\\\\
\underline{M_{n-1}}\arrow[rrr,"{[-;-]}^{n-1,\gamma}_{n,j}"{below}]&&&M_n
\end{tikzcd}

\item If $1\leq j<i\leq n+1$

\begin{tikzcd}
\underline{M_n}\arrow[dd,"s^{n}_{n-1,i-1}\underset{\pi}\times s^{n}_{n-1,i-1}"{left}]
\arrow[rrr,"{[-;-]}^{n}_{n+1,j}"]&&&M_{n+1}
\arrow[dd,"s^{n}_{n-1,i}"]\\\\
\underline{M_{n-1}}\arrow[rrr,"{[-;-]}^{n-1}_{n,j}"{below}]&&&M_n
\end{tikzcd}\qquad\begin{tikzcd}
\underline{M_n}\arrow[dd,"t^{n}_{n-1,i-1}\underset{\pi}\times t^{n}_{n-1,i-1}"{left}]
\arrow[rrr,"{[-;-]}^{n}_{n+1,j}"]&&&M_{n+1}
\arrow[dd,"t^{n}_{n-1,i}"]\\\\
\underline{M_{n-1}}\arrow[rrr,"{[-;-]}^{n-1}_{n,j}"{below}]&&&M_n
\end{tikzcd}

\begin{tikzcd}
\underline{M}^{\gamma}_{n,j}\arrow[dd,"s^{n}_{n-1,i-1}\underset{\gamma}\times s^{n}_{n-1,i-1}"{left}]
\arrow[rrr,"{[-;-]}^{n,\gamma}_{n+1,j}"]&&&M_{n+1}
\arrow[dd,"s^{n}_{n-1,i}"]\\\\
\underline{M}^{\gamma}_{n-1,j}\arrow[rrr,"{[-;-]}^{n-1,\gamma}_{n,j}"{below}]&&&M_n
\end{tikzcd}\qquad\begin{tikzcd}
\underline{M_n}\arrow[dd,"t^{n}_{n-1,i-1}\underset{\gamma}\times t^{n}_{n-1,i-1}"{left}]
\arrow[rrr,"{[-;-]}^{n,\gamma}_{n+1,j}"]&&&M_{n+1}
\arrow[dd,"t^{n}_{n-1,i}"]\\\\
\underline{M}^{\gamma}_{n-1,j}\arrow[rrr,"{[-;-]}^{n-1,\gamma}_{n,j}"{below}]&&&M_n
\end{tikzcd}

\item If $i=j$
\begin{itemize}

\item for the operations ${[-;-]}^{n}_{n+1,j}$

\begin{tikzcd}
\underline{M_n}
\arrow[rrdd,"\pi^{n}_{0}"{left}]\arrow[rrrr,"{[-;-]}^{n}_{n+1,j}"]&&&&
M_{n+1}\arrow[lldd,"s^{n+1}_{n,j}"]\\\\
&&M_n
\end{tikzcd}\qquad\begin{tikzcd}
\underline{M_n}
\arrow[rrdd,"\pi^{n}_{1}"{left}]
\arrow[rrrr,"{[-;-]}^{n}_{n+1,j}"]&&&&
M_{n+1}\arrow[lldd,"t^{n+1}_{n,j}"]\\\\
&&M_n
\end{tikzcd}

\item for the operations ${[-;-]}^{n,\gamma}_{n+1,j}$

\begin{tikzcd}
\underline{M}^{\gamma}_{n,j}
\arrow[rrdd,"q^{n,\gamma}_{0}"{left}]
\arrow[rrrr,"{[-;-]}^{n,\gamma}_{n+1,j}"]&&&&
M_{n+1}\arrow[lldd,"s^{n+1}_{n,j}"]\\\\
&&M_n
\end{tikzcd}\qquad\begin{tikzcd}
\underline{M}^{\gamma}_{n,j}
\arrow[rrdd,"q^{n,\gamma}_{1}"{left}]
\arrow[rrrr,"{[-;-]}^{n,\gamma}_{n+1,j}"]&&&&
M_{n+1}\arrow[lldd,"t^{n+1}_{n,j}"]\\\\
&&M_n
\end{tikzcd}

\item Other possible diagram for this definition :

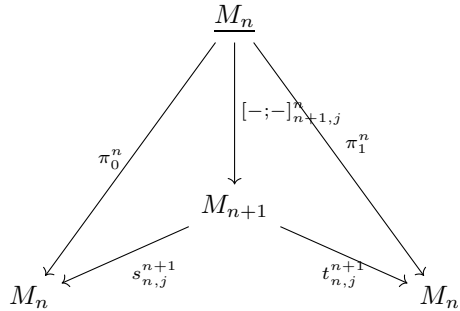
\begin{tikzcd}
&&\underline{M_n}\arrow[ddddll,"\pi^{n}_{0}"{left}]\arrow[ddddrr,"\pi^{n}_{1}"]
\arrow[ddd,"{[-;-]}^{n}_{n+1,j}"]\\\\\\
&&M_{n+1}\arrow[dll,"s^{n+1}_{n,j}"]\arrow[drr,"t^{n+1}_{n,j}"{below}]\\
M_n&&&&M_n\end{tikzcd}

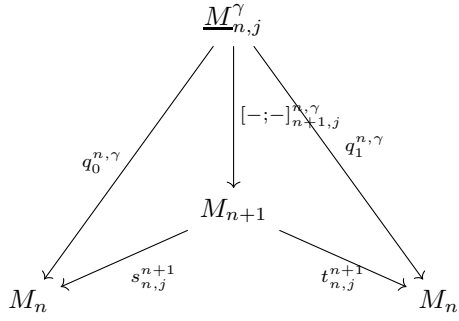
\begin{tikzcd}
&&\underline{M}^{\gamma}_{n,j}
\arrow[ddddll,"q^{n,\gamma}_{0}"{left}]
\arrow[ddddrr,"q^{n,\gamma}_{1}"]
\arrow[ddd,"{[-;-]}^{n,\gamma}_{n+1,j}"]\\\\\\
&&M_{n+1}\arrow[dll,"s^{n+1}_{n,j}"]\arrow[drr,"t^{n+1}_{n,j}"{below}]\\
M_n&&&&M_n\end{tikzcd}
\end{itemize}

\end{itemize}

\item \underline{Commutative diagrams for axioms} :

\begin{itemize}

\item For operations ${[-;-]}^{n}_{n+1,j}$

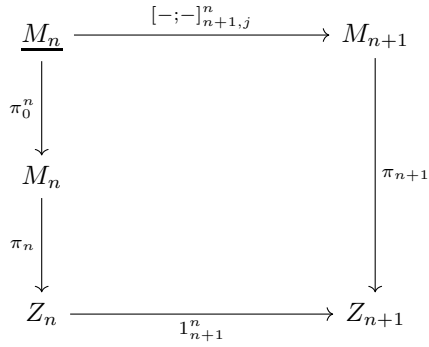
\[\begin{tikzcd}
\underline{M_n}\arrow[dd,"\pi^{n}_{0}"{left}]\arrow[rrrr,"{[-;-]}^{n}_{n+1,j}"]&&&&
M_{n+1}\arrow[dddd,"\pi_{n+1}"]\\\\
M_n\arrow[dd,"\pi_n"{left}]\\\\
Z_n\arrow[rrrr,"1^{n}_{n+1}"{below}]&&&&Z_{n+1}
\end{tikzcd}\qquad
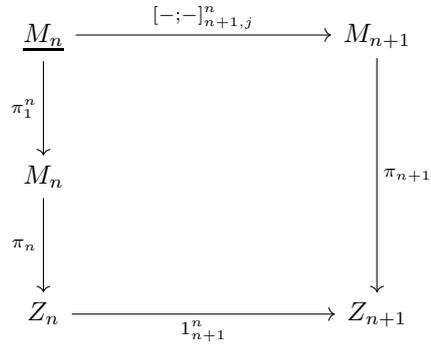
\begin{tikzcd}
\underline{M_n}\arrow[dd,"\pi^{n}_{1}"{left}]\arrow[rrrr,"{[-;-]}^{n}_{n+1,j}"]&&&&M_{n+1}
\arrow[dddd,"\pi_{n+1}"]\\\\
M_n\arrow[dd,"\pi_n"{left}]\\\\
Z_n\arrow[rrrr,"1^{n}_{n+1}"{below}]&&&&Z_{n+1}
\end{tikzcd}\] 

\item For operations ${[-;-]}^{n,\gamma}_{n+1,j}$

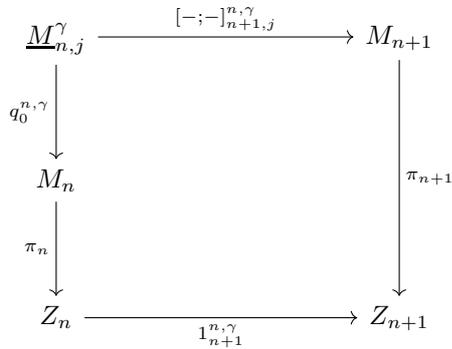
\[\begin{tikzcd}
\underline{M}^{\gamma}_{n,j}
\arrow[dd,"q^{n,\gamma}_{0}"{left}]
\arrow[rrrr,"{[-;-]}^{n,\gamma}_{n+1,j}"]&&&&
M_{n+1}\arrow[dddd,"\pi_{n+1}"]\\\\
M_n\arrow[dd,"\pi_n"{left}]\\\\
Z_n\arrow[rrrr,"1^{n,\gamma}_{n+1}"{below}]&&&&Z_{n+1}
\end{tikzcd}\qquad
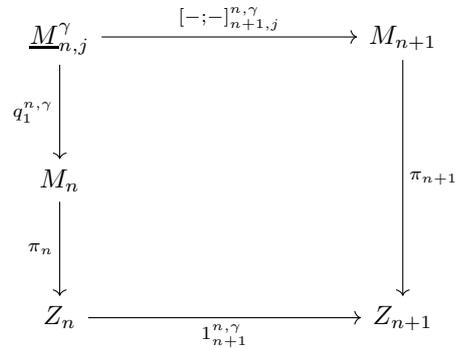
\begin{tikzcd}
\underline{M}^{\gamma}_{n,j}
\arrow[dd,"q^{n,\gamma}_{1}"{left}]
\arrow[rrrr,"{[-;-]}^{n,\gamma}_{n+1,j}"]&&&&M_{n+1}
\arrow[dddd,"\pi_{n+1}"]\\\\
M_n\arrow[dd,"\pi_n"{left}]\\\\
Z_n\arrow[rrrr,"1^{n,\gamma}_{n+1}"{below}]&&&&Z_{n+1}
\end{tikzcd}\] 

\end{itemize}

\item \underline{Comeback to the first generation}

We build the diagonal with the following commutative diagrams :

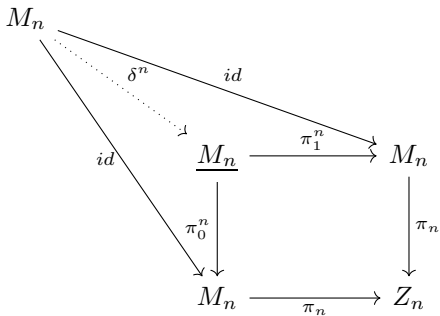
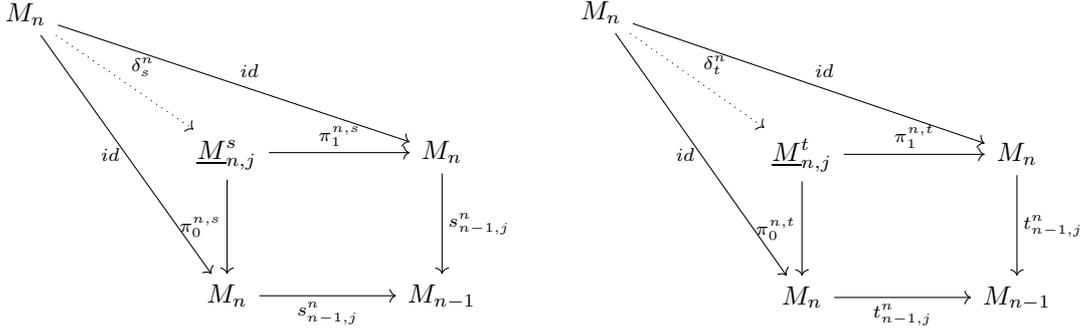
\begin{tikzcd}
M_{n}\arrow[ddddrr,"id"{left}]\arrow[ddrr,dotted,"\delta^n"]\arrow[ddrrrr,"id"]\\\\
&&\underline{M_n}\arrow[rr,"\pi^{n}_{1}"]\arrow[dd,"\pi^{n}_{0}"{left}]
&&M_{n}\arrow[dd,"\pi_n"]\\\\
&&M_{n}\arrow[rr,"\pi_n"{below}]&&Z_{n}
\end{tikzcd}

\begin{tikzcd}
M_{n}\arrow[ddddrr,"id"{left}]\arrow[ddrr,dotted,"\delta^n_{s}"]\arrow[ddrrrr,"id"]\\\\
&&\underline{M}^{s}_{n,j}\arrow[rr,"\pi^{n,s}_{1}"]\arrow[dd,"\pi^{n,s}_{0}"{left}]
&&M_{n}\arrow[dd,"s^{n}_{n-1,j}"]\\\\
&&M_{n}\arrow[rr,"s^{n}_{n-1,j}"{below}]&&M_{n-1}
\end{tikzcd}\qquad \begin{tikzcd}
M_{n}\arrow[ddddrr,"id"{left}]\arrow[ddrr,dotted,"\delta^n_{t}"]\arrow[ddrrrr,"id"]\\\\
&&\underline{M}^{t}_{n,j}\arrow[rr,"\pi^{n,t}_{1}"]\arrow[dd,"\pi^{n,t}_{0}"{left}]
&&M_{n}\arrow[dd,"t^{n}_{n-1,j}"]\\\\
&&M_{n}\arrow[rr,"t^{n}_{n-1,j}"{below}]&&M_{n-1}
\end{tikzcd}

\item \underline{Comeback to the second generation}

The previous diagrams generate the following commutative diagrams :

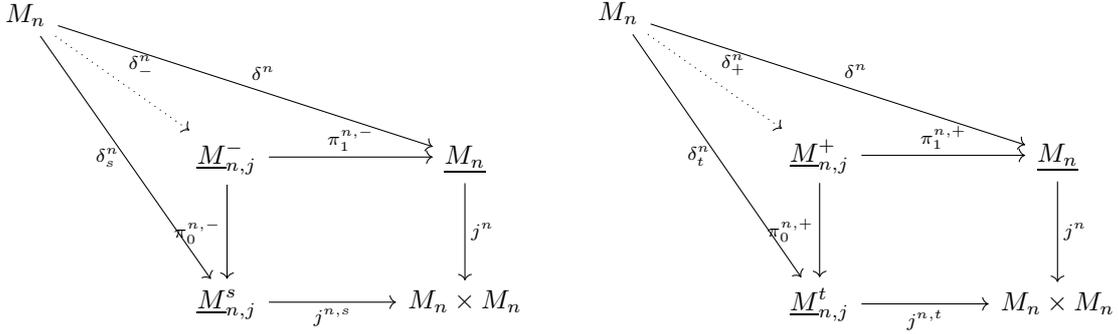
\begin{tikzcd}
M_{n}\arrow[ddddrr,"\delta^n_s"{left}]\arrow[ddrr,dotted,"\delta^n_{-}"]\arrow[ddrrrr,"\delta^n"]\\\\
&&\underline{M}^{-}_{n,j}\arrow[rr,"\pi^{n,-}_{1}"]\arrow[dd,"\pi^{n,-}_{0}"{left}]
&&\underline{M_n}\arrow[dd,"j^n"]\\\\
&&\underline{M}^{s}_{n,j}\arrow[rr,"j^{n,s}"{below}]&&
M_{n}\times M_{n}
\end{tikzcd}\qquad\begin{tikzcd}
M_{n}\arrow[ddddrr,"\delta^n_t"{left}]\arrow[ddrr,dotted,"\delta^n_{+}"]\arrow[ddrrrr,"\delta^n"]\\\\
&&\underline{M}^{+}_{n,j}\arrow[rr,"\pi^{n,+}_{1}"]\arrow[dd,"\pi^{n,+}_{0}"{left}]
&&\underline{M_n}\arrow[dd,"j^n"]\\\\
&&\underline{M}^{t}_{n,j}\arrow[rr,"j^{n,t}"{below}]&&
M_{n}\times M_{n}
\end{tikzcd} 

\item Then we obtain the following commutative diagrams of first generation for axioms of reflexivity of the operations ${[-;-]}^{n}_{n+1,j}$ :

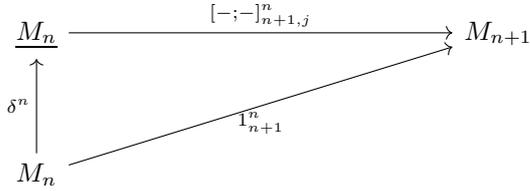
\begin{tikzcd}
\underline{M_n}
\arrow[rrrrrr,"{[-;-]}^{n}_{n+1,j}"]&&&&&&M_{n+1}\\\\
M_n\arrow[uu,"\delta^n"]\arrow[uurrrrrr,"1^{n}_{n+1}"{below}]
\end{tikzcd}

\item And the following commutative diagrams of second generation for axioms of reflexivity of the operations ${[-;-]}^{n,\gamma}_{n+1,j}$ :

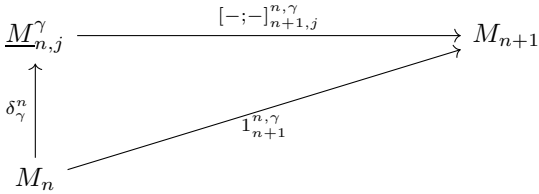
\begin{tikzcd}
\underline{M}^{\gamma}_{n,j}
\arrow[rrrrrr,"{[-;-]}^{n,\gamma}_{n+1,j}"]&&&&&&M_{n+1}\\\\
M_n\arrow[uu,"\delta^n_\gamma"]\arrow[uurrrrrr,"1^{n,\gamma}_{n+1}"{below}]
\end{tikzcd}
\end{description}

\end{proof}

\begin{definition}
\label{weak-def}
  Cubical weak $\infty$-categories are algebras for the monad $\mathbb{W}$ above.
\end{definition}

\label{example-coherence}
Let us show with a simple example how cubical weak $\infty$-categories provide a richer 
weakened structure than the one of globular weak $\infty$-categories : for simplicity
we show it inside an object $\mathbb{E}$ of $\Et$ :
\[ \mathbb{E}=(M,C,\pi,([-;-]^{n}_{n+1,j})_{n\in\mathbb{N};j\in\{1,...,n+1\}},
  ([-;-]^{n,\gamma}_{n+1,j})_{n\in\mathbb{N};j\in\{1,...,n\};\gamma\in\{-,+\}})\] 
Consider the following string in $M(1)$
$$\begin{tikzcd}
a\arrow[rr,"f"]&&b\arrow[rr,"g"]&&c\arrow[rr,"h"]&&d
\end{tikzcd}$$

and take the $1$-cells $x=(h\circ g)\circ f$ and $y=h\circ (g\circ f)$. Because 
these cells belong to $\underline{M^-_{1,0}}\cap\underline{M^+_{1,0}}$ we get the 
following $2$-cells 

   \begin{tikzcd}
a\arrow[rrrddd,phantom,"{[x,y]^1_{2,1}}"] \arrow[ddd,"1_a"{left}]\arrow[rrr,"(h\circ g)\circ f"] &&&d\arrow[ddd,"1_d"] \\\\\\
a\arrow[rrr,"h\circ (g\circ f)"{below}] &&&d
\end{tikzcd}\qquad\begin{tikzcd}
a\arrow[rrrddd,phantom,"{[x,y]^1_{2,2}}"] \arrow[ddd,"(h\circ g)\circ f"{left}]\arrow[rrr,"1_a"] &&&a
\arrow[ddd,"h\circ (g\circ f)"] \\\\\\
d\arrow[rrr,"1_d"{below}] &&&d
\end{tikzcd}

\begin{tikzcd}
a\arrow[rrrddd,phantom,"{[x,y]^{1,-}_{2,1}}"] \arrow[ddd,"h\circ (g\circ f)"{left}]\arrow[rrr,"(h\circ g)\circ f"] &&&d\arrow[ddd,"1_d"] \\\\\\
d\arrow[rrr,"1_d"{below}] &&&d
\end{tikzcd}\qquad\begin{tikzcd}
a\arrow[rrrddd,phantom,"{[x,y]^{1,+}_{2,1}}"] \arrow[ddd,"1_a"{left}]\arrow[rrr,"1_a"] &&&d
\arrow[ddd,"h\circ (g\circ f)"] \\\\\\
d\arrow[rrr,"(h\circ g)\circ f"{below}] &&&d
\end{tikzcd}

\begin{remark}
We could have defined cubical categorical stretchings slightly differently than those just above by means of using
just the operations $([-;-]^{n}_{n+1,j})_{n\in\mathbb{N};j\in\{1,...,n\}}$ to weakened the structure of cubical strict 
$\infty$-categories. Denote by $\I\mathbb{C}\mathbb{E}\text{tC}'$ the category of these slightly impoverished 
structures. We also have a forgetful functor 
\begin{tikzcd}
\I\mathbb{C}\mathbb{E}\text{tC}'\arrow[rr,"U'"]&&\Cu\E
\end{tikzcd}
which is right adjoint and which produce an other monad $\mathbb{W}'=(W',\eta',\nu')$ which algebras could be 
also considered as enough good models of cubical weak $\infty$-categories. Also we have an evident forgetful functor
\begin{tikzcd}
\I\mathbb{C}\mathbb{E}\text{tC}\arrow[rr,"U"]&&\I\mathbb{C}\mathbb{E}\text{tC}'
\end{tikzcd}
which is right adjoint and which produce a functor
\begin{tikzcd}
\mathbb{W}\text{-}\mathbb{A}\text{lg}\arrow[rr,"U"]&&\mathbb{W}\text{-}\mathbb{A}\text{lg}'
\end{tikzcd}
which shows that the models that we have chosen for our article are also models for these impoverished 
structures. And our choice to add operations $([-;-]^{n,\gamma}_{n+1,j})_{n\in\mathbb{N};j\in\{1,...,n\};\gamma\in\{-,+\}}$
to get our models of cubical weak $\infty$-categories 
is similar to the one who choses cubical strict $\infty$-categories with connections instead of considering it without 
connections. We believe that our choice gives not only more refined models than those of 
the category $\mathbb{W}\text{-}\mathbb{A}\text{lg}'$ but also is in fact really necessary for a good approach 
of cubical weak $\infty$-categories, where formalism of connections are implicit and used in our weakened 
structures.
\end{remark}

\begin{remark}

In \cite{brown} the authors have proved that the category of cubical strict $\infty$-categories is 
equivalent to the category of globular strict $\infty$-categories. We suspect that such phenomena 
is still right in the world of weak models. Let us be more precise about what we are saying : denote by 
$\mathbb{P}$ the Penon's monad on the category of globular sets
(see \cite{penon}) which algebras are particularly nice models of globular weak $\infty$-categories
(see for example \cite{bat-pen,cheng-makkai}). It is suspected (see \cite{remy-thesis})
that its category of algebras $\mathbb{P}\text{-}\mathbb{A}\text{lg}$ can be equipped with 
a \textit{canonical Quillen model structure} similar to the one build in \cite{folk} for strict globular 
$\infty$-categories, and we also suspect that $\mathbb{W}\text{-}\mathbb{A}\text{lg}$ can
be equipped with such canonical Quillen model structure. Thus a weak version of the article 
\cite{brown} should be that the category $\mathbb{P}\text{-}\mathbb{A}\text{lg}$ is Quillen equivalent 
to $\mathbb{W}\text{-}\mathbb{A}\text{lg}$ when these categories are equipped with their
canonical Quillen model structure.

\end{remark}

\subsection{Magmatic properties of cubical weak $\infty$-categories}
\label{magma}
Consider a cubical weak $\infty$-category 
\begin{tikzcd}
W(C)\arrow[r,"v"]&C
\end{tikzcd}. In this monadic presentation, $W(C)$ has to be thought as the free
 cubical weak $\infty$-category representing the underlying syntax with which all algebras
 with underlying cubical set $C$ are interpretations of it via their
 morphisms structural. For example here $v$ is the morphism structural which plays
 the role of interpreting in $C$ the "syntax" $W(C)$, and thus put on $C$ a structure
 of $\mathbb{W}$-algebra. We shall distinguished well notations of operations inside $W(C)$
 and inside $C$ in order to separate the syntactic part from the model part of our
 algebras. For example the operations of compositions shall be denoted $\circ^{n}_{j}$
  in the models, whereas we shall use the notation $\star^{n}_{j}$ instead when we work 
 in the free models. The reflexions are denoted $\iota^{n}_{n+1,j}$ in the models and
 $1^{n}_{n+1,j}$ in the free models. The connections are denoted $\iota^{n,\alpha}_{n+1,j}$ 
 in the models and $1^{n,\alpha}_{n+1,j}$ in the free models. Definitions of operations
 for models use those for free models and the interpretative nature of \begin{tikzcd}
W(C)\arrow[r,"v"]&C
\end{tikzcd} emerges then with the axiomatic of algebras for monads : for example 
we consider first the following definition of operations on $C$ :
\begin{enumerate}[(i)]
\item If $a$, $b\in C(n)$ are such that $s^{n}_{j}(b)=t^{n}_{j}(a)$ for $j\in\{1,...,n\}$
then we put $a\circ^{n}_{j}b=v_n(\eta(a)\star^{n}_{j}\eta(b))$ 

\item If $a\in C(n)$ is an $n$-cell then we put 
$\iota^{n}_{n+1,j}(a)=v_{n+1}(1^{n}_{n+1,j}(\eta(a)))$, 
$n\in\mathbb{N}$, $j\in\llbracket 1,n+1 \rrbracket$

\item If $a\in C(n)$ is an $n$-cell then we put 
$\iota^{n,\gamma}_{n+1,j}(a)=v_{n+1}(1^{n,\gamma}_{n+1,j}(\eta(a)))$, 
$n\geq 1$, $j\in\llbracket 1,n \rrbracket$, $\gamma\in\{-,+\}$
\end{enumerate}
Thus $v$ puts on $C$ a cubical $\infty$-magma structure and its interpretative nature  
is primarily expressed by the fact that it is a morphism of cubical $\infty$-magmas between
the free cubical $\infty$-magma $W(C)$ and this cubical $\infty$-magma on $C$. It
is the axioms of algebras which show us such important fact : actually we need to show that 
$v(a\star^{n}_{j}b)=v(a)\circ^{n}_{j}v(b)$, $v(1^{n}_{n+1,j}(a))=\iota^{n}_{n+1,j}(v(a))$, 
$v(1^{n,\gamma}_{n+1,j}(a))=\iota^{n,\gamma}_{n+1,j}(v(a))$. Let us show the first
equality : 
\begin{align*}
v(a)\circ^{n}_{j}v(b) &= v(\eta(v(a))\star^{n}_{j}\eta(v(b))) \\
        &= v(W(v)\eta_{W(C)}(a)\star^{n}_{j} W(v)\eta_{W(C)}(b)) \\
        &=v(W(v)(\eta_{W(C)}(a)\star^{n}_{j}\eta_{W(C)}(b)))\\
        &=v(\nu(C)(\eta_{W(C)}(a)\star^{n}_{j}\eta_{W(C)}(b)))\\
        &=v(\nu(C)(\eta_{W(C)}(a))\star^{n}_{j}\nu(C)(\eta_{W(C)}(b)))\\
        &=v(a\star^{n}_{j}b)
\end{align*}
Other equalities are shown similarly. In \cite{penon} J.Penon called \textit{magmatic} such properties of algebras. In particular these shall be useful for concrete 
computations in any $\mathbb{W}$-algebras.

\subsection{Computations for low dimensions}
\label{computations}

\begin{definition}
Consider a reflexive cubical set $C\in\Cr\E$.
It has dimension $p\in\mathbb{N}$ for reflexions if all its $q$-cells $x\in C(q)$ for which $q>p$ are
of the form $x=1^{q-1}_{q,j}(y)$ and if there is at least one $p$-cell which is not of this form. 
It has dimension $p\in\mathbb{N}$ for connections if all its $q$-cells $x\in C(q)$ for which $q>p$ are
of the form $x=1^{q-1,\gamma}_{q,j}(y)$ and if there is at least one $p$-cell which is not of this form. 
It has dimension $p\in\mathbb{N}$, if it has dimension $p\in\mathbb{N}$ for reflexions and 
connections.
\end{definition}

\begin{definition}
Consider a $\mathbb{W}$-algebra $(C,v)$. It has dimension $p\in\mathbb{N}$ for reflexion if its
underlying reflexive set produced by its underlying $\infty$-magma structure (see \ref{magma}) has 
dimension $p\in\mathbb{N}$ for reflexion. It has dimension $p\in\mathbb{N}$ for connection if its 
underlying reflexive set produced by its underlying $\infty$-magma structure has dimension 
$p\in\mathbb{N}$ for connections. It has dimension $p\in\mathbb{N}$, if it has dimension $p\in\mathbb{N}$ 
for reflexions and connections.
\end{definition}

\begin{exercice} 
$2$-dimensional $\mathbb{W}$-algebras are cubical bicategories.
\end{exercice}

\section{Cubical weak $\infty$-groupoids}
\label{isostretchings}
The author doesn't know any work on cubical weak $\infty$-categories with kind of inverses involved. However such
work exists for the strict case in \cite{lucas}. In low dimensions, simplicial
methods have been used in \cite{bro2,bro1,bro-hig2,bro-hig1,bro-spen,dakin,salleh} to study it. Some applications of it to
homology have been considered in \cite{ashley,bro-hig3,bro-hig2,bro-hig1}, and other applications in algebraic topology
have also been carried out in \cite{bro-lod,gilbert,porter}. As we said in the beginning of this article our first scope is to provide these higher cubical notions with the perspective to carry on applications of these cubical higher structures to homological algebra, algebraic topology 
and computer sciences.

Models of cubical weak $\infty$-groupoids that we are going to define are algebras 
for a monad $\mathbb{W}^0$ on the category $\Cu\E$ of cubical sets. This monad is built with adapted
stretchings : the \textit{cubical $\infty$-groupoidal stretchings} \ref{isostretchings}, which themselves are built with
cubical strict $\infty$-groupoids with connections \ref{cub-m-strict} which has been characterized by Lucas in \cite{lucas}, and with 
cubical reflexive $(\infty,0)$-magmas \ref{isostretchings}. As in \ref{strict-infty} these cubical $\infty$-groupoidal stretchings
are tools which fill with cubical coherences cells their underlying cubical reflexive $(\infty,0)$-magmas, controlled by their underlying cubical strict $\infty$-groupoids.  

\subsection{Cubical $(\infty,0)$-sets.}
Cubical $(\infty,0)$-sets underly a new sketch
(see diagrams below) 
which we use in \ref{isostretchings} to define algebraic models of cubical weak $\infty$-groupoids.

Here we define cubical version of
the formalism developed in \cite{Cam} for globular $(\infty,0)$-sets. This formalism of this cubical world is very similar to its globular world analogue.

\label{reversible-omega-graphs}

Consider a cubical set $\mathcal{C}=(C_n,s^{n}_{n-1,j},t^{n}_{n-1,j})_{1\leq j\leq n}$. If $n\geq 1$ and 
$1\leq j\leq n$, then a $(n,j)$-reversor on it 
is given by a map $\xymatrix{C_{n}\ar[rr]^{j^{n}_{j}}&&C_{n}}$ such that the following two diagrams commute :
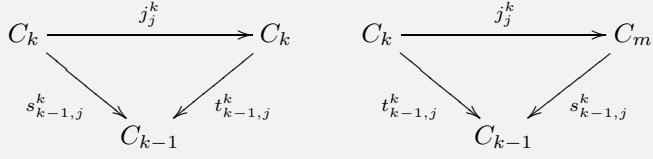
\[\xymatrix{C_{n}\ar[rr]^{j^{n}_{j}}\ar[rd]_{s^{n}_{n-1,j}}&&C_n\ar[ld]^{t^{n}_{n-1,j}}\\
&C_{n-1}}\qquad\xymatrix{C_{n}\ar[rr]^{j^{n}_{j}}\ar[rd]_{t^{n}_{n-1,j}}&&C_m\ar[ld]^{s^{n}_{n-1,j}}\\
&C_{n-1}}\]

If for each $n>0$ and for each $1\leq j\leq n$, there are such $(n,j)$-reversor $j^{n}_{j}$ on $\mathcal{C}$, then 
we say that $\mathcal{C}$ is a cubical $(\infty,0)$-set. The family of maps $(j^{n}_{j})_{n>0,1\leq j\leq n}$ for all
 $(n\in\mathbb{N}^*)$   is
called an $(\infty,0)$-structure and in that case we shall say that $\mathcal{C}$ is equipped with the $(\infty,0)$-structure 
$(j^{n}_{j})_{n>0,1\leq j\leq n}$. When we speak about such $(\infty,0)$-structure $(j^{n}_{j})_{n>0,1\leq j\leq n}$ on $\mathcal{C}$, 
it means that it is for all integers $n\in\mathbb{N}^*$ such that $C_n$ is non-empty. Seen as cubical $(\infty,0)$-set we denote it by
$\mathcal{C}=((C_n,s^{n}_{n-1,j},t^{n}_{n-1,j})_{1\leq j\leq n},(j^{n}_{j})_{n>0,1\leq j\leq n})$. If 
$\mathcal{C}'=((C_n',s'^{n}_{n-1,j},t'^{n}_{n-1,j})_{1\leq j\leq n},(j'^{n}_{j})_{n>0,1\leq j\leq n})$ is another
$(\infty,0)$-set, then a morphism of $(\infty,0)$-sets 
\[\xymatrix{\mathcal{C}\ar[rr]^{f}&&\mathcal{C}'}\]
is given by a morphism of cubical sets such that for each $n>0$ and for each $1\leq j\leq n$ we have the following commutative diagrams
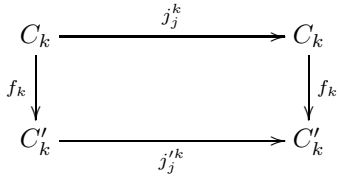
\[\xymatrix{C_{n}\ar[d]_{f_{n}}\ar[rrr]^{j^{n}_{j}}&&&C_{n}\ar[d]^{f_{n}}\\
C'_{n}\ar[rrr]_{j'^{n}_{j}}&&&C'_{n}}\]
The category of cubical $(\infty,0)$-sets is denoted $(\infty,0)$-$\mathbb{C}\E$.

\begin{remark}
A cubical set $\mathcal{C}=(C_n,s^{n}_{n-1,j},t^{n}_{n-1,j})_{1\leq j\leq n}$ can be equipped with other operations
which give the inverse for degeneracies which come from connections. These operations are called \textit{connectors}
in \cite{cam-anatole}. For each integer $n\in\mathbb{N}^*$ such that $C_n$ is non-empty, we define connectors on $\mathcal{C}$
as maps :

$$\begin{tikzcd}
C_{n}\arrow[rr,"j^{n,\gamma}_j"]&&C_{n}
\end{tikzcd}$$

such that we have the following commutative diagrams :

$$\begin{tikzcd} 
C_{n}\arrow[rd,"s^{n}_{n-1,j}"{left}]\arrow[rr,"j^{n,\gamma}_j"]&&C_{n}\arrow[ld,"s^{n}_{n-1,j+1}"]\\
&C_{n-1}
\end{tikzcd}\qquad
\begin{tikzcd}C_{n}\arrow[rd,"t^{n}_{n-1,j}"{left}]\arrow[rr,"j^{n,\gamma}_j"]&&C_{n}\arrow[ld,"t^{n}_{n-1,j+1}"]\\
&C_{n-1}
\end{tikzcd}$$

where $j\in\{1,\cdots,n-1\}$. This provide on $\mathcal{C}$ with another kind of structure of $(\infty,0)$-set, which could
be used to define cubical inverses related to connections. But we prefer to avoid such structure, thought very 
interesting, because our scope is first to define cubical weak $\infty$-groupoids with connections, which use (through adapted 
stretchings, see \ref{isostretchings}) the characterization in \cite{lucas} 
of cubical strict $\infty$-groupoids with connections, just by using reversors as above. Such cubical set $\mathcal{C}$
can be also equipped with reversors and connectors, which is still another kind of structure of $(\infty,0)$-set which
must also deserve more investigations in the Future. It is also important to notice that the different kind of inverses
for strict cubical $\infty$-categories has been invested by Lucas in \cite{lucas} but without our level of structures.
\end{remark}

\begin{remark}
This structural approach of inverses is much more powerful that the simplicial methods because with it we are able to build
any kind of \textit{reversible higher structure}. For example in our framework it is a simple exercise to build some exotic 
one which could be difficult to be build with simplicial method.
\end{remark}

\subsection{Cubical strict $\infty$-groupoids}
\label{cub-m-strict}
In this section we use the characterization of Lucas \cite{lucas} for cubical strict $\infty$-groupoids with connections. 
Thus cubical strict $\infty$-groupoids are just cubical strict $\infty$-categories with connections such that all $n$-cells for $n>0$ are 
$\circ^{n}_{j}$-isomorphisms for $1\leq j\leq n$. The studies of cubical strict structures using inverses has
been done in \cite{bro2,bro1,bro-hig2,bro-hig1,bro-spen,dakin,lucas,salleh}, especially for the scope to generalize
many known results involving low dimensional groupoids and algebraic topology. For example generalization 
of cubical strict fundamental groupoid to higher dimensions has been undertaken 
in \cite{bro-hig2,bro-hig3} in order to obtain higher version of Van Kampen type Theorem. 

In this article we use cubical strict $\infty$-groupoids as an underlying part of the structure of the \textit{cubical $(\infty,0)$-groupoidal
stretchings} (see \ref{isostretchings}) which are the adapted stretchings to weakened cubical strict $\infty$-groupoids. Thus
they are an important step for our approach of cubical weak $\infty$-groupoids.

Consider a cubical strict $\infty$-category $\mathcal{C}$ as defined in \ref{strict-infty}. We say that it is a cubical strict 
$\infty$-groupoid if its underlying cubical set is equipped with an $(\infty,0)$-structure $(j^{n}_{j})_{n>0,1\leq j\leq n}$ such that we require 
the following identities : $\forall j,n$ such that $1\leq j\leq n$ and $0<n$,
  \[\forall\alpha\in C_n, \hspace{.1cm}\alpha\circ^{n}_{j}j^{n}_{j}(\alpha)=1^{n-1}_{n,j}(t^{n}_{n-1,j}(\alpha))
  \hspace{.1cm}\text{and}\hspace{.1cm}j^{n}_{j}(\alpha)\circ^{n}_{j}\alpha=1^{n-1}_{n,j}(s^{n}_{n-1,j}(\alpha))\]
\begin{proposition}
A cubical strict $\infty$-groupoids $\mathcal{C}$ as above has a unique underlying
$(\infty,0)$-set.
\end{proposition}

A cubical strict $\infty$-functor preserve $(k,j)$-reversors. Thus morphisms between cubical strict $\infty$-groupoids are just
cubical strict $\infty$-functors. The category of cubical strict $\infty$-groupoids is denoted 
$\Igp$. 
As in \ref{monad-strict} it is not difficult to show the following proposition
\begin{proposition}
The evident forgetful functor 
\[\xymatrix{\Igp\ar[rr]^{U}&&\mathbb{C}\E}\]
is right adjoint and monadic.
\end{proposition}
The monad of cubical strict $\infty$-groupoids on cubical sets is denoted $\mathbb{S}^0=(S^0,\lambda^0,\mu^0)$. Here 
$\lambda^0$ is 
the unit map of $\mathbb{S}^0$ : $\xymatrix{1_{\mathbb{C}\E}\ar[rr]^{\lambda^0}&&S^0}$ and $\mu^0$ is the multiplication of 
$\mathbb{S}^0$ : $\xymatrix{(S^0)^2\ar[rr]^{\mu^{0}}&&S^0}$

\subsection{The category of cubical weak 
$\infty$-groupoids}
\label{isostretchings} 

A cubical reflexive $(\infty,0)$-magma is an object of $\I\mathbb{C}\mathbb{M}\text{ag}_\text{r}$ such that its underlying
cubical set is equipped with an $(\infty,0)$-structure. Morphisms between cubical reflexive $(\infty,0)$-magmas are
those of $\I\mathbb{C}\mathbb{M}\text{ag}_\text{r}$ which are also morphisms of $(\infty,0)\text{-}\mathbb{C}\mathbb{S}ets$, i.e
they preserve the underlying $(\infty,0)$-structures. The category of cubical reflexive $(\infty,0)$-magmas is denoted 
$(\infty,0)\text{-}\mathbb{C}\mathbb{M}\text{ag}_\text{r}$.

The category $\Etg$ of cubical $\infty$-groupoidal stretchings has as objects 
 quintuples 
  \[ \mathbb{E}=(M,C,\pi,([-;-]^{n}_{n+1,j})_{n\in\mathbb{N};j\in\{1,...,n\}},([-;-]^{n,\gamma}_{n+1,j})_{n\in\mathbb{N};j\in\{1,...,n\};\gamma\in\{-,+\}})\] 
  where $M$ is a cubical reflexive $(\infty,0)$-magma, $C$ is a cubical strict $\infty$-groupoid, $\pi$ is
a morphism in $(\infty,0)\text{-}\mathbb{C}\mathbb{M}\text{ag}_\text{r}$
\[\xymatrix{M\ar[rr]^{\pi}&&C}\] 
and $([-;-]^{n}_{n+1,j})_{n\in\mathbb{N};j\in\{1,...,n\}}, ([-;-]^{n,\gamma}_{n+1,j})_{n\in\mathbb{N};j\in\{1,...,n\};\gamma\in\{-,+\}}$ are the cubical bracketing structures which have
already been defined in \ref{stretchings}. A morphism of cubical $\infty$-groupoidal stretchings
\[\xymatrix{\mathbb{E}\ar[rr]^{(m,c)}&&\mathbb{E}'}\]
is given by the following commutative square in $(\infty,0)\text{-}\mathbb{C}\mathbb{M}\text{ag}_\text{r}$,
\[\xymatrix{M\ar[d]_{\pi}\ar[rr]^{m}&&M'\ar[d]^{\pi'}\\
  C\ar[rr]_{c}&&C' }\]
such that for all $n\in \mathbb{N}$, and for all
$(\alpha,\beta)\in\widetilde{M}_{n}$,
\[m_{n+1}([\alpha,\beta]^{n}_{n+1,j})=[m_{n}(\alpha),m_{n}(\beta)]^{n}_{n+1,j} 
\hspace{.2cm}(j\in\{1,...,n+1\})\]

and

\[m_{n+1}([\alpha,\beta]^{n,\gamma}_{n+1,j})=[m_{n}(\alpha),
m_{n}(\beta)]^{n,\gamma}_{n+1,j}\hspace{.2cm}(j\in\{1,...,n\}, \gamma\in\{-,+\})\]
The category of cubical $\infty$-groupoidal stretchings is denoted $\Etg$.

Now consider the forgetful functor:
\[\xymatrix{\Etg\ar[r]^(.6){U_{}}&\Cu\E}\]
given by 

$$\xymatrix{(M,C,\pi,([-;-]^{n}_{n+1,j})_{n\in\mathbb{N};j\in\{1,...,n\}},([-;-]^{n,\gamma}_{n+1,j})_{n\in\mathbb{N};j\in\{1,...,n\};\gamma\in\{-,+\}})\ar@{|->}[r]^{}&M }$$

This
functor has a left adjoint which produces a monad
$\mathbb{W}^0=(W^0,\eta^0,\nu^0)$ on the category of cubical sets.

\begin{definition}
Cubical weak $\infty$-groupoids are algebras for 
the monad $\mathbb{W}^0=(W^0,\eta^0,\nu^0)$ defined above on the category $\mathbb{C}\E$ of cubical sets.
\end{definition}

Thus the category of our models of cubical weak $\infty$-groupoids is 
denoted $\mathbb{W}^0\text{-}\mathbb{A}\text{lg}$.

If it is evident to see that we have "an embedded" of $\Igp$ in $\Cc$, this is
also the case for the weak case, even it is a bit more subtle : it comes from the 
forgetful functor  
$$\begin{tikzcd}
\Etg\arrow[rrr,"U"]&&&\Et
\end{tikzcd}$$
which forgets the underlying $(\infty,0)$-structures. Also we have the following 
morphism of the category $\mathbb{A}\text{dj}$ of adjunctions :
$$\begin{tikzcd}
        \Etg\arrow[rrr,"V"]
             \arrow[ddd,xshift=1ex,"\dashv"{left},"U^0"]&&&\Et
             \arrow[ddd,xshift=1ex,"\dashv"{left},"U"]\\\\\\
  \CS\arrow[rrr,"id"]
  \arrow[uuu,xshift=-1ex,"F"] 
        &&&\CS
  \arrow[uuu,xshift=-1ex,"F"] 
 \end{tikzcd}$$
 because
$U\circ V=U^0$ (see \cite{kamel}), thus it produces a morphism $V^*$ in the category $\mathbb{M}\text{nd}$ of monads
 $$\begin{tikzcd}
 (\Cu\E,\mathbb{W})\arrow[rrr,"V^*"]&&&(\Cu\E,\mathbb{W}^0) 
 \end{tikzcd}$$
 and passing to algebras, gives the following functor $\mathbb{A}\text{lg}(V^*)$ which is the "embedded"
we were looking for

$$\begin{tikzcd}
\mathbb{W}^0\text{-}\mathbb{A}\text{lg}\arrow[rrr,"\mathbb{A}\text{lg}(V^*)"]&&&\mathbb{W}\text{-}\mathbb{A}\text{lg}
\end{tikzcd}$$

In \cite{brown} it is proved that the category of cubical strict $\infty$-categories with connections is equivalent to the category of 
globular strict $\infty$-categories. In \cite{lucas} it is proved that the category of cubical strict $(\infty,m)$-categories
with connections is equivalent to the category of globular strict $(\infty,m)$-categories (which were first defined in
\cite{Cam}) and in particular he proved that the category of cubical strict $\infty$-groupoids
with connections is equivalent to the category of globular strict $\infty$-groupoids. 

Because our models of cubical weak $\infty$-groupoids with connections are the direct analogue of globular
weak $\infty$-groupoids (models of it are built in \cite{Cam,batanin-main,Maltsin-Gr}), and also because of the conjecture of 
Batanin in \cite{bat-pen} which says that his globular weak $\infty$-categories are "equivalent" to Penon 
globular weak $\infty$-categories, and plus the recent result of Bourke \cite{bourke}, which resolve the conjecture
of Ara \cite{ara} which says that Batanin's globular weak $\infty$-categories are equivalent to 
Grothendieck-Maltsiniotis weak $\infty$-categories, all these facts together impose to put the following 
hypothesese :

\begin{conj}[1]
The category $\mathbb{W}^0\text{-}\mathbb{A}\text{lg}$ of cubical weak $\infty$-groupoids with connections, 
is equipped with a Quillen model structure which is Quillen
equivalent to the category of globular weak $\infty$-groupoids equipped with a Quillen model 
structure.
\end{conj}

\begin{conj}[2]
The category $\mathbb{W}^0\text{-}\mathbb{A}\text{lg}$ of cubical weak $\infty$-groupoids with connections, 
is equiped with a Quillen model structure, the one of the first conjecture, such that its localization is equivalent to the category of $CW$-spaces. 
\end{conj}

This second conjecture is inspired by the result in \cite{malts} which says that the category 
$\mathbb{C}_r$ defined in \ref{reflexive_cubical_sets} is a \textit{test category}, that is the 
category of presheaves on $\mathbb{C}_r$ is equipped with a Quillen model 
structure (the \textit{Cisinski one}, see \cite{malts}), such that its localization is equivalent
to the category of $CW$-spaces. 
 
These hypotheses together must solve the \textit{Grothendieck conjecture on homotopy types} :

\begin{conj}[Grothendieck]
The category of globular weak $\infty$-groupoids is equipped with a Quillen model structure such that 
its localization is equivalent to the category of $CW$-spaces. 
\end{conj}

\section{Steps toward the cubical weak $\infty$-category of cubical weak $\infty$-categories.}

In this last section we are going to show how cubical stretchings give
algebraic models of cubical weak $\infty$-functors and cubical weak $\infty$-natural 
transformations. We can go further as it was done for the globular paradigm in \cite{kamel},
but just with this level $2$ of cubical algebras we obtain some 
interesting canonical $2$-cubes in \ref{c1}, \ref{c2} and \ref{c3}. In the end of this section \ref{c4}
we draw the cocubical object in some category of monads that we hope
to describe with more precision in the future and which contains all cubical analogue of globular algebras of 
\cite{kamel}. This cocubical object should be 
an important step to obtain the expected cubical weak $\infty$-category of cubical weak 
$\infty$-categories and thus the expected cubical Grothendieck $\infty$-topos, cubical 
$\infty$-stacks, etc. More precision for such cubical higher structure is in \cite{camark-cub-2}.

\subsection{The category of cubical weak $\infty$-functors}
\label{cub-funct}

In Theoretical Physics basic data of an \textit{Extended Topological Quantum Field Theory} (\textit{ETQFT}) 
\cite{jacob-etqft,feshbach-voronov} is given by a weak monoidal $\infty$-functor between a higher monoidal category of 
cobordisms and kind of higher categorification of the monoidal category of Hilbert spaces. For example in \cite{jacob-etqft} 
the \textit{simplicial geometry} is used to define ETQFT, while in \cite{feshbach-voronov} they used instead the \textit{multiple geometry} 
of Charles Ehresmann \cite{Ehresmann}. In this section we are going to define algebraic model of cubical weak
$\infty$-functors with the hope that in the future it can be useful for accurate algebraic models of cubical ETQFT.
 
In \cite{penon-amiens} Jacques Penon had proposed algebraic models
of globular weak $\infty$-functors which were extended to all kind of globular weak higher transformations in \cite{kamel}. The methods used
in \cite{kamel,penon-amiens} has consisted to use different kind of \textit{stretchings} in order to weakened different kind of strict structure. For 
example in \cite{penon-amiens} he has built a category of stretchings named in \cite{kamel} the category of $(0,\infty)$-\textit{categorical stretchings}
\footnote{$(0,\infty)$-categorical stretchings must not be confused with $(\infty,0)$-categorical stretchings used in \ref{isostretchings} to 
weakened cubical strict $(\infty,0)$-categories and which were used in \cite{Cam} to weakened globular strict $(\infty,0)$-categories.} 
 which were adapted to weakened strict $\infty$-functors. And in \cite{kamel} the author used the category of $(n,\infty)$-\textit{categorical stretchings}
 to weakened all kind of globular strict $n$-transformations for all $n\geq 2$ (strict natural $\infty$-transformations correspond to $n=2$ and strict $\infty$-modifications
 correspond to $n=3$, etc.). As we are going to see, our models of cubical weak $\infty$-functors are build with similar technology : we are going
 to define \textit{cubical functorial stretchings} which contains all "informations"\footnote{here we put these brackets, because these structural informations must be thought up to some weak equivalences, following the idea of \textit{models} developed in any higher category theory} of the structure behind cubical weak $\infty$-functors. This
 structure produces a monad on the category $\Cu\E\times\Cu\E$ which algebras are our models of cubical weak $\infty$-functors. In 
 \ref{cub-trans} we shall investigate similar constructions for models of cubical weak $\infty$-natural transformations. 
 
 Cubical strict $\infty$-functors have been defined in \ref{strict-cat}. A morphism
 between two cubical strict $\infty$-functors
 \begin{tikzcd}
 C\arrow[rr,"F"]&&D
 \end{tikzcd}
 and 
 \begin{tikzcd}
 C'\arrow[rr,"F'"]&&D'
 \end{tikzcd} 
 is given by a commutative $2$-cube in $\Cu\mathbb{C}\text{AT}$
 
 $$\begin{tikzcd}
C \arrow[ddd,"F"{left}]\arrow[rrr,"c"] &&& C'\arrow[ddd,"F'"] \\\\\\
D \arrow[rrr,"d"{below}] &&& D'
\end{tikzcd}$$

The category of cubical strict $\infty$-functors is denoted $\I\Cu\mathbb{F}\text{unct}$.

\subsubsection{The category of cubical $(0,\infty)$-magmas}
A cubical $(0,\infty)$-magma is given by a morphism $\xymatrix{M_0\ar[rr]^{F_M}&&M_1}$ of $\Cu\E$ such that $M_0$ and $M_1$ are 
objects of $\I\Cu\Mag_\text{r}$. This object is denoted $(M_0,F_M,M_1)$. 

A morphism between $(0,\infty)$-magmas
\[\begin{tikzcd}
(M_0,F_M,M_1) \arrow[rrr,"m"]&&&(M'_0,F'_M,M'_1)
\end{tikzcd}\]

is given by two morphisms of $\I\Cu\Mag_\text{r}$ :
\begin{tikzcd}
M_0 \arrow[rr,"m_0"]&&M'_0
\end{tikzcd},
\begin{tikzcd}
M_1 \arrow[rr,"m_1"]&&M'_1
\end{tikzcd},
such that the following diagram commutes in $\Cu\E$

$$\begin{tikzcd}
M_0 \arrow[ddd,"m_0"{left}]\arrow[rrr,"F_M"] &&& M_1\arrow[ddd,"m_1"] \\\\\\
M'_0 \arrow[rrr,"F'_M"{below}] &&& M'_1
\end{tikzcd}$$

The category of cubical $(0,\infty)$-magmas is denoted by $(0,\infty)\text{-}\Cu\Mag_\text{r}$

\subsubsection{The category of cubical $(0,\infty)$-categorical stretchings}
A $(0,\infty)$-stretching is given by a triple
 \[\mathbb{E}=(\mathbb{E}_0,\mathbb{E}_1,F_M,F_C)\]
 such that $\mathbb{E}_0$, $\mathbb{E}_1$ are cubical categorical stretchings given by
 \[ \mathbb{E}_0=(M_0,C_0,\pi_0,(^{0}[-;-]^{n}_{n+1,j})_{n\in\mathbb{N};j\in\{1,...,n\}},(^{0}[-;-]^{n,\gamma}_{n+1,j})_{n\in\mathbb{N};j\in\{1,...,n\};\gamma\in\{-,+\}})\] 
 and 
 \[ \mathbb{E}_1=(M_1,C_1,\pi_1,(^{1}[-;-]^{n}_{n+1,j})_{n\in\mathbb{N};j\in\{1,...,n\}},(^{1}[-;-]^{n,\gamma}_{n+1,j})_{n\in\mathbb{N};j\in\{1,...,n\};\gamma\in\{-,+\}})\] 
 also : $(M_0,F_M,M_1)$ is an object of $(0,\infty)\text{-}\Cu\Mag_\text{r}$, 
 \begin{tikzcd}
C_0\arrow[rr,"F_C"]&&C_1
\end{tikzcd} 
is a strict cubical $\infty$-functor, such that the following square is commutative in $\Cu\E$ :

$$\begin{tikzcd}
M_0 \arrow[ddd,"\pi_0"{left}]\arrow[rrr,"F_M"] &&& M_1\arrow[ddd,"\pi_1"] \\\\\\
C_0 \arrow[rrr,"F_C"{below}] &&& C_1
\end{tikzcd}$$

A morphism of $(0,\infty)$-stretchings
$$\begin{tikzcd}
\mathbb{E}=(\mathbb{E}_0,\mathbb{E}_1,F_M,F_C)\arrow[rrr,""]&&&\mathbb{E}'=(\mathbb{E}'_0,\mathbb{E}'_1,F'_M,F'_C)
\end{tikzcd}$$
is given by the following commutative diagram in $\Cu\E$ : 

$$\begin{tikzcd}[row sep=scriptsize, column sep=scriptsize]
&& M_0' \arrow[rrr,"F'_M"] \arrow[ddd,"\pi'_0"] &&& M_1' \arrow[ddd,"\pi'_1"] \\ M_0\arrow[rru,"m_0"] \arrow[rrr, "F_M",crossing over] \arrow[ddd,"\pi_0"] & && M_1 \arrow[rru,"m_1"]\\\\
&& C_0' \arrow[rrr,"F'_C"] && & C_1' \\
C_0 \arrow[rru,"c_0"] \arrow[rrr,"F_C"] & && C_1 \arrow[rru,"c_1"] \arrow[from=uuu, "\pi_1", pos=0.3, crossing over]\\
\end{tikzcd}$$

such that $(m_0,m_1)$ is a morphism of $(0,\infty)\text{-}\Cu\Mag_\text{r}$, $(m_0,c_0)$ and $(m_1,c_1)$ are 
morphisms of $\I\Cu\mathbb{E}\text{tC}$. The category of $(0,\infty)$-stretchings is denoted $(0,\infty)\text{-}\Cu\mathbb{E}\text{tC}$.

Now consider the forgetful functor:
$$\xymatrix{(0,\infty)\text{-}\Cu\mathbb{E}\text{tC}\ar[rr]^{U_{}}&&\Cu\E\times\Cu\E}$$
given by 

$$\xymatrix{\mathbb{E}=(\mathbb{E}_0,\mathbb{E}_1,F_M,F_C)\ar@{|->}[r]^(.6){U}&(M_0,M_1) }$$

It is not difficult to show that the category $(0,\infty)\text{-}\Cu\mathbb{E}\text{tC}$ is projectively sketchable
and that its sketch contains the projective sketch of $\Cu\E\times\Cu\E$. 
Thus this
functor has a left adjoint which produces a monad
$\mathbb{T}^0=(T^0,\lambda^0,\mu^0)$ on the category $\Cu\E\times\Cu\E$.
\begin{definition}
  Cubical weak $\infty$-functors are algebras for the monad $\mathbb{T}^0$ above.
 \end{definition}
 
 Thus a cubical weak $\infty$-functor is given by a quadruple $(C_0,C_1,v_0,v_1)$ such that if we note 
 $T^0(C_0,C_1)=(T^0_0(C_0,C_1),T^0_1(C_0,C_1)$ then we get its underling morphisms of $\Cu\E$
 $$\begin{tikzcd}
 T^0_0(C_0,C_1)\arrow[rr,"v_0"]&&C_0 
 \end{tikzcd}\qquad
 \begin{tikzcd}
 T^0_1(C_0,C_1)\arrow[rr,"v_1"]&&C_1 
 \end{tikzcd}$$  
 
 and these morphisms of $\Cu\E$ put on $(C_0,C_1)$ a structure of cubical weak $\infty$-functor 
  \begin{tikzcd}
 C_0\arrow[rrr,"F"]&&&C_1 
 \end{tikzcd}  
 
 defined by $$F=v_1\circ F_M\circ\lambda^{0}_{0}(C_0,C_1)$$
 with 
 
 \begin{tikzcd}
 C_0\arrow[ddd,"\lambda^{0}_{0}{(C_0,C_1)}"{left}]\\\\\\
 T^0_0(C_0,C_1)\arrow[ddd,"v_0"{left}]\arrow[rrr,"F",dotted]&&&
 T^0_1(C_0,C_1)\arrow[ddd,"v_1"]\\\\\\
 C_0\arrow[rrr,"v_1{\circ F_M\circ}\lambda^0_0"{below},dotted]&&&C_1
 \end{tikzcd}

\subsection{The category of cubical weak $\infty$-natural transformations}
\label{cub-trans}
Now we describe a monad on the category 
$$(\Cu\E)^4=\Cu\E\times\Cu\E\times\Cu\E\times\Cu\E$$ 
which algebras are our models of cubical weak $\infty$-natural transformations. In \cite{kamel} we defined globular $\infty$-natural transformations by using the structure given by an adapted category of globular stretchings, namely the category of $(1,\infty)$-stretchings. 
Here we use similar technology by defining first the category of \textit{cubical} $(1,\infty)$-\textit{stretchings} which contains the underlying structure needed to weakened cubical strict natural $\infty$-transformations. This last monad $\mathbb{T}^1$ (see below) gives some $2$-cubes in \ref{c3} and a first flavor of an expected cocubical object of monads described in the end of this article. 

\subsubsection{The category of cubical strict $\infty$-natural transformations}
Cubical strict natural transformations were introduced in \cite{grandis-pare}. Here we give the evident strict and higher version of it.
A cubical strict $\infty$-natural transformation is given by a $2$-cube in $\CC$

$$\begin{tikzcd}
C_{0,0} \arrow[rrrddd,phantom,"\tau\Downarrow"] \arrow[ddd,"H"{left}]\arrow[rrr,"F"] &&& C_{1,0}\arrow[ddd,"G"] \\\\\\
C_{0,1} \arrow[rrr,"K"{below}] &&&C_{1,1}
\end{tikzcd}$$

which $0$-cells corresponds to four cubical strict $\infty$-categories $C_{0,0}$, $C_{0,1}$, $C_{1,0}$, $C_{1,1}$, which $1$-cells corresponds to four 
cubical strict $\infty$-functors $F$, $G$, $H$, $K$, and which the only $2$-cell $\tau$ corresponds, for all $0$-cells $a$ in $C_{0,0}$, to a $1$-cell
\begin{tikzcd}
G(F(a))\arrow[rr,"\tau(a)"]&&K(H(a))
\end{tikzcd}
such that for all $1$-cells 
\begin{tikzcd}
a\arrow[rr,"f"]&&b
\end{tikzcd}
of $C_{0,0}$ we have the following commutative diagram :

$$\begin{tikzcd}
G(F(a)) \arrow[dd,"G(F(f))"{left}]\arrow[rr,"\tau(a)"] &&K(H(a))\arrow[dd,"K(H(f))"] \\\\
G(F(b)) \arrow[rr,"\tau(b)"{below}] &&K(H(b))
\end{tikzcd}$$

A morphism between two cubical strict $\infty$-natural transformations $\tau$ and 
$\tau'$ is given by a $3$-cube in $\CC$

$$\begin{tikzcd}[row sep=scriptsize, column sep=scriptsize]
&& C'_{0,0} \arrow[rrrddd,phantom,"\tau'\Downarrow",pos=0.6] \arrow[rrr,"F'"] \arrow[ddd,"H'"] &&& C'_{1,0} \arrow[ddd,"G'"] \\ C_{0,0}\arrow[rrrddd,phantom,"\tau\Downarrow",pos=0.4] \arrow[rru,"c_{0,0}"] \arrow[rrr, "F",crossing over] \arrow[ddd,"H"{left}] & && C_{1,0} \arrow[rru,"c_{1,0}"]\\\\
&& C'_{0,1}\arrow[rrr,"K'"{below}] && & C'_{1,1} \\
C_{0,1}\arrow[rru,"c_{0,1}"] \arrow[rrr,"K"{below}] & && C_{1,1} \arrow[rru,"c_{1,1}"{below}] \arrow[from=uuu, "G", pos=0.4, crossing over]\\
\end{tikzcd}$$

such that $c_{1,0} F=F' c_{0,0}$, $c_{1,1} G=G' c_{1,0}$, 
$c_{0,1} H=H' c_{0,0}$ and $c_{1,1} K=K' c_{0,1}$. The category of 
cubical strict $\infty$-natural transformations is denoted $(1,\infty)\text{-}\Cu\Trans$, and
we obtain an internal $2$-cube in $\mathbb{C}\text{AT}$

\label{c1}
$$\begin{tikzcd}
   (1,\infty)\text{-}\Cu\Trans \arrow[rrrr, yshift=4ex,"\sigma^2_{1,1}"]   
   \arrow[rrrr, yshift=1ex,"\sigma^2_{1,2}"]
   \arrow[rrrr, yshift=-1ex,"\tau^2_{1,1}"{below}]
    \arrow[rrrr, yshift=-4ex,"\tau^2_{1,2}"{below}]
      &&&& \I\Cu\mathbb{F}\text{unct} 
      \arrow[rrrr, yshift=1ex,"\sigma^1_0"]
      \arrow[rrrr, yshift=-1ex,"\tau^1_0"{below}]   
        &&&&\Cu\mathbb{C}\text{AT}
\end{tikzcd}$$

\begin{proposition}
\label{deux-cubes}
The internal $2$-cube of $\mathbb{C}\text{AT}$ just above can be structured 
in a strict cubical $2$-category
\end{proposition}

\begin{proof}
\begin{description}

\item Consider the following object $\tau\in (1,\infty)\text{-}\Cu\Trans$ :

$$\begin{tikzcd}
C_{0,0} \arrow[rrrddd,phantom,"\tau\Downarrow"] \arrow[ddd,"H"{left}]\arrow[rrr,"F"] &&& C_{1,0}\arrow[ddd,"G"] \\\\\\
C_{0,1} \arrow[rrr,"K"{below}] &&&C_{1,1}
\end{tikzcd}$$
such that $\sigma^2_{1,1}(\tau)=F$, $\sigma^2_{1,2}(\tau)=H$, $\tau^2_{1,1}(\tau)=K$ and
$\tau^2_{1,2}(\tau)=G$, and such that $\sigma^1_{0}$ and $\tau^1_{0}$ are 
clearly defined.

\item \underline{Definition of the classical reflexivity} :

$$\begin{tikzcd}
   (1,\infty)\text{-}\Cu\Trans
      &&&& \I\Cu\mathbb{F}\text{unct} 
         \arrow[llll, yshift=-1ex,"1^{1}_{2,2}"]
      \arrow[llll, yshift=1ex,"1^{1}_{2,1}"{above}]   
        &&&&\arrow[llll,"1^{0}_{1}"{above}]        
        \Cu\mathbb{C}\text{AT}
\end{tikzcd}$$

$1^{1}_{2,1}(F)$ is given by 
$$\begin{tikzcd}
C_{0,0} \arrow[rrrddd,phantom,"1^{1}_{2,1}(F)\Downarrow"] 
\arrow[ddd,"1_{C_{0,0}}"{left}]\arrow[rrr,"F"] &&& C_{1,0}\arrow[ddd,"1_{C_{1,0}}"] \\\\\\
C_{0,0} \arrow[rrr,"F"{below}] &&&C_{1,0}
\end{tikzcd}$$
and is such that $1^{1}_{2,1}(F)(a)=1^{0}_{1}(F(a))$ for all $0$-cells 
$a\in C_{0,0}(0)$, 

and also $1^{1}_{2,2}(F)$ is given by 

$$\begin{tikzcd}
C_{0,0} \arrow[rrrddd,phantom,"1^{1}_{2,2}(F)\Downarrow"] 
\arrow[ddd,"F"{left}]\arrow[rrr,"1_{C_{0,0}}"] &&& C_{0,0}\arrow[ddd,"F"] \\\\\\
C_{1,0} \arrow[rrr,"1_{C_{1,0}}"{below}] &&&C_{1,0}
\end{tikzcd}$$
and is such that $1^{1}_{2,2}(F)(a)=1^{0}_{1}(F(a))$ for all $0$-cells 
$a\in C_{0,0}(0)$.

\item \underline{Definition of the connections} :

$$\begin{tikzcd}
   (1,\infty)\text{-}\Cu\Trans
      &&&& \I\Cu\mathbb{F}\text{unct} 
         \arrow[llll, yshift=-1ex,"1^{1,-}_{2,2}"]
      \arrow[llll, yshift=1ex,"1^{1,+}_{2,1}"{above}]   
\end{tikzcd}$$

$1^{1,-}_{2,1}(F)$ is given by 
$$\begin{tikzcd}
C_{0,0} \arrow[rrrddd,phantom,"1^{1,-}_{2,1}(F)\Downarrow"] 
\arrow[ddd,"F"{left}]\arrow[rrr,"F"] &&& C_{1,0}\arrow[ddd,"1_{C_{1,0}}"] \\\\\\
C_{1,0} \arrow[rrr,"1_{C_{1,0}}"{below}] &&&C_{1,0}
\end{tikzcd}$$
and is such that $1^{1,-}_{2,1}(F)(a)=1^{0}_{1}(F(a))$ for all $0$-cells 
$a\in C_{0,0}(0)$, 

and $1^{1,+}_{2,1}(F)$ is given by 
$$\begin{tikzcd}
C_{0,0} \arrow[rrrddd,phantom,"1^{1,+}_{2,1}(F)\Downarrow"] 
\arrow[ddd,"1_{C_{0,0}}"{left}]\arrow[rrr,"1_{C_{0,0}}"] &&& C_{0,0}\arrow[ddd,"F"] \\\\\\
C_{0,0}\arrow[rrr,"F"{below}] &&&C_{1,0}
\end{tikzcd}$$
and is such that $1^{1,+}_{2,1}(F)(a)=1^{0}_{1}(F(a))$ for all $0$-cells 
$a\in C_{0,0}(0)$.

\item The following shape of $2$-cells

$$\begin{tikzcd}
C_{0,0} \arrow[rrrddd,phantom,"\tau\Downarrow"] \arrow[ddd,"H"{left}]\arrow[rrr,"F"] &&& C_{1,0}\arrow[ddd,"G"] \\\\\\
C_{0,1} \arrow[rrrddd,phantom,"\rho\Downarrow"] \arrow[ddd,"H'"{left}]\arrow[rrr,"K"{below}] &&&C_{1,1}\arrow[ddd,"G'"]\\\\\\
C_{0,2} \arrow[rrr,"K'"{below}] &&&C_{1,2}
\end{tikzcd}$$
allows to define the composition $\rho\circ^{2}_{1,1}\tau$

$$\begin{tikzcd}
C_{0,0} \arrow[rrrddd,phantom,"\rho\circ^{2}_{1,1}\tau\Downarrow"] \arrow[ddd,"H'\circ H"{left}]\arrow[rrr,"F"] &&& C_{1,0}\arrow[ddd,"G'\circ G"] \\\\\\
C_{0,2} \arrow[rrr,"K"{below}] &&&C_{1,2}
\end{tikzcd}$$
by the formula : 
$$(\rho\circ^{2}_{1,1}\tau)(a)=\rho(H(a))\circ G'(\tau(a))$$

and the following shape of $2$-cells
$$\begin{tikzcd}
C_{0,0} \arrow[rrrddd,phantom,"\tau\Downarrow"] \arrow[ddd,"H"{left}]
\arrow[rrr,"F"] &&& C_{1,0}\arrow[ddd,"G"]\arrow[rrrddd,phantom,"\rho\Downarrow"]
\arrow[rrr,"F'"]&&&C_{2,0}\arrow[ddd,"G'"]\\\\\\
C_{0,1} \arrow[rrr,"K"{below}] &&&C_{1,1}\arrow[rrr,"K'"{below}] &&&C_{2,1}
\end{tikzcd}$$
allows to define the composition $\rho\circ^{2}_{1,2}\tau$

$$\begin{tikzcd}
C_{0,0} \arrow[rrrddd,phantom,"\rho\circ^{2}_{1,2}\tau\Downarrow"] \arrow[ddd,"H"{left}]\arrow[rrr,"F'\circ F"] &&& C_{2,0}\arrow[ddd,"G"] \\\\\\
C_{0,1} \arrow[rrr,"K'\circ K"{below}] &&&C_{2,1}
\end{tikzcd}$$
by the formula :
$$(\rho\circ^{2}_{1,2}\tau)(a)=K'(\tau(a))\circ\rho(F(a))$$

The proof that these datas put a structure of cubical strict $2$-categories on the internal $2$-cube
of the proposition is left to the reader.
\end{description}
\end{proof}

\subsubsection{The category of cubical $(1,\infty)$-magmas}
A cubical $(1,\infty)$-magma is an object with shape

$$\begin{tikzcd}
M_{0,0} \arrow[rrrddd,phantom,"\tau_M\Downarrow"] \arrow[ddd,"H_M"{left}]\arrow[rrr,"F_M"] &&& M_{1,0}\arrow[ddd,"G_M"] \\\\\\
M_{0,1} \arrow[rrr,"K_M"{below}] &&&M_{1,1}
\end{tikzcd}$$

such that $(M_{0,0},F_M,M_{1,0})$, $(M_{1,0},G_M,M_{1,1})$, $(M_{0,0},H_M,M_{0,1})$
and $(M_{0,1},K_M,M_{1,1})$ are objects of $(0,\infty)\text{-}\Cu\Mag_\text{r}$, and such that
$\tau_M$ is a map 
\begin{tikzcd}
M_{0,0}(0)\arrow[rrr,"\tau_M"]&&&M_{1,1}(1)
\end{tikzcd}
which sends each $0$-cells $a$ of $M_{0,0}$ to an $1$-cell $\tau_M(a)\in M_{1,1}(1)$ such that
$s^{1}_{0}(\tau_M(a))=G_M(F_M(a))$ and $t^{1}_{0}(\tau_M(a))=K_M(H_M(a))$. We want
to avoid heavy notations and shall denote usually just by $\tau_M$ such object of a category
$(1,\infty)\text{-}\Cu\Mag_\text{r}$, where we have to think this greek letter $\tau$ as the variable
 usually used for natural transformations and the subscript $M$ in it just means "Magmatic".

Given $\tau_M$ and $\tau'_M$ two objects of $(1,\infty)\text{-}\Cu\Mag_\text{r}$, a morphism
between them is given by a commutative diagram in $\I\Cu\E$

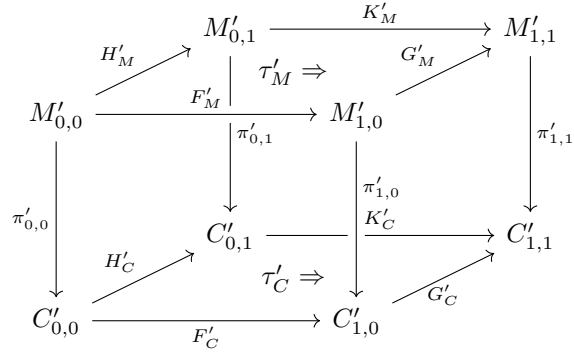
$$\begin{tikzcd}[row sep=scriptsize, column sep=scriptsize]
&& M'_{0,0} \arrow[rrrddd,phantom,"\tau'_M\Downarrow",pos=0.6] \arrow[rrr,"F'_M"] \arrow[ddd,"H'_M"] &&& M'_{1,0} \arrow[ddd,"G'_M"] \\ M_{0,0}\arrow[rrrddd,phantom,"\tau_M\Downarrow",pos=0.4] \arrow[rru,"m_{0,0}"] \arrow[rrr, "F_M",crossing over] \arrow[ddd,"H_M"{left}] & && M_{1,0} \arrow[rru,"m_{1,0}"]\\\\
&& M'_{0,1}\arrow[rrr,"K'_M"{below}] && & M'_{1,1} \\
M_{0,1}\arrow[rru,"m_{0,1}"] \arrow[rrr,"K_M"{below}] & && M_{1,1} \arrow[rru,"m_{1,1}"{below}] \arrow[from=uuu, "G_M", pos=0.4, crossing over]\\
\end{tikzcd}$$

such that $(m_{0,0},m_{1,0})$, $(m_{1,0},m_{1,1})$, $(m_{0,0},m_{0,1})$, $(m_{0,1},m_{1,1})$
are morphisms of $(0,\infty)\text{-}\Cu\Mag_\text{r}$. It is important to note that commutativity 
of this diagram means also the equality $m_{1,1}\circ\tau_M=\tau'_M\circ m_{0,0}$.

We obtain an internal $2$-cube in $\mathbb{C}\text{AT}$
\label{c2}
$$\begin{tikzcd}
   (1,\infty)\text{-}\Cu\Mag_\text{r} \arrow[rrrr, yshift=4ex,"\sigma^2_{1,1}"]   
   \arrow[rrrr, yshift=1ex,"\sigma^2_{1,2}"]
   \arrow[rrrr, yshift=-1ex,"\tau^2_{1,1}"{below}]
    \arrow[rrrr, yshift=-4ex,"\tau^2_{1,2}"{below}]
      &&&& (0,\infty)\text{-}\Cu\Mag_\text{r} 
      \arrow[rrrr, yshift=1ex,"\sigma^1_0"]
      \arrow[rrrr, yshift=-1ex,"\tau^1_0"{below}]   
        &&&&\I\Cu\Mag_\text{r}
 \end{tikzcd}$$

\begin{proposition}
The internal $2$-cube of $\mathbb{C}\text{AT}$ just above can be structured 
in a cubical reflexive $2$-magma
\end{proposition}
\begin{proof}
The proof is easy and basic datas have been already defined in \ref{deux-cubes}.
\end{proof}

\subsubsection{The category of cubical $(1,\infty)$-categorical stretchings}

A cubical $(1,\infty)$-categorical stretching is given by a commutative diagram in $\I\Cu\E$ :

$$\begin{tikzcd}[row sep=scriptsize, column sep=scriptsize]
&& M_{0,1}\arrow[rrr,"K_M"] \arrow[ddd,"\pi_{0,1}"] &&& M_{1,1}\arrow[ddd,"\pi_{1,1}"] \\ M_{0,0}\arrow[rrrrru,phantom,"\tau_M\Rightarrow"]
\arrow[rru,"H_M"] \arrow[rrr, "F_M",crossing over] \arrow[ddd,"\pi_{0,0}"{left}] & && M_{1,0} \arrow[rru,"G_M"]\\\\
&& C_{0,1} \arrow[rrr,"K_C"] && & C_{1,1} \\
C_{0,0} \arrow[rrrrru,phantom,"\tau_C\Rightarrow"]\arrow[rru,"H_C"] \arrow[rrr,"F_C"{below}] & && C_{1,0} \arrow[rru,"G_C"{below}] \arrow[from=uuu, "\pi_{1,0}", pos=0.3, crossing over]\\
\end{tikzcd}$$

such that $(\pi_{0,0},F_M,F_C,\pi_{1,0})$, $(\pi_{1,0},G_M,G_C,\pi_{1,1})$, 
$(\pi_{0,0},H_M,H_C,\pi_{0,1})$ and $(\pi_{0,1},K_M,K_C,\pi_{1,1})$ are 
objects of $(0,\infty)\text{-}\Cu\mathbb{E}\text{tC}$, and also $\tau_M$ is an object
of $(1,\infty)\text{-}\Cu\Mag_\text{r}$ and $\tau_C$ is an object of
$(1,\infty)\text{-}\Cu\Trans$. It is important to note that commutativity 
of this diagram means also that the equality $\pi_{1,1}\circ\tau_M=\tau_C\circ\pi_{0,0}$ holds.
Such cubical $(1,\infty)$-categorical stretching can be denoted $(\tau_M,\tau_C)$. Given
an other cubical $(1,\infty)$-categorical stretching $(\tau'_M,\tau'_C)$ :

$$\begin{tikzcd}[row sep=scriptsize, column sep=scriptsize]
&& M'_{0,1}\arrow[rrr,"K'_M"] \arrow[ddd,"\pi'_{0,1}"] &&& M'_{1,1}\arrow[ddd,"\pi'_{1,1}"] \\ M'_{0,0}\arrow[rrrrru,phantom,"\tau'_M\Rightarrow"]
\arrow[rru,"H'_M"] \arrow[rrr, "F'_M",crossing over] \arrow[ddd,"\pi'_{0,0}"{left}] & && M'_{1,0} \arrow[rru,"G'_M"]\\\\
&& C'_{0,1} \arrow[rrr,"K'_C"] && & C'_{1,1} \\
C'_{0,0} \arrow[rrrrru,phantom,"\tau'_C\Rightarrow"]\arrow[rru,"H'_C"] \arrow[rrr,"F'_C"{below}] & && C'_{1,0} \arrow[rru,"G'_C"{below}] \arrow[from=uuu, "\pi'_{1,0}", pos=0.3, crossing over]\\
\end{tikzcd}$$

a morphism $\xymatrix{(\tau_M,\tau_C)\ar[r]^{}&(\tau'_M,\tau'_C)}$ of such cubical $(1,\infty)$-categorical stretchings is given by

\begin{itemize}
\item a morphism of $(1,\infty)\text{-}\Cu\Mag_\text{r}$ underlied by $(m_{0,0},m_{1,0},m_{0,1},m_{1,1})$, a
morphism of $(1,\infty)\text{-}\Cu\Trans$ underlied by $(c_{0,0},c_{1,0},c_{0,1},c_{1,1})$ :

$$
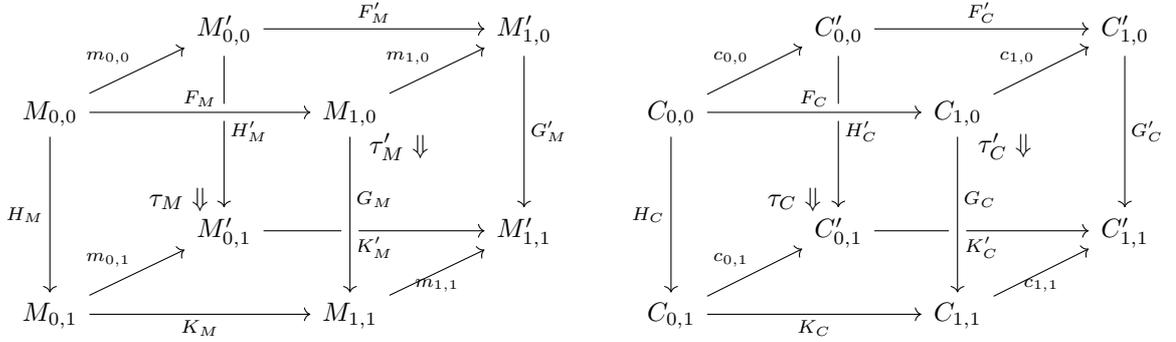
\begin{tikzcd}[row sep=scriptsize, column sep=scriptsize]
&& M'_{0,0} \arrow[rrrddd,phantom,"\tau'_M\Downarrow",pos=0.6] \arrow[rrr,"F'_M"] \arrow[ddd,"H'_M"] &&& M'_{1,0} \arrow[ddd,"G'_M"] \\ M_{0,0}\arrow[rrrddd,phantom,"\tau_M\Downarrow",pos=0.4] \arrow[rru,"m_{0,0}"] \arrow[rrr, "F_M",crossing over] \arrow[ddd,"H_M"{left}] & && M_{1,0} \arrow[rru,"m_{1,0}"]\\\\
&& M'_{0,1}\arrow[rrr,"K'_M"{below}] && & M'_{1,1} \\
M_{0,1}\arrow[rru,"m_{0,1}"] \arrow[rrr,"K_M"{below}] & && M_{1,1} \arrow[rru,"m_{1,1}"{below}] \arrow[from=uuu, "G_M", pos=0.4, crossing over]\\
\end{tikzcd}\qquad \begin{tikzcd}[row sep=scriptsize, column sep=scriptsize]
&& C'_{0,0} \arrow[rrrddd,phantom,"\tau'_C\Downarrow",pos=0.6] \arrow[rrr,"F'_C"] \arrow[ddd,"H'_C"] &&& C'_{1,0} \arrow[ddd,"G'_C"] \\ C_{0,0}\arrow[rrrddd,phantom,"\tau_C\Downarrow",pos=0.4] \arrow[rru,"c_{0,0}"] \arrow[rrr, "F_C",crossing over] \arrow[ddd,"H_C"{left}] & && C_{1,0} \arrow[rru,"c_{1,0}"]\\\\
&& C'_{0,1}\arrow[rrr,"K'_C"{below}] && & C'_{1,1} \\
C_{0,1}\arrow[rru,"c_{0,1}"] \arrow[rrr,"K_C"{below}] & && C_{1,1} \arrow[rru,"c_{1,1}"{below}] \arrow[from=uuu, "G_C", pos=0.4, crossing over]\\
\end{tikzcd}$$

\item the following morphisms : $((m_{0,0},c_{0,0}),(m_{1,0},c_{1,0}))$, 
$((m_{1,0},c_{1,0}),(m_{1,1},c_{1,1}))$, $((m_{0,0},c_{0,0}),(m_{0,1},c_{0,1}))$ 
and $((m_{0,1},c_{0,1}),(m_{1,1},c_{1,1}))$, of
$(0,\infty)\text{-}\Cu\mathbb{E}\text{tC}$

$$\begin{tikzcd}[row sep=scriptsize, column sep=scriptsize]
&& M'_{0,0}\arrow[rrr,"F'_M"] \arrow[ddd,"\pi'_{0,0}"] &&& M'_{1,0}\arrow[ddd,"\pi'_{1,0}"] \\ M_{0,0}
\arrow[rru,"m_{0,0}"] \arrow[rrr, "F_M",crossing over] \arrow[ddd,"\pi_{0,0}"{left}] & && M_{1,0} \arrow[rru,"m_{1,0}"]\\\\
&& C'_{0,0} \arrow[rrr,"F'_C"] && & C'_{1,0} \\
C_{0,0} \arrow[rru,"c_{0,0}"]\arrow[rrr,"F_C"{below}] & && C_{1,0} \arrow[rru,"c_{1,0}"{below}] \arrow[from=uuu, "\pi_{1,0}", pos=0.3, crossing over]\\
\end{tikzcd}\qquad  \begin{tikzcd}[row sep=scriptsize, column sep=scriptsize]
&& M'_{0,0}\arrow[rrr,"H'_M"] \arrow[ddd,"\pi'_{0,0}"] &&& M'_{0,1}\arrow[ddd,"\pi'_{0,1}"] \\ M_{0,0}
\arrow[rru,"m_{0,0}"] \arrow[rrr, "H_M",crossing over] \arrow[ddd,"\pi_{0,0}"{left}] & && M_{0,1} \arrow[rru,"m_{0,1}"]\\\\
&& C'_{0,0} \arrow[rrr,"H'_C"] && & C'_{0,1} \\
C_{0,0} \arrow[rru,"c_{0,0}"]\arrow[rrr,"H_C"{below}] & && C_{0,1} \arrow[rru,"c_{0,1}"{below}] \arrow[from=uuu, "\pi_{0,1}", pos=0.3, crossing over]\\
\end{tikzcd}$$

 $$\begin{tikzcd}[row sep=scriptsize, column sep=scriptsize]
&& M'_{1,0}\arrow[rrr,"G'_M"] \arrow[ddd,"\pi'_{1,0}"] &&& M'_{1,1}\arrow[ddd,"\pi'_{1,1}"] \\ M_{1,0}
\arrow[rru,"m_{1,0}"] \arrow[rrr, "G_M",crossing over] \arrow[ddd,"\pi_{1,0}"{left}] & && M_{1,1} \arrow[rru,"m_{1,1}"]\\\\
&& C'_{1,0} \arrow[rrr,"G'_C"] && & C'_{1,1} \\
C_{1,0} \arrow[rru,"c_{1,0}"]\arrow[rrr,"G_C"{below}] & && C_{1,1} \arrow[rru,"c_{1,1}"{below}] \arrow[from=uuu, "\pi_{1,1}", pos=0.3, crossing over]\\
\end{tikzcd}\qquad \begin{tikzcd}[row sep=scriptsize, column sep=scriptsize]
&& M'_{0,1}\arrow[rrr,"K'_M"] \arrow[ddd,"\pi'_{0,1}"] &&& M'_{1,1}\arrow[ddd,"\pi'_{1,1}"] \\ M_{0,1}
\arrow[rru,"m_{0,1}"] \arrow[rrr, "K_M",crossing over] \arrow[ddd,"\pi_{0,1}"{left}] & && M_{1,1} \arrow[rru,"m_{1,1}"]\\\\
&& C'_{0,1} \arrow[rrr,"K'_C"] && & C'_{1,1} \\
C_{0,1} \arrow[rru,"c_{0,1}"]\arrow[rrr,"K_C"{below}] & && C_{1,1} \arrow[rru,"c_{1,1}"{below}] \arrow[from=uuu, "\pi_{1,1}", pos=0.3, crossing over]\\
\end{tikzcd}$$
\end{itemize}

We denote $(1,\infty)\text{-}\Cu\mathbb{E}\text{tC}$ the category of cubical $(1,\infty)$-categorical stretchings. Now we have a forgetful functor :

$$\xymatrix{(1,\infty)\text{-}\Cu\mathbb{E}\text{tC}\ar[rr]^{U_{}}&&(\Cu\E)^4}$$
which sends the object $(\tau_M,\tau_C)$ to the object $(M_{0,0},M_{1,0},M_{0,1},M_{1,1})$.

This
functor has a left adjoint which produces a monad
$\mathbb{T}^1=(T^1,\lambda^1,\mu^1)$ on the category $(\Cu\E)^4$.
\begin{definition}
  Cubical weak natural $\infty$-transformations are algebras for the monad $\mathbb{T}^1$ above.
\end{definition}

Thus we obtain a $2$-cube in the category $\mathbb{A}\text{dj}$ of pairs
of adjunctions defined in \cite{kamel}

\label{c3}
$$
\begin{tikzcd}
(1,\infty)\text{-}\Cu\mathbb{E}\text{tC}\arrow[rrrr, yshift=4ex,"\sigma^2_{1,1}"]
   \arrow[dddddd,xshift=1ex,"\dashv"{left},"U"]   
   \arrow[rrrr, yshift=1ex,"\sigma^2_{1,2}"]
   \arrow[rrrr, yshift=-1ex,"\tau^2_{1,1}"{below}]
    \arrow[rrrr, yshift=-4ex,"\tau^2_{1,2}"{below}]
      &&&&(0,\infty)\text{-}\Cu\mathbb{E}\text{tC}
    \arrow[dddddd,xshift=1ex,"\dashv"{left},"U"] 
      \arrow[rrrr, yshift=1ex,"\sigma^1_0"]
      \arrow[rrrr, yshift=-1ex,"\tau^1_0"{below}]   
        &&&&\I\mathbb{C}\mathbb{E}\text{tC}
             \arrow[dddddd,xshift=1ex,"\dashv"{left},"U"]
     \\\\\\\\\\\\
 (\CS)^4\arrow[rrrr, yshift=4ex,"\sigma^2_{1,1}"]
 \arrow[uuuuuu,xshift=-1ex,"F"] 
    \arrow[rrrr, yshift=1ex,"\sigma^2_{1,2}"]
   \arrow[rrrr, yshift=-1ex,"\tau^2_{1,1}"{below}]
    \arrow[rrrr, yshift=-4ex,"\tau^2_{1,2}"{below}]
      &&&&(\CS)^2
  \arrow[uuuuuu,xshift=-1ex,"F"]     
      \arrow[rrrr, yshift=1ex,"\sigma^1_0"]
      \arrow[rrrr, yshift=-1ex,"\tau^1_0"{below}]   
        &&&&\CS
  \arrow[uuuuuu,xshift=-1ex,"F"] 
 \end{tikzcd}$$
 
 which allow to obtain a $2$-cocube in the category $\mathbb{M}\text{nd}$ of categories equipped with monads defined in \cite{kamel} 
 $$
\begin{tikzcd}
  ((\CS)^4,\mathbb{T}^1) 
      &&&&((\CS)^2,\mathbb{T}^0) \arrow[llll, yshift=4ex,"\sigma^2_{1,1}"{above}]   
   \arrow[llll, yshift=1ex,"\sigma^2_{1,2}"{above}]
   \arrow[llll, yshift=-1ex,"\tau^2_{1,1}"]
    \arrow[llll, yshift=-4ex,"\tau^2_{1,2}"]      
      &&&&(\CS,\mathbb{W}) 
      \arrow[llll, yshift=1ex,"\sigma^1_0"{above}]
      \arrow[llll, yshift=-1ex,"\tau^1_0"]   
 \end{tikzcd}$$
 
 And finally it gives the following $2$-cube in $\mathbb{C}\text{AT}$
 
 $$
\begin{tikzcd}
  \mathbb{T}^1\text{-}\mathbb{A}\text{lg}  \arrow[rrrr, yshift=4ex,"\sigma^2_{1,1}"]   
   \arrow[rrrr, yshift=1ex,"\sigma^2_{1,2}"]
   \arrow[rrrr, yshift=-1ex,"\tau^2_{1,1}"{below}]
    \arrow[rrrr, yshift=-4ex,"\tau^2_{1,2}"{below}]
      &&&& \mathbb{T}^0\text{-}\mathbb{A}\text{lg}
      \arrow[rrrr, yshift=1ex,"\sigma^1_0"]
      \arrow[rrrr, yshift=-1ex,"\tau^1_0"{below}]   
        &&&&\mathbb{W}\text{-}\mathbb{A}\text{lg}
 \end{tikzcd}$$ 

\begin{proposition}
The internal $2$-cube of $\mathbb{C}\text{AT}$ just above can be structured 
in a cubical weak $2$-category
\end{proposition}

\begin{proof}
Detail of the proof is quite long but is not difficult. For example basic datas of such structure are similar 
to those build in \ref{deux-cubes}.
\end{proof}
\label{c4}
We finish this article by drawing the cocubical shape of monads for all cubical weak higher transformations 
that we hope to describe in a future work. For clarity we change the denotation of the monads
$\mathbb{W}$, $\mathbb{T}^0$ and $\mathbb{T}^1$ described in this article with : 
$\mathbb{W}^0:=\mathbb{W}$, $\mathbb{W}^1:=\mathbb{T}^0$ and 
$\mathbb{W}^2:=\mathbb{T}^1$ :

$$\begin{tikzcd}
 \mathbb{W}^0 \arrow[rr, yshift=1.5ex,"s^{0}_{1}"]
 \arrow[rr, yshift=-1.5ex,"t^{0}_{1}"{below}]
 &&\mathbb{W}^1 
\arrow[rr, yshift=1.5ex,"s^{2}_{1,1}"]
\arrow[rr, yshift=-1.5ex,"t^{2}_{1,1}"{below}]
\arrow[rr, yshift=4.5ex,"s^{2}_{1,2}"] 
\arrow[rr, yshift=-4.5ex,"t^{2}_{1,2}"{below}]   
        && \mathbb{W}^2
         \arrow[rr, yshift=1.5ex,"s^{3}_{2,1}"]
        \arrow[rr, yshift=-1.5ex,"t^{3}_{2,1}"{below}]
        \arrow[rr, yshift=4.5ex,"s^{3}_{2,2}"] 
        \arrow[rr, yshift=-4.5ex,"t^{3}_{2,2}"{below}]
        \arrow[rr, yshift=7.5ex,"s^{3}_{2,3}"] 
        \arrow[rr, yshift=-7.5ex,"t^{3}_{2,3}"{below}]&& 
         \mathbb{W}^{3} 
        \arrow[rr, yshift=1.5ex,"s^{4}_{3,4}"]
        \arrow[rr, yshift=-1.5ex,"t^{4}_{3,4}"{below}]
        \arrow[rr, yshift=4.5ex,"s^{4}_{3,3}"] 
        \arrow[rr, yshift=-4.5ex,"t^{4}_{3,3}"{below}]
        \arrow[rr, yshift=7.5ex,"s^{4}_{3,2}"] 
        \arrow[rr, yshift=-7.5ex,"t^{4}_{3,2}"{below}]        
        \arrow[rr, yshift=10.5ex,"s^{4}_{3,1}"] 
        \arrow[rr, yshift=-10.5ex,"t^{4}_{3,1}"{below}] && 
        \mathbb{W}^{4}\cdots \mathbb{W}^{n-1}  
\arrow[rr, yshift=1.5ex,"s^{n}_{n-1,n-1}"]
        \arrow[rr, yshift=-1.5ex,"t^{n}_{n-1,n-1}"{below}]
        \arrow[rr, yshift=4.5ex,dotted] 
        \arrow[rr, yshift=-4.5ex,dotted]
        \arrow[rr, yshift=7.5ex,"s^{n}_{n-1,i}"] 
        \arrow[rr, yshift=-7.5ex,"t^{n}_{n-1,i}"{below}]        
        \arrow[rr, yshift=10.5ex,dotted] 
        \arrow[rr, yshift=-10.5ex,dotted]    
  \arrow[rr, yshift=13.5ex,"s^{n}_{n-1,1}"] 
        \arrow[rr, yshift=-13.5ex,"t^{n}_{n-1,1}"{below}]&& \mathbb{W}^{n}\cdots    
  \end{tikzcd}$$
For example $\mathbb{W}^3$-algebras are cubical weak $\infty$-modifications\footnote{Its globular analogue has
been described in \cite{kamel}.}. This cocubical object of monads should be a cocubical object of operads \cite{camark-cub-2} 
and we believe that it is in fact a $\mathbb{W}^0$-coalgebra or at least a $\mathbb{B}^0_C$-coalgebra in the sense of
\cite{camark-cub-1} where $\mathbb{B}^0_C$-algebras are the operadic models of cubical weak $\infty$-categories
described in \cite{camark-cub-1}.


\bigbreak{}
\begin{minipage}{1.0\linewidth}
  Camell \textsc{Kachour}\\
  Laboratoire de Mathématique d'Orsay\\
  Equipe d'Arithmétique et de Géométrie Algébrique,
  CNRS,
  France.\\
  Phone: 00 33760407993\\
  Email:\href{mailto:camell.kachour@gmail.com}{\url{camell.kachour@gmail.com}}
\end{minipage}

\end{document}